\documentclass[]{siamart220329}
\usepackage{geometry}

\usepackage{url}
\usepackage{amssymb}
\usepackage{cleveref}

\newsiamremark{remark}{Remark}
\usepackage[utf8]{inputenc}
\usepackage{amsmath} 
\usepackage{amsfonts}
\usepackage{array}
\usepackage[english]{babel}
\usepackage{float}
\newcommand{\enstq}[2]{\left\{#1\mathrel{}\middle|\mathrel{}#2\right\}}
\newcommand{\norm}[1]{\left\|#1\right\|}
\newcommand{\N}{\mathbb{N}}

\newcommand{\R}{\mathbb{R}}
\newcommand{\duality}[2]{\left\langle #1,#2\right\rangle}
\newcommand{\inner}[2]{\left( #1,#2\right)}
\newcommand{\abs}[1]{\left\lvert #1 \right\rvert}

\newcommand{\isdef}{\mathrel{\mathop:}=}

\usepackage{todonotes}
\DeclareMathOperator{\DL}{\textup{DL}}
\DeclareMathOperator{\W}{\textup{W}}
\DeclareMathOperator{\curl}{\mathbf{curl}} 
\let\div\undefined
\DeclareMathOperator{\div}{\textup{div}} 
\makeatletter
\newsavebox{\@brx}
\newcommand{\llangle}[1][]{\savebox{\@brx}{\(\m@th{#1\langle}\)}%
	\mathopen{\copy\@brx\kern-0.5\wd\@brx\usebox{\@brx}}}
\newcommand{\rrangle}[1][]{\savebox{\@brx}{\(\m@th{#1\rangle}\)}%
	\mathclose{\copy\@brx\kern-0.5\wd\@brx\usebox{\@brx}}}
\makeatother
\newcommand{\dduality}[2]{\llangle#1\,, #2\rrangle}
\makeatletter    
\renewcommand*{\vec}[1]{\boldsymbol{#1}}
\makeatother
\usepackage{mathrsfs}
\let\div\undefined
\DeclareMathOperator{\div}{\textup{div}} 
\newcommand{\vertiii}[1]{{\left\vert\kern-0.25ex\left\vert\kern-0.25ex\left\vert #1 
		\right\vert\kern-0.25ex\right\vert\kern-0.25ex\right\vert}}
\crefname{lemma}{Lemma}{Lemmas}
\crefname{corollary}{Corollary}{Lemmas}

\headers{Hypersingular equation on Multiscreens}{M. Averseng, X. CLaeys and R. Hiptmair}
\title{Boundary Element Methods for the Laplace Hypersingular Integral Equation on Multiscreens: a two-level Substructuring Preconditioner}
\author{Martin Averseng\thanks{Laboratoire Angevin de Recherche Mathématique, UMR 6093 CNRS, 2 Boulevard de Lavoisier, Angers, France.} \and Xavier Claeys\thanks{Laboratoire Jacques-Louis Lions, Université Pierre et Marie Curie, 4 place Jussieu, 75005 Paris, France. \url{https://claeys.pages.math.cnrs.fr/}} \and Ralf Hiptmair\thanks{Seminar For Applied Mathematics, ETH Zurich, Rämistrasse 101, Zurich, Switzerland. \url{https://math.ethz.ch/sam/the-institute/people/ralf-hiptmair.html}}}

\begin{document}
	\maketitle
	
	\begin{abstract}
		We present a preconditioning method for the linear systems arising from the boundary element discretization of the Laplace hypersingular equation on a $2$-dimensional triangulated surface $\Gamma$ in $\R^3$. We allow $\Gamma$ to belong to a large class of geometries that we call polygonal multiscreens, which can be non-manifold. After introducing a new, simple conforming Galerkin discretization, we analyze a substructuring domain-decomposition  preconditioner based on ideas originally developed for the Finite Element Method. The surface $\Gamma$ is subdivided into non-overlapping regions, and the application of the preconditioner is obtained via the solution of the hypersingular equation on each patch, plus a coarse subspace correction. We prove that the condition number of the preconditioned linear system grows poly-logarithmically with $H/h$, the ratio of the coarse mesh and fine mesh size, and our numerical results indicate that this bound is sharp. This domain-decomposition algorithm therefore guarantees significant speedups for iterative solvers, even when a large number of subdomains is used.
		
	\end{abstract}

	\section{Introduction}
		The problem that we study arises in the numerical computation, via the Boundary Element Method (BEM), of the solution $\mathcal{U}$ to the exterior Neumann boundary value problem
		\begin{equation}
			\label{PDE_intro}
			\left\{\begin{array}{rclll}
				\Delta \mathcal{U} &=& 0 && \textup{in } \R^3 \setminus \Gamma\,,\\
				\mathcal{U} &=& O(\norm{x}^{-1}) && \textup{uniformly for $\norm{x} \to \infty$}\,,\\
				\nabla \mathcal{U} \cdot \vec n &=& \vec g \cdot \vec n	&& \textup{on } \Gamma\,.
			\end{array}\right.
		\end{equation}
		Here, $\vec n$ is a normal vector field on $\Gamma$, $\vec g$ is a continuous vector field in $\R^3$, and $\Gamma$ is a ``polygonal multi-screen", that is a $2$-dimensional surface in $\R^3$ made of various flat panels allowed to intersect at non-manifold junction points and lines (a more precise definition of the allowed geometries is given below). An example of such a geometry is displayed in Figure \ref{fig:junction2} (left). The ideas that we present can likely be adapted to other constant-coefficient elliptic partial differential equations (PDEs). To keep the presentation focused, we restrict our analysis to the model problem \eqref{PDE_intro} for the time being.
		
		Both for the continuous and the discrete analysis, the challenge in solving eq.~\eqref{PDE_intro} lies in the singular nature of the geometry on which the boundary condition is imposed. Such singular geometric models occur regularly in engineering applications, see, e.g., \cite{alad2013capacitance,chen1970transmission,fan2020high,glisson1980simple,lenti2003bem,sladek1993nonsingular,zhao2020iterative,bettini2015boundary}.
		
		The first difficulty for the BEM is that, for general polygonal multi-screens $\Gamma$, a reformulation of eq.~\eqref{PDE_intro} as a boundary integral equation involving a coercive bilinear form acting on densities on $\Gamma$ has been analyzed only recently \cite{claeys2013integral}, and a {\em conforming} and {\em converging} Galerkin discretization of this variational problem has remained elusive. So far, all proposed methods involved a non-definite variational form on the finite-dimensional subspaces, and a ``quotient-space" iterative resolution see \cite{claeys2021quotient,cools2022calder,cools2022preconditioners}.

		Secondly, for such irregular surfaces $\Gamma$, reformulations of the PDE \eqref{PDE_intro} as a second-kind integral equation -- which are often preferred to first-kind alternatives due to their inherent good conditioning -- do not seem to be known, and hence, preconditioning becomes a crucial issue. This has been the main focus of the recent works \cite{cools2022calder,cools2022preconditioners} of Cools and Urz\'{u}a-Torres for acoustic and electromagnetic scattering\footnote{It is worth mentioning that the analysis in those references accommodates for a indefinite framework, whereas the present work relies heavily on the positive-definiteness of the bilinear form.}, using the idea of {\em operator preconditioning} \cite{christiansen2000preconditionneurs,hiptmair2006operator,steinbach1998construction}.
		
		The present contribution addresses both difficulties, with a focus on rigorous numerical analysis. The first part of this work describes a reformulation of the PDE \eqref{PDE_intro} into a \emph{coercive variational problem}, and proposes a conforming and converging Galerkin discretization, also covering key aspects of the implementation. The second part is concerned with preconditioning; here we opt for a domain-decomposition strategy. More precisely, we introduce a preconditioner in the form of a \emph{two-level additive Schwarz subspace decomposition via substructuring}. Although these tools were originally developed for Finite Element Methods (FEM) (this was started in \cite{bramble1986construction} by Bramble Pasciak and Schatz, see \cite[Chap. 5]{toselli2004domain} for a comprehensive presentation), their use in BEM has received some attention in the past 30 years, see e.g. \cite{heuer1996additive,tran1996additive,heuer1998multilevel,heuer1998iterative,heuer2001additive,maischak2009multilevel,leydecker2012additive,marchand2020two} and references therein.
		
	 	We generalize this type of methods to multiscreen geometries. Our approach is original concerning the analysis of the splitting of the discretized space of jumps. Instead of relying on ``almost local" properties of the $H^{1/2}$ norm, we harness stability results that are known for volume splittings in FEM, and transfer them to $\Gamma$ by applying the jump operator $[\cdot]_\Gamma$. We show that stability is preserved by this operation under a set of conditions related to the existence of stable extension operators from the trace space back to the volume, see also \cite[Thm 2.2]{hiptmair2012stable}. By checking that these conditions hold, we obtain an upper bound on the condition number of the preconditioned BEM linear system which is polylogarithmic in the ratio of the coarse and fine mesh size, see \Cref{thm:main}. This bound holds for all polygonal multi-screens, even those excluded from the analysis of \cite{cools2022calder} (such as the one represented in \Cref{fig:junction2}).

		\begin{figure}
			\label{fig:junction2}
			\centering
			\raisebox{-0.5\height}{\includegraphics[width=0.25\textwidth]{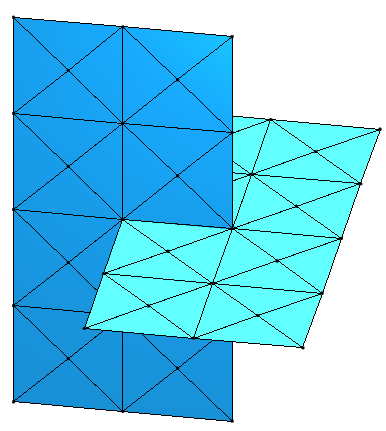}} \quad \raisebox{-0.5\height}{\includegraphics[width=0.35\textwidth]{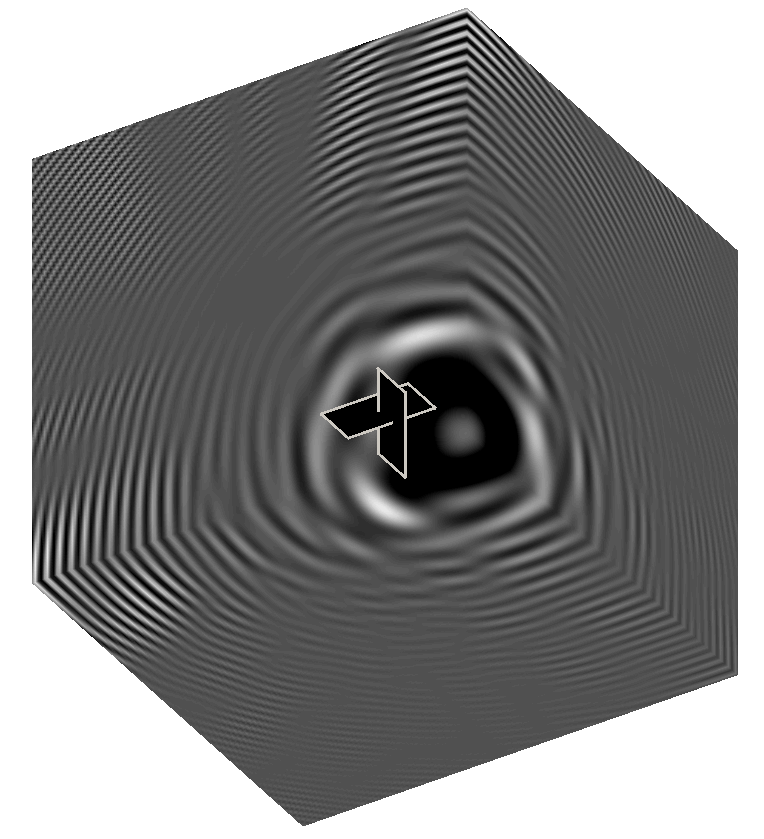}}
			\caption{Example of a 3-dimensional multi-screen (left) and plane-wave scattering by this non-manifold obstacle with incidence along the $(-1,-1,-1)$ direction. This computation has been performed using the Galerkin method described in this paper.}
		\end{figure}
		
		The outline is as follows. We state the main result and illustrate it with numerical experiments in \Cref{sec:main}. In \Cref{sec:HypersingularBIE}, we recast the PDE \eqref{PDE_intro} into a coercive variational problem, and present a conforming Galerkin discretization method in \Cref{sec:GalerkinBEM}. \Cref{sec:quotientSplit} deals with the stability of induced splittings on quotient spaces. We formulate the splitting of the jump space required to define our preconditioner in \Cref{sec:dd} and then prove the condition number estimate. We also collect in the appendix proofs of some useful results previously stated in the multi-screen literature (see \Cref{thm:densityXinfty,thmWweak}).
		
		A full Matlab/C++ prototype of the algorithm described in this paper is freely available and includes the scripts to reproduce the numerical results presented below.\footnote{\url{https://github.com/MartinAverseng/multi-screen-bem3D-ddm}}

	\section{Main result and numerical experiments}
	\label{sec:main}
	
	\subsection{Main result}
	
	We compute the solution $\mathcal{U}$ to \eqref{PDE_intro} as a suitable double-layer potential on $\Gamma$ (see \Cref{def:DL}) where the unknown density $\varphi\in \widetilde{H}^{1/2}([\Gamma])$ is the unique solution of a variational problem of the form
	\begin{equation}
		a(\varphi,\psi) = l_g(\psi) \quad \forall \psi \in \widetilde{H}^{1/2}([\Gamma])\,.
	\end{equation}
	The Hilbert space $\widetilde{H}^{1/2}([\Gamma])$ models Dirichlet jumps across $\Gamma$. Its precise definition is recalled in \Cref{sec:HypersingularBIE}; we will show that the symmetric bilinear form $a$ of \eqref{def:defa} induces an equivalent norm on this space, and that $l_g: \widetilde{H}^{1/2}([\Gamma]) \to \R$ defined by \eqref{defRhs} is a continuous linear form. Introducing a family of nested, shape-regular and quasi-uniform triangulations $(\mathcal{T}_h)_{h>0}$ of $\Gamma$, indexed by an upper bound $h > 0$ on the maximal element diameter, we build an asymptotically dense sequence of subspaces $\widetilde{V}_{h}([\Gamma]) \subset \widetilde{H}^{1/2}([\Gamma])$, which correspond to jumps of continuous piecewise linear functions on $\Gamma$, and define a converging sequence of approximations $\varphi_h$ of $\varphi$ via the Galerkin method
	\begin{equation}
		\label{eq:varfintro}
		a(\varphi_h,\psi_h) = l_g(\psi_h) \quad \forall \psi_h \in \widetilde{V}_{h}([\Gamma])\,.
	\end{equation}
	Given two triangulations, $\mathcal{T}_h,\mathcal{T}_H$, with $h < H$, we define an additive Schwarz preconditioner based on a subspace splitting \cite[Chap. 2]{toselli2004domain}
	\begin{equation}
		\label{schwarz_split_intro}
		\widetilde{V}_h(\Gamma) = \Big(\sum_{\mathcal{F} \textup{ element of } \mathcal{T}_H} \widetilde{V}_{\mathcal{F}} \Big)+ \widetilde{V}_{\mathcal{W}} + \widetilde{V}_H\,.
	\end{equation}
	The definition of the ``face spaces" $\{\widetilde{V}_{\mathcal{F}}\}_{\mathcal{F}}$ and the ``wire-basket" space $\widetilde{V}_{\mathcal{W}}$ is based on a decomposition of the vertex set of $\mathcal{T}_h$ into the vertices lying in the interior of a triangular element $\mathcal{F}$ of $\mathcal{T}_H$, and those lying on edges or vertices of $\mathcal{T}_H$, respectively. In addition, $\widetilde{V}_H := \widetilde{V}_H([\Gamma]) \subset \widetilde{V}_h([\Gamma])$ defines a coarse space for the splitting. The precise definitions of the subspaces are given in \Cref{def:proposedJumpSplit} and a sketch in Figure \ref{fig:sketchSplit} visualizes the elements of the subspaces. Additive Schwarz preconditioning based on this splitting turns the discrete variational problem \eqref{eq:varfintro} into an equation where the operator $P_{\rm ad}: \widetilde{V}_h([\Gamma]) \to \widetilde{V}_h([\Gamma])$ to be evaluated is defined by 
	\[P_{\rm ad}(H;h) \isdef \sum_{\mathcal{F} \textup{ face of } \Gamma} P_{\mathcal{F}} + P_{\mathcal{W}} + P_H\,,\] 
	with $P_X$ the $a(\cdot,\cdot)$ orthogonal projection of $\widetilde{V}_h([\Gamma])$ onto the subspace $\widetilde{V}_X$.  The main result of this paper is the following bound on the spectral condition number $\kappa(P_{\rm ad}(H;h))$ of this operator. 
	\begin{theorem}
		\label{thm:main}
		There exists $C > 0$ such that for all $0 < h < H$, 
		\[\kappa(P_{\rm ad}(H;h)) \leq C \left(1 + \log (H/h)^2\right)\,.\]
	\end{theorem}
Numerical results in \Cref{subsec:numerics} show that this bound is sharp, and in particular that the logarithmic term cannot be removed. The method presented here and the condition number estimate in \Cref{thm:main} are very similar to the ones obtained in \cite[Theorem 1]{heuer2001additive} \cite[Theorem 1]{heuer1998iterative} for planar surfaces in dimension 3.

\begin{figure}[H]
	\centering
	\includegraphics[width=0.3\linewidth]{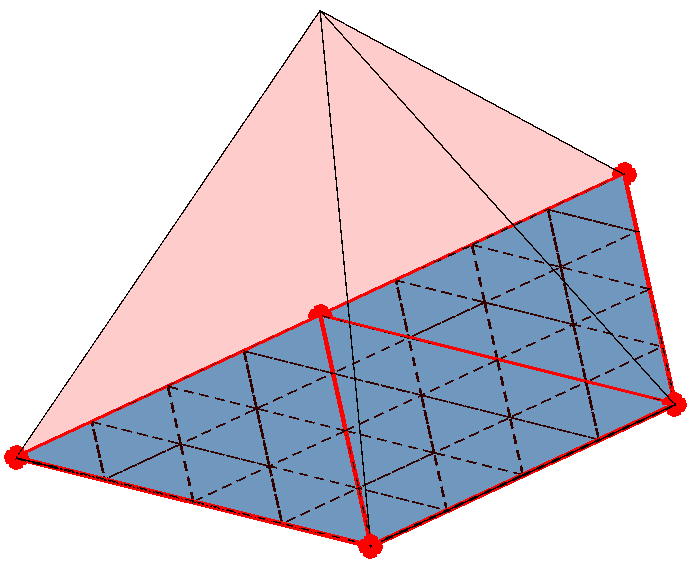} \includegraphics[width=0.3\linewidth]{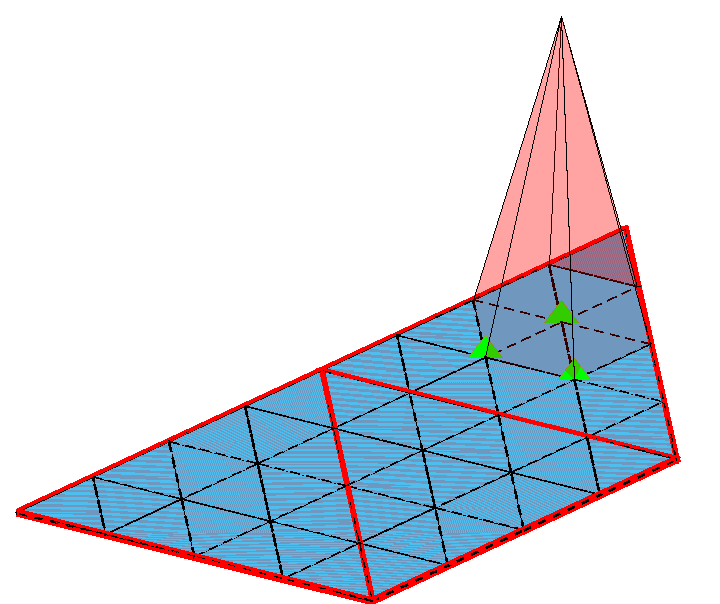}
	\includegraphics[width=0.3\linewidth]{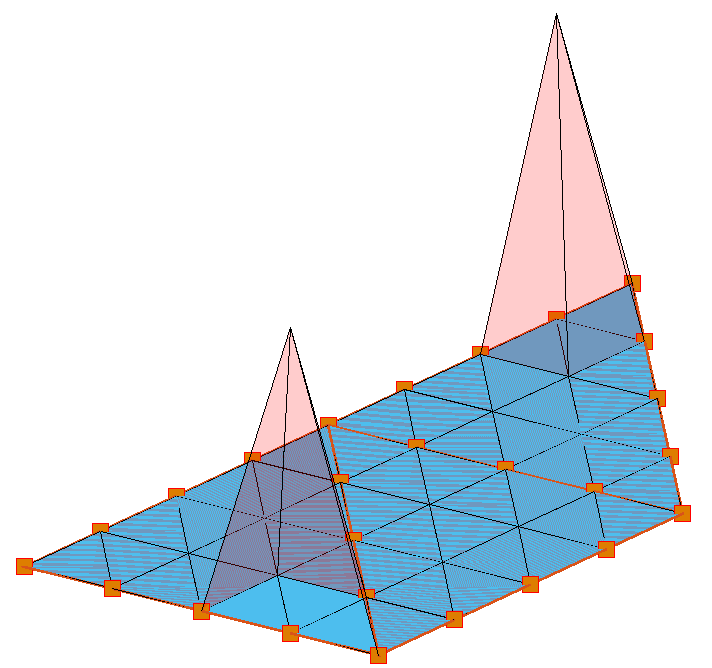}
	\caption{For a simple multiscreen $\Gamma$ composed of $3$ coarse triangles and for piecewise-linear $\widetilde{H}^{1/2}(\Gamma)$-conforming boundary element spaces, sketch of the boundary values of the basis functions belonging to each of the three types of sets in the splitting \eqref{eq:proposedSplit}. Left: a coarse basis function. Middle: a face basis function. Right: two wire-basket basis functions. {In this example, there are three face spaces, one associated to each coarse triangle. The vertices corresponding to the right-most face space are highlighted as green triangles in the middle figure. Similarly, the vertices corresponding to the wire-basket space are highlighted as orange squares in the third figure, and the vertices associated to the coarse space, as red circles in the first figure.}}
	\label{fig:sketchSplit}
\end{figure}
	\begin{remark}[Bound on the number of Preconditioned Conjugate Gradient iterations]
		Let $A_h: \widetilde{V}_h([\Gamma]) \to \widetilde{V}_h([\Gamma])'$ be the operator defined by
		\[\langle A_h u_h,v_h\rangle \isdef a(u_h,v_h)\,, \quad \forall u_h,v_h \in \widetilde{V}_h([\Gamma])\,,\]
		Then one can check that the variational problem \eqref{eq:varfintro} is equivalent to 
		\begin{equation}
			\label{eq:preconditioned}
			P_{\rm ad}(H;h) \varphi = M(H;h) l_g\,.
		\end{equation}
		with $M(H;h):\widetilde{V}_h([\Gamma])' \to \widetilde{V}_h([\Gamma])$ defined by
		$$\widetilde{V}_h([\Gamma])' \ni l \mapsto \sum_{\mathcal{F} \textup{ face of } \Gamma} \phi_{\mathcal{F}} + \phi_{\mathcal{W}} + \phi_H$$
		and where $\phi_X \in \widetilde{V}_X$ is the unique solution of the variational problem
		\[\textup{Find } \phi_X \in \widetilde{V}_X \,\textup{ s.t. } \, a(\phi_X,v_X) = l(v_X) \quad \forall v_X \in \widetilde{V}_X\,.\]
		Note that $P_{\rm ad}(H;h) = M(H;h)A_h$ and that $M(H;h)$ -- hence also $P_{\rm ad}(H;h)$ -- can be evaluated in parallel. The quantity $\kappa = \kappa(P_{\rm ad}(H;h))$ controls the rate of convergence, in the $a(\cdot,\cdot)^{1/2}$ norm, of the preconditioned conjugate gradient method for the resolution of \eqref{eq:preconditioned} in the sense that the error $e_n = \varphi_n - \varphi$ after $n$ iterations satisfies $a(e_n, e_n)^{1/2} \leq 2\rho^{n}a(e_0,e_0)^{1/2}$, where $\rho = \frac{\sqrt{\kappa}-1}{\sqrt{\kappa}+1}$, see e.g. \cite[p.163]{patterson2006iterative}. 
	\end{remark}
	\begin{remark}[Approximate solvers]
		It is possible to extend the theory to accommodate for ``approximate solvers" on the subspaces, which amounts to defining the operators $P_{X}$ in eq. \eqref{schwarz_split_intro} as $\widetilde{a}_X(\cdot,\cdot)$-orthogonal projections onto $\widetilde{V}_X$, for some suitable choice of the local bilinear form $\widetilde{a}_X(\cdot,\cdot)$ on the considered subspace $\widetilde{V}_X$. For instance, using the quasi-uniformity assumption of the mesh, it is possible to prove that the condition number bound of Theorem \ref{thm:main} still holds when replacing the exact bilinear form on the wire-basket by a cheaper, pointwise scalar product, much in the spirit of \cite[Remark 4.3]{bramble1986construction}. Similarly, it is a natural idea to consider approximate solvers on the faces using for instance Calder\'{o}n preconditioning (this is essentially the central idea of \cite{cools2022calder}), or other approximations of the Laplace layer potentials on screens \cite{hiptmair2018closed,averseng2022quasi}, but tracking the dependence with respect to the coarse mesh parameter $H$ seems more delicate in this case; we leave that question to future work.
	\end{remark}
\begin{remark}[Provenance of the logarithmic factor]
	The logarithmic factor comes from the use of a decomposition of $\Gamma$ into panels with {\em no overlap}, and, in the analysis, from discrete trace inequalities for edges in $\R^3$ \cite[Lemma 4.16]{toselli2004domain}. The work \cite{cools2022calder}, in which a similar condition number estimate is proved for a BEM preconditioner on multi-screens, can be thought of as using panels with {\em generous overlap}, a situation which in principle (in view of the corresponding properties for substructuring algorithms in FEM) should lead to the complete removal of the logarithmic factors. However, in that reference, {\em approximate solvers} are used on the face spaces, given by the standard Calder\'{o}n preconditioners. This re-introduces the logarithmic factor, but from a somewhat different source, namely the so-called ``duality mismatch" between the spaces $H^{\pm1/2}(S)$ when $S$ is a smooth manifold with boundary. 
\end{remark}

\begin{remark}[Case where $\Gamma$ is a manifold]
	All the material discussed in this paper also applies to the case where $\Gamma$ is a regular manifold with or without boundary. The Galerkin method then reduces to the standard boundary element method for the hypersingular equation on $\Gamma$. In this regard, our presentation differs from other works on BEM for multiscreens \cite{claeys2021quotient,cools2022calder,cools2022preconditioners}; the difference comes from the fact that we remove the kernel from the hypersingular operator, cf. \Cref{def:defa}. 
\end{remark}
	\subsection{Motivating numerical experiments}
	\label{subsec:numerics}
	\paragraph{Experiment 1. Failure of ``naive" BEM with multiscreens}

	We consider a ``plus-shaped" geometry $\Gamma = [-1,1] \times \{0\} \cup \{0\} \times [-1,1]$ and let 
	\[\mathcal{U}(x_1,x_2) = \textup{Re} \left(\frac{-1}{2iw}\right)\,,\]
	where $z = x_1 + ix_2$ and $w$ is defined by the conformal mapping $z = \frac{1}{2}\left(w + \frac{1}{w}\right)$
	from the region $\abs{z} > 1$ to the region $\mathbb{C} \setminus [-1,1]$ (see \cite[Exercise 8.16]{McLean}). Note that $\mathcal{U}$ is the potential generated by a dipole distribution of density $\varphi(x_1,x_2) = \sqrt{1 - x_1^2}$ on $\Gamma$. One can check that $\mathcal{U}$ is harmonic on $\R^2 \setminus \Gamma$ (it is even harmonic on $\R^2 \setminus ([-1,1] \times \{0\})$) and satisfies an appropriate decay condition at infinity. This explicit solution to the Laplace equation in the complement of $\Gamma$ can thus be used to test a boundary element method. {Taking the cue from the hypersingular boundary integral equation on screens,} a naive approach is to discretize $\Gamma$ using an edge mesh with $4$ coarse elements corresponding to the $4$ arms of the cross, and subdividing each element into a finite number of segments, giving a mesh $\mathcal{M}_{\Gamma,h}$. The surface is not orientable, but in principle, one can attempt to pick an arbitrary choice of a normal vector field $\vec n$ on each element and solve for the surface density $\varphi_{h,\rm naive} \in V_h(\Gamma)$ such that 
	\begin{equation}
		\label{naive}
		\frac{-1}{2\pi}\iint_{\Gamma\times \Gamma} \vec n_{x} \times \nabla_\Gamma \varphi_{h,\rm naive}(x) \cdot \vec n_{y} \times \nabla_\Gamma \psi(y) \ln(\norm{x - y})dx dy = \int_{\Gamma} \vec n_{x} \cdot \nabla \mathcal{U}(x)\psi(x)d x
	\end{equation}
	for all $\psi \in V_h(\Gamma)$. Here, $V_h(\Gamma)$ is the set of \emph{continuous piecewise linear functions} on the mesh $\mathcal{M}_{\Gamma,h}$, with a Dirichlet condition on $\partial \Gamma$, and $\nabla_\Gamma$ is the tangential gradient on $\Gamma$. The corresponding potential $\mathcal{U}_h$ is given by the formula
	\[\mathcal{U}_{h}(x) \isdef \frac{-1}{2\pi} \int_{\Gamma} \frac{\vec n_{y}\cdot (y - x)}{\norm{x - y}^2}\varphi_{h,\rm naive}(y)dy\,.\]
	This would be the standard boundary element methodology, albeit applied to a non-manifold mesh $\mathcal{M}_{\Gamma,h}$. However, as is obvious in \Cref{qualiWrong}, the solution $\mathcal{U}_{h}$ obtained in this way is incorrect. We examine this problem further by computing the discrete $\ell^2$ norm of $\mathcal{U}_{h}^{(i)} - \mathcal{U}$ on a Cartesian grid in a square box surrounding $\Gamma$, for two families $(\mathcal{U}_{h}^{(i)})_{h>0}$, $i = 1,2$, of ``naive approximations", indexed by the average mesh size $h$, where the mesh of $\Gamma$ is uniform ($i = 1$) or quadratically refined near the $4$ vertices of $\partial\Gamma$ ($i = 2$). The results are plotted as the solid and dashed blue curves in Figure \ref{fig:plusShapedCV}, respectively. In both cases, they show a slow decrease of this error as $h \to 0$. We compare those convergence curves (``naive method") to the ones obtained when the approximation of $\mathcal{U}$ is computed via the conforming Galerkin method described in this paper (``new method"). In this case, we observe convergence orders of $O(h)$ for the uniform mesh and $O(h^2)$ for the quadratically refined mesh (solid and dashed red curves, respectively).
	\begin{figure}
		\label{fig:plusShapedCV}
		\centering
		\raisebox{-0.5\height}{\includegraphics[width=.4\textwidth]{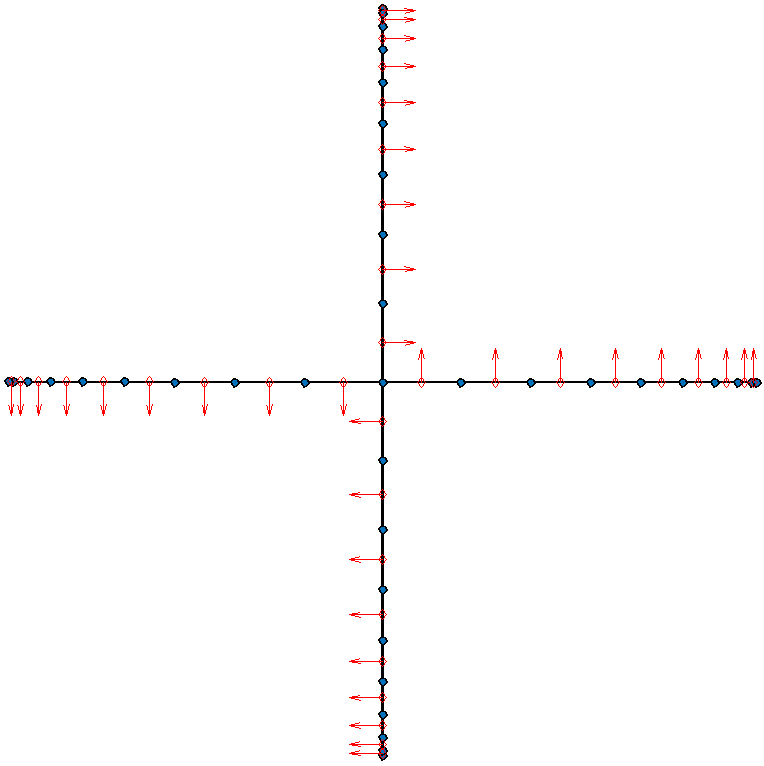}}\qquad
		\raisebox{-0.5\height}{\includegraphics[width=.43\textwidth]{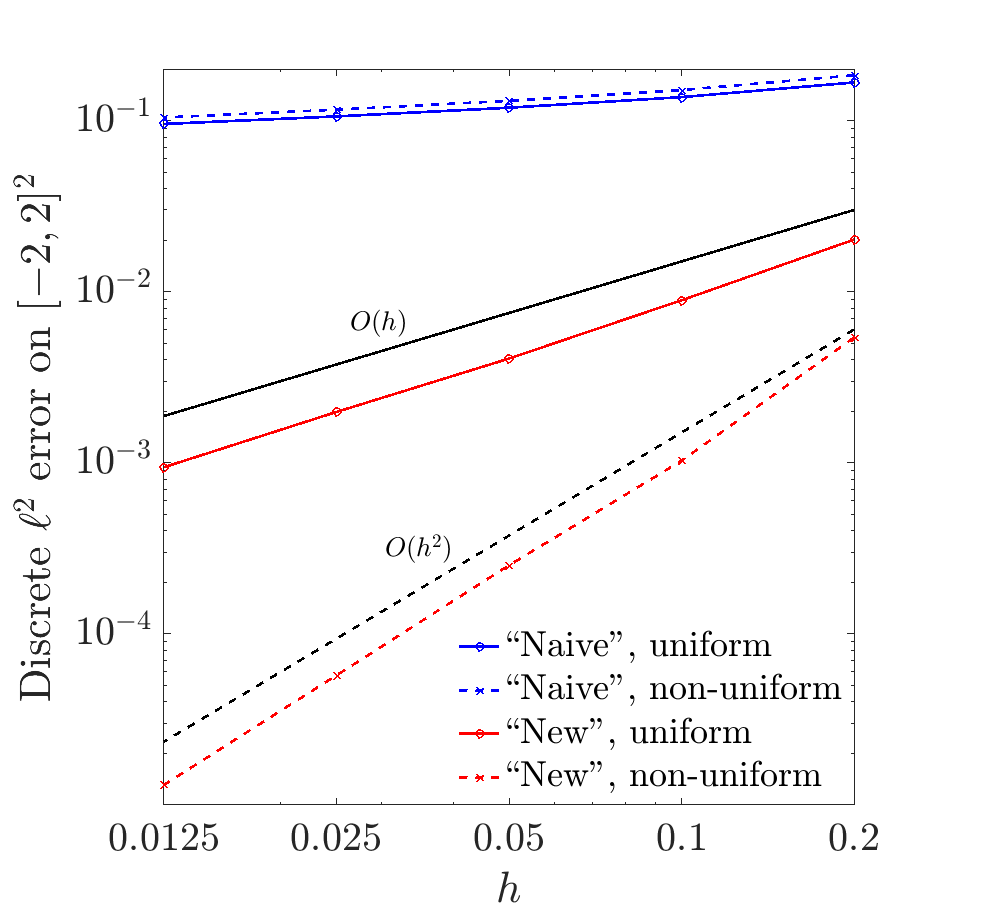}}
		\caption{Left: multi-screen $\Gamma$ with a non-uniform mesh, and choice of normal vector $\vec n$ for the computation of $\mathcal{D}$. Right: discrete $\ell^2$ error for the ``naive method" (blue curves) for the uniform (solid line) and non-uniform mesh (dashed line). Comparison with the ``new method", the Galerkin method presented in this paper (red curves)}
	\end{figure}
	
	\begin{figure}
		\centering
		\includegraphics[width=0.3\textwidth]{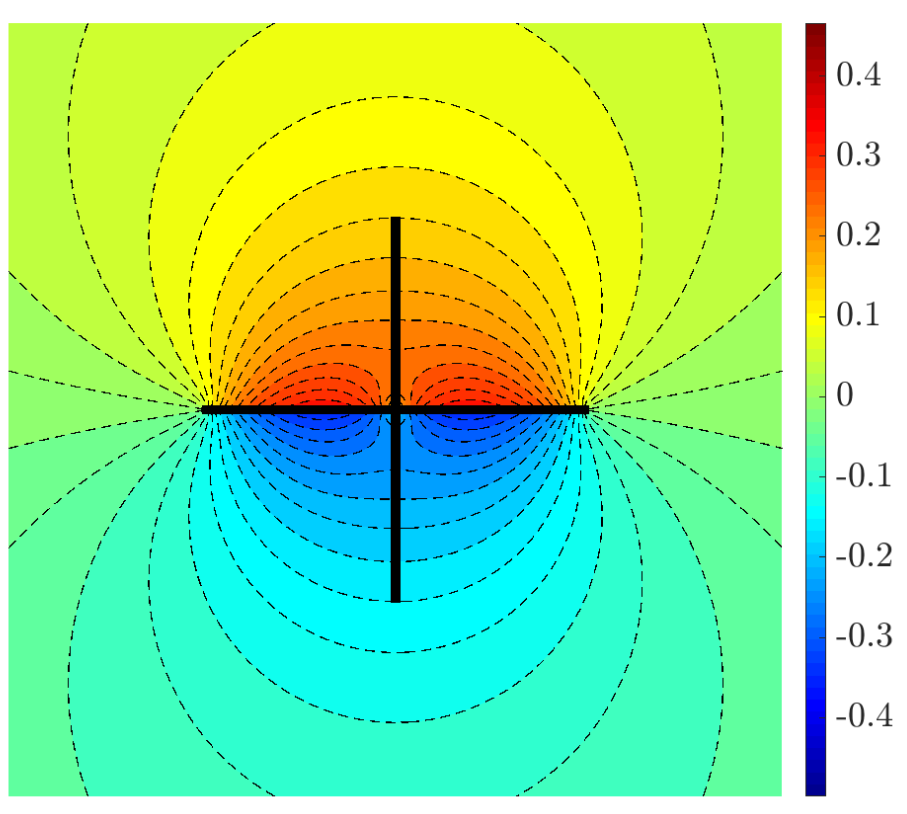}\quad\includegraphics[width=0.3\textwidth]{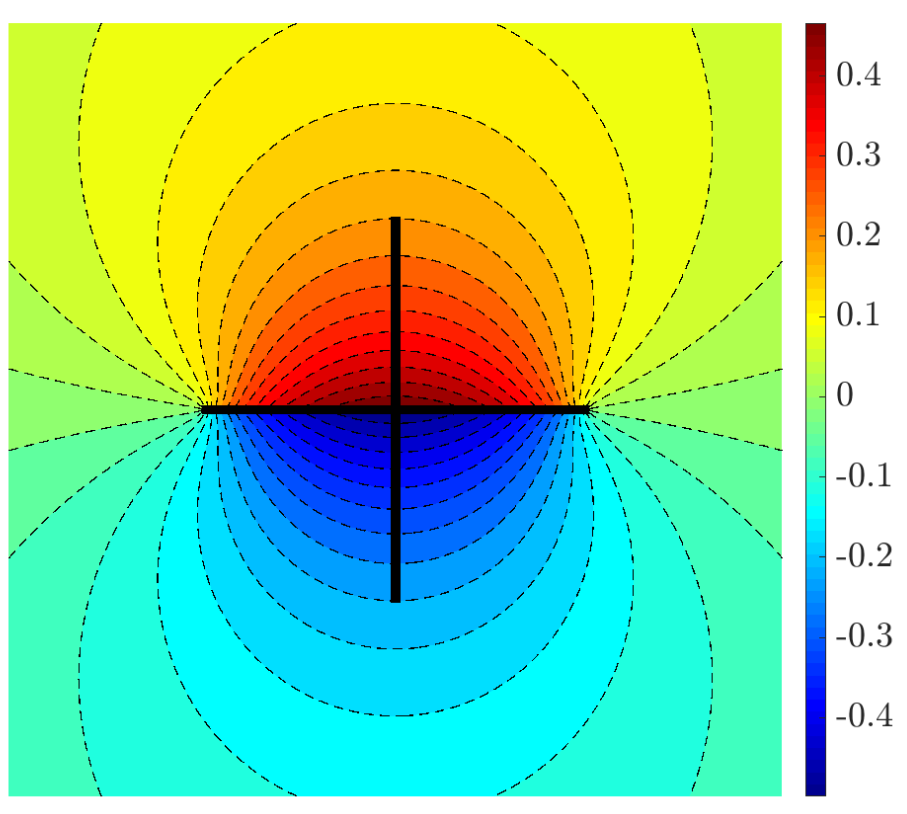}
		\quad\includegraphics[width=0.3\textwidth]{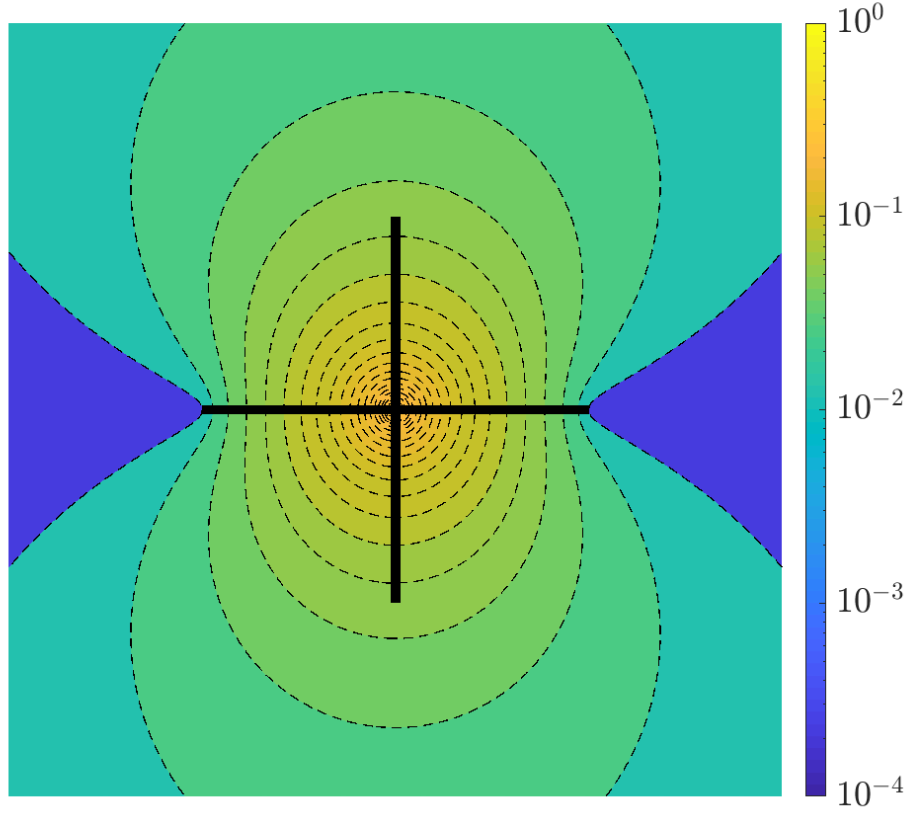}
		\caption{Left: solution $\mathcal{U}_h$ computed via the ``naive" method involving the variational problem \eqref{naive}, with a mesh size of  with a uniform mesh, $h = 0.025$. Middle: exact solution. Right: error in base 10 logarithmic scale. The ``naive" method produces a qualitatively wrong solution, with the error concentrated at the cross-point.}
		\label{qualiWrong}
	\end{figure}
	\begin{figure}
		\centering 
		\includegraphics[width=0.3\textwidth]{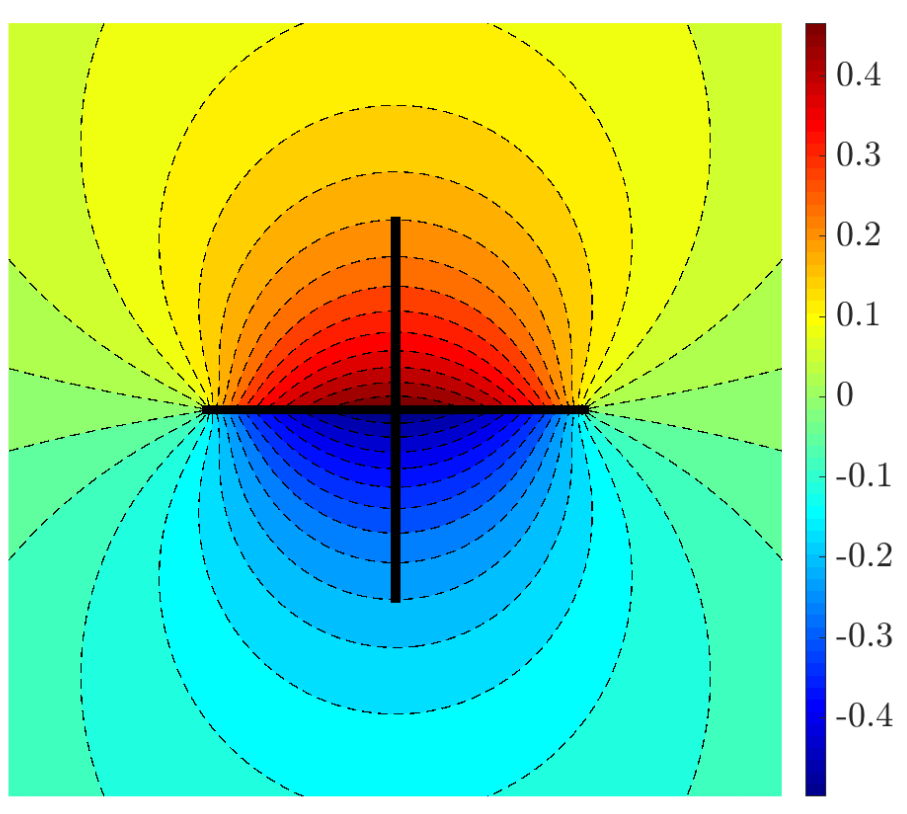}\quad\includegraphics[width=0.3\textwidth]{trueSol}
		\quad\includegraphics[width=0.3\textwidth]{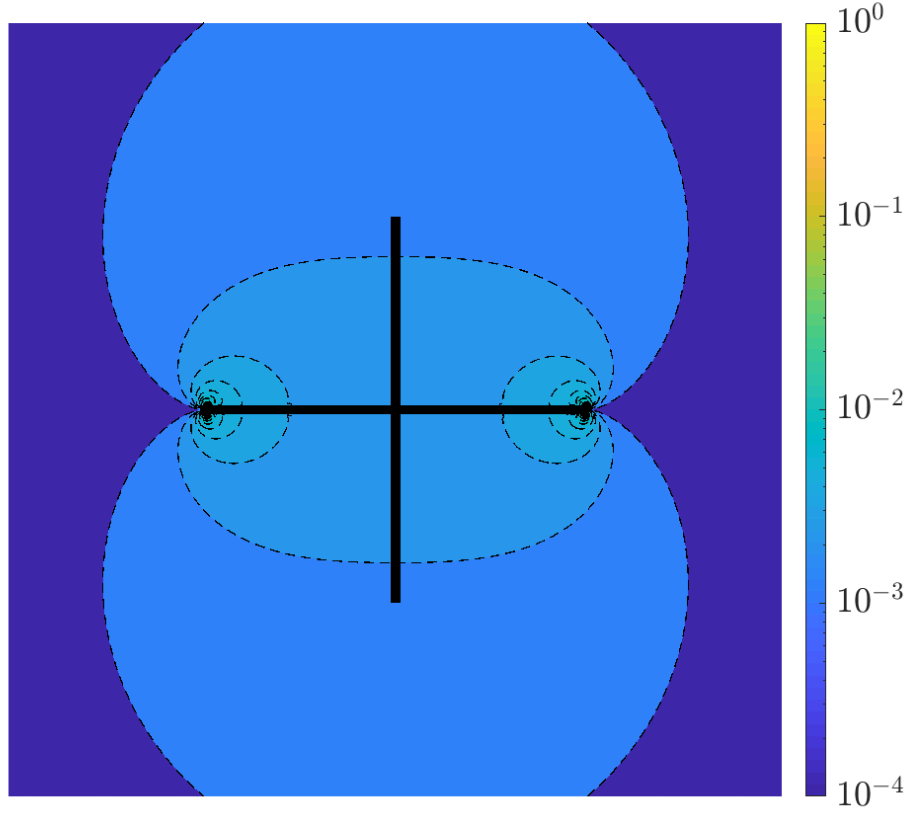}
		\caption{Left: solution $\mathcal{D}$ computed via the new method described in this paper, with a uniform mesh of size $h = 0.05$. Middle: exact solution. Right: error in base 10 logarithmic scale. The ``new" method produces the correct solution, up to a small error concentrated near the edge singularities of the exact solution $\mathcal{U}$.}
	\end{figure}

	In the example above, the true solution has a single-valued jump on the multi-screen, hence one may expect that some better choice of normal vector might still allow the naive BEM to find the right solution. In the next example, we change the Neumann condition in a way that makes the solution truly $4$-valued at the cross-point. In this case, the exact solution is not known analytically, but it is clear that the naive BEM cannot converge to the right solution, since it can only have up to $2$ different limits at the cross-point. \Cref{fig:otherSol} shows a comparison between the two methods in such a case. 
	
	\begin{figure}
		\centering 
		\includegraphics[width=0.45\textwidth]{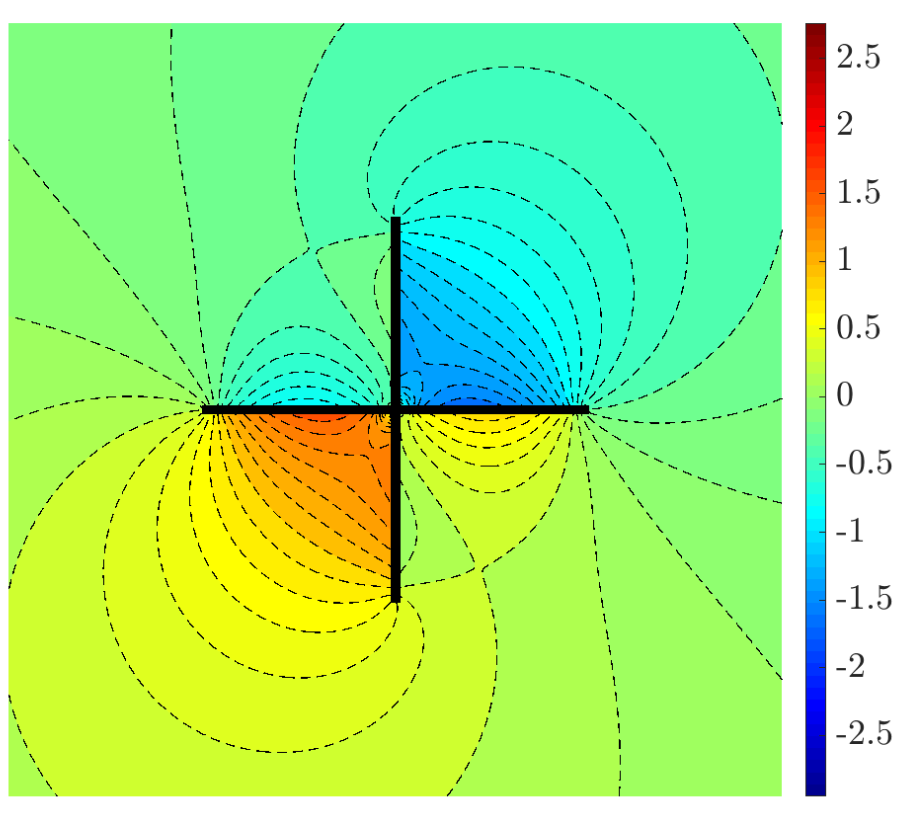}\qquad\includegraphics[width=0.45\textwidth]{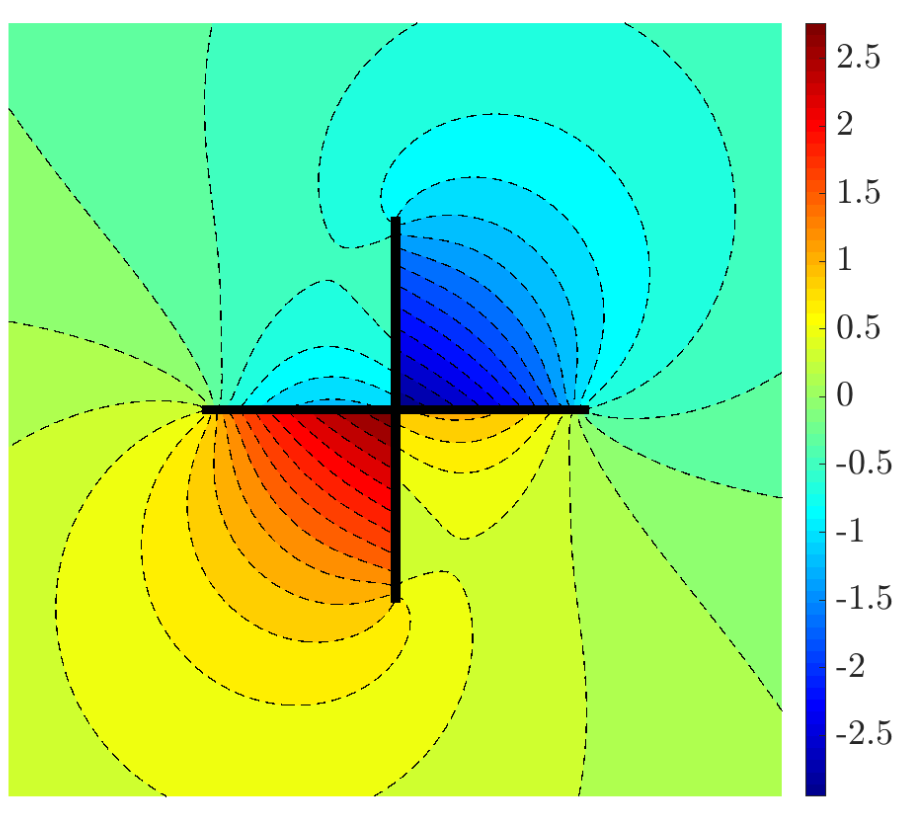}
		\caption{Approximate solutions of Problem \eqref{PDE_intro} with a Neumann condition given by the constant vector field $\vec g = (1,2)^T$. Left: ``naive" method, Right: ``new" method. The solutions produced by the two methods are conspicuously different at the center of the cross.}
		\label{fig:otherSol}
	\end{figure}
	
	\paragraph{Experiment 2. Condition numbers for a 2D multi-screen}
	
	We now illustrate our main result about the substructuring preconditioner, first for a $2D$ setting (our presentation is restricted to dimension $3$, but the analysis carries over in the easier case of dimension $2$). The multi-screen used in this example is a ``threefold junction", that is, a set of three line segments joining the center of gravity of an equilateral triangle to its vertices. We compare in \Cref{fig:curves2D} the spectral condition number for the linear system when no preconditioner is used, to the condition number $\kappa(P_{\rm ad}(H;h))$ of the preconditioned linear system using our substructuring domain decomposition method. As expected from results available for regular geometries, see \cite[Section 4.5]{sauter2011boundary}, we observe a condition number of the linear system without preconditioner behaving like $O(h^{-1})$. The growth of the spectral condition number for the preconditioned linear system is in agreement with \Cref{thm:main}.
\begin{figure}[H]
	\centering
	\includegraphics[width=0.95\textwidth]{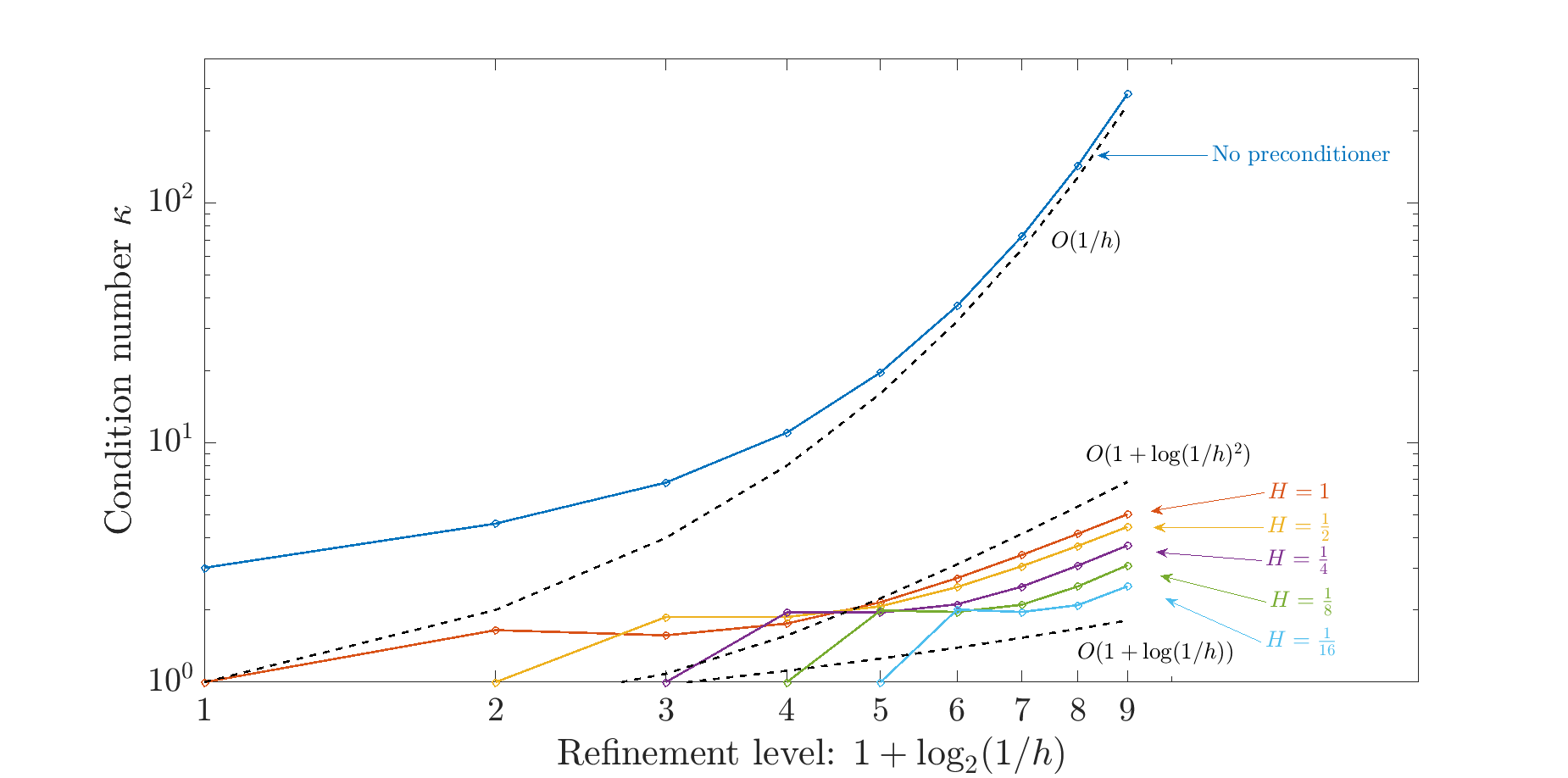}
	\includegraphics[width=0.95\textwidth]{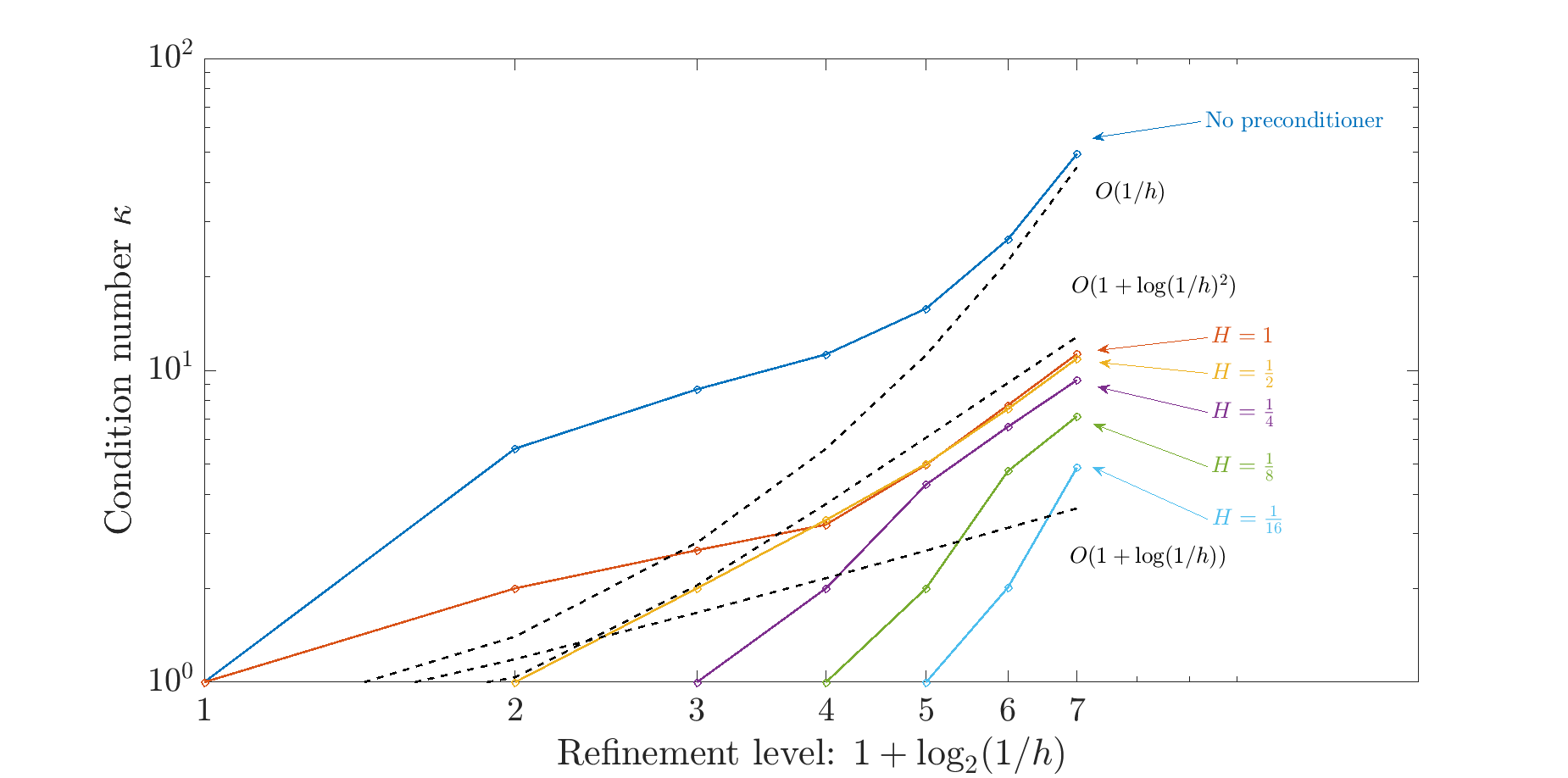}
	\caption{
		Condition number $\kappa$ (in log-scale) of the preconditioned linear system, as a function of $H$ and $h$, and comparison with the condition number when no preconditioner is used. Top: 2D problem (threefold junction). Bottom: 3D problem (``bow-tie" multiscreen depicted in \Cref{fig:bowtie}).}
	\label{fig:curves2D}
\end{figure}
	\paragraph{Experiment 3. Condition numbers for a 3D multi-screen}

	We include analogous experimental results in a $3D$ setting in \Cref{fig:curves2D}, bottom panel, using the multi-screen geometry in \Cref{fig:bowtie}. The results are qualitatively similar to those in $2D$, and illustrate the sharpness of \Cref{thm:main}, in particular with respect to the power of the logarithmic factor in the estimate. 
	\begin{figure}
		\centering
		\raisebox{-0.5\height}{\includegraphics[height=5cm]{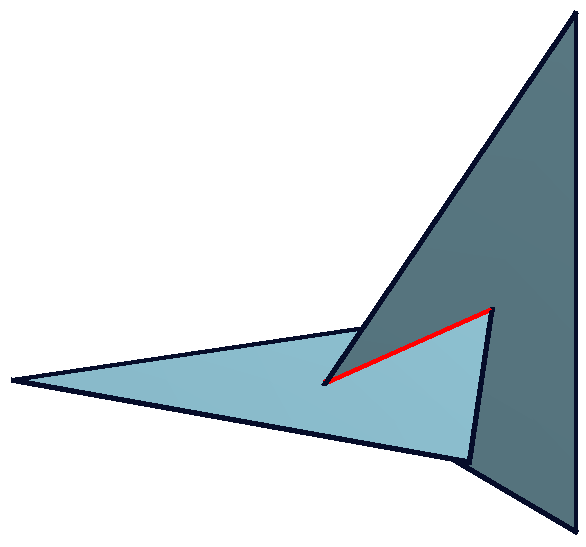}} \qquad \raisebox{-0.5\height}{\includegraphics[height=6cm]{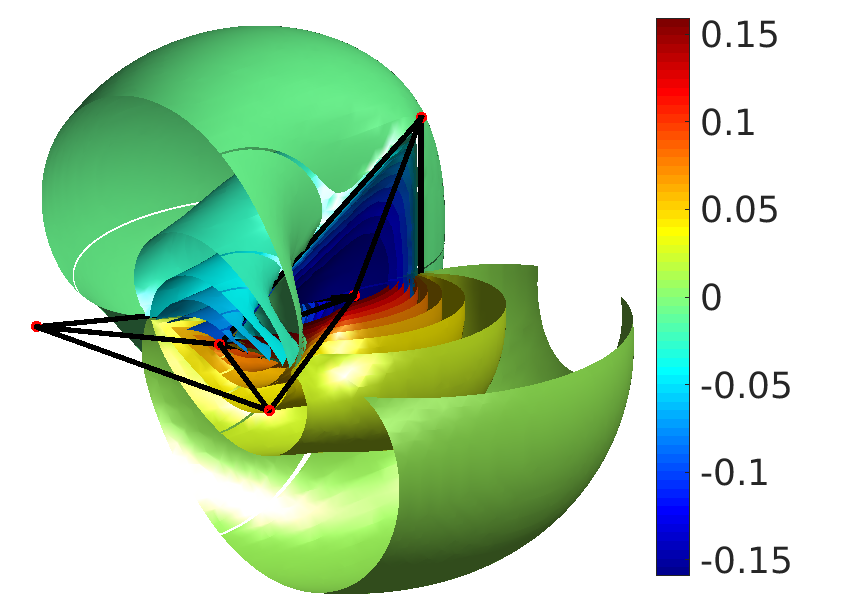}}
		\caption{Left: the polygonal multiscreen used in the experiments of Figure \ref{fig:curves2D} (bottom panel). It is composed of two equal-sized and perpendicularly arranged equilateral triangles, intersecting along a common median (highlighted in red). Right: some isosurfaces of the numerical approximation of the field $\mathcal{U}$ solving the problem \eqref{PDE_intro} with a constant vector $\vec g = (1,0.5,0.25)^T$.}
		\label{fig:bowtie}
	\end{figure}

	We now continue with the definition and analysis of the additive Schwarz preconditioner, and the proof of \Cref{thm:main}.
	
	\section{Laplace hypersingular Boundary Integral Equation on Multiscreens}
	\label{sec:HypersingularBIE}
	
	In this section, we formulate a precise boundary value problem for the Laplace equation in $\R^3 \setminus \Gamma$ with Neumann conditions on the multiscreen. We give an equivalent reformulation of this problem a boundary integral equation. Most of the material is recalled from \cite{claeys2013integral}. For conciseness, and with our boundary element application in mind, we restrict the presentation to {\em polygonal} multiscreens.
	
	\subsection{Polygonal multi-screens}
	\label{sec:polyMS}
 	
 	We use the same notation as in \cite{averseng2022fractured}. An {\em $n$-simplex} $S$, for $n \in \{0,1,2,3\}$, is a set of $n+1$ affinely independent points in $\R^3$, called the vertices of $S$. The closed convex hull of the vertices of $S$ is denoted by $\abs{S}$. The simplex $S$ is a {\em vertex}, {\em edge}, {\em triangle}, and {\em tetrahedron} when $n = 0$, $1$, $2$, and $3$, respectively. 
	For $n \geq 1$, the {\em facets} of $S$ are the $(n-1)$-simplices $S'$ such that $S' \subset S$; the set of all facets of $S$ is denoted by $\mathcal{F}(S)$.
	\begin{definition}[Simplicial mesh]
		 An \emph{$n$-dimensional {\em mesh}} $\mathcal{M}$ is a finite set of $n$-simplices satisfying the condition
		\[\forall (S,S') \in \mathcal{M} \times \mathcal{M}\,, \quad \abs{S \cap S'} = \abs{S} \cap \abs{S'}\,.\]
		Given an $n$-dimensional mesh $\mathcal{M}$, let 
		\[\mathcal{F}(\mathcal{M}) \isdef \bigcup_{S \in \mathcal{M}} \mathcal{F}(S)\,.\]
		For $n\geq 1$, the {\em boundary} $\partial \mathcal{M} \subset \mathcal{F}(\mathcal{M})$ is the $(n-1)$-dimensional mesh defined as the set of faces that occur in exactly one simplex of $\mathcal{M}$, that is, 
		\[\partial \mathcal{M} \isdef \enstq{F \in \mathcal{F}(\mathcal{M})}{\exists! S \in \mathcal{M} \,\textup{ s.t. }\, F \in \mathcal{F}(S)}\,.\] 
		The {\em geometry} of a mesh $\mathcal{M}$ is defined by 
		\[\abs{\mathcal{M}} \isdef \bigcup_{S \in \mathcal{M}} \abs{S}\,.\]
		The $n$-dimensional mesh $\mathcal{M}$ is {\em regular} if its geometry is a manifold. If $n \geq 1$, then $\partial \mathcal{M}$ is also a regular mesh, and there holds $\abs{\partial \mathcal{M}} = \partial \abs{\mathcal{M}}$.
	\end{definition}  
	
	\begin{definition}[Polygonal multi-screen]
		\label{def:polygonalMS}
		A set $\Gamma$ is a \emph{polygonal multi-screen}, if there exists a regular tetrahedral mesh $\mathcal{M}_{\Omega}$ of a sufficiently large open cube $\Omega = [-l,l]^3\subset \R^3$, $l > 0$, and a triangular mesh $\mathcal{M}_\Gamma \subset \mathcal{F}(\mathcal{M}_{\Omega}) \setminus \partial \mathcal{M}_\Omega$ such that 
		\[\Gamma = \abs{\mathcal{M}_\Gamma}\,.\]
		The mesh $\mathcal{M}_\Omega$ is further assumed to be partitioned into a collection of regular tetrahedral meshes $\mathcal{M}_{\Omega_1},\ldots,\mathcal{M}_{\Omega_J}$, in such a way that $\mathcal{M}_\Gamma \subset \partial \mathcal{M}_{\Omega_1} \cup \ldots \cup \partial \mathcal{M}_{\Omega_J}$ and, for each $j \in \{1\,,\ldots,J\}$, the intersection $\Gamma \cap \partial \Omega_j$ is a Lipschitz screen (i.e. a Lipschitz manifold with Lipschitz boundary) where $\Omega_j = \textup{int}(\abs{\mathcal{M}_{\Omega_j}})$. 
	\end{definition} 
	It follows from the definition that a multi-screen $\Gamma$ is a compact set. Setting in addition $\Omega_0 \isdef \R^3 \setminus \overline{\Omega}$, the sets $\Omega_0,\ldots,\Omega_J$ then define a Lipschitz partition of $\R^3$, in the sense of \cite[Definition 2.2]{claeys2013integral}. A polygonal multi-screen is thus a particular case of a multi-screen in the sense of \cite[Definition 2.3]{claeys2013integral}. The mesh $\mathcal{M}_\Omega$ is merely used for theoretical analysis and is not needed in our algorithm.

	In the remainder of this work, we fix a polygonal multi-screen $\Gamma$. For convenience, we further assume that $\R^3 \setminus \Gamma$ is connected.\footnote{This ensures uniqueness of the solution $\mathcal{U}$ of \eqref{PDE_intro}. If $\R^3 \setminus \Gamma$ has several connected components, the solution is unique up to adding constants in the bounded connected components. The material discussed here can easily be adapted to handle this situation  -- in particular the case where $\Gamma$ is a closed surface such as the boundary $\partial \mathcal{P}$ of a polyhedron $\mathcal{P}$ -- but we omit this for conciseness.} We denote by $\gamma_j: H^1(\Omega_j) \to L^2(\partial \Omega_j)$ the pointwise trace operator \cite[p. 100]{McLean}. In \Cref{sec:dd}, we require that all $\Omega_j$  for $j \neq 0$ be tetrahedra of diameter bounded by some constant $H >0$, thus providing a coarse mesh of $\mathcal{M}_\Omega$. This can be achieved, if necessary, by redefining the sets $\Omega_j$. For now, we impose no restrictions on the size of the domains $\Omega_j$ and the constants in the estimates proved in the next section are thus independent of $H$.

	\subsection{Quotient trace spaces} For an open set $U \subset \R^3$, let $C^\infty_c(U)$ be the set of real-valued functions $u$ that are infinitely differentiable and compactly supported on $U$. Let $L^2(U)$ be the set of real-valued square-integrable functions on $U$. We denote by $H^1(U)$ the Sobolev space of functions $u \in L^2(U)$ such that there exists a square-integrable vector field $\vec p \in (L^2(U))^3$ satisfying
	\[\int_{U} u \div \vec \phi\,dx  = - \int_{U} \vec p \cdot \vec \phi\,dx \quad \forall \vec \phi \in (C^\infty_c(U))^3\,.\]
	Writing $\nabla u \isdef \vec p$ for the {\em weak gradient} of $u$ on $U$, a Hilbert structure is defined on $H^1(U)$ by
	\[\norm{u}^2_{H^1(U)} \isdef \norm{u}^2_{L^2(U)} + \norm{\nabla u}^2_{L^2(U)}\,.\footnote{We emphasize that our notation for $H^1(U)$ differs from the standard \cite[Chap. 3]{McLean}, where this space is denoted by $W^1(U)$, and where $H^1(U)$ is instead defined via Fourier transforms (with the two definitions coinciding, e.g., when $U$ is a Lipschitz domain, but this will not be the case for most instances of $U$ below).}\]
	Let $H^1_{0,\Gamma}(\R^3)$ be the closure of $C^\infty_c(\R^3 \setminus \Gamma)$ in $H^1(\R^3)$. The {\em multi-trace space} $\mathbb{H}^{1/2}(\Gamma)$ is the Hilbert space defined by the quotient \cite[eq. (5.1)]{claeys2013integral}
	\[\mathbb{H}^{1/2}(\Gamma) \isdef H^1(\R^3 \setminus \Gamma) / H^1_{0,\Gamma}(\R^3) \,.\]
	The {\em (Dirichlet) multi-trace operator} is defined as the canonical surjection 
	\begin{equation}
		\label{defMultiTrace}
		\gamma: H^1(\R^3 \setminus \Gamma) \to \mathbb{H}^{1/2}(\Gamma)
	\end{equation} 
	associated to this quotient space. By definition of quotients of Hilbert spaces,
		\begin{align*}
		\norm{u}_{\mathbb{H}^{1/2}(\Gamma)} &= \min \enstq{\norm{V}_{H^1(\R^3 \setminus \Gamma)}}{\gamma V = u,\,\,\, V \in H^1(\R^3 \setminus \Gamma)}\\
		& = \min\enstq{\norm{U + U_0}_{H^1(\R^3 \setminus \Gamma)}}{U_0 \in H^1_{0,\Gamma}(\R^3)} \quad \forall U \in H^1(\R^3 \setminus \Gamma)\textup{ s.t. } \gamma U = u\,.
	\end{align*}
	Let $H^{1/2}([\Gamma])$ be the {\em single-trace space}, which is the closed subspace of $\mathbb{H}^{1/2}(\Gamma)$ defined by 
	\[H^{1/2}([\Gamma]) \isdef \gamma(H^1(\R^3))\,.\] 
	In turn, the {\em jump space} $\widetilde{H}^{1/2}(\Gamma)$ is the Hilbert space defined by the quotient \cite[Proposition 6.8]{claeys2013integral}
	\[\widetilde{H}^{1/2}([\Gamma]) \isdef \mathbb{H}^{1/2}(\Gamma) / H^{1/2}([\Gamma])\,,\]
	and the {\em jump operator} $[\cdot]_{\Gamma}$ is defined as the corresponding canonical surjection. We will also conveniently write $[u]_\Gamma$ as short for $[\gamma(u)]_\Gamma$ when $u \in H^1(\R^3 \setminus \Gamma)$. 
	
	With a similar construction, where the role of the gradient is played by the divergence, one defines $H(\div,\R^3)$, $H_{0,\Gamma}(\div,\R^3)$ and the quotient space
	\[\mathbb{H}^{-1/2}(\Gamma) \isdef H(\div,\R^3 \setminus \Gamma) / H_{0,\Gamma}(\div,\R^3)\,,\]
	and $\pi_{n}$ refers to the associated canonical surjection. Again, $H^{-1/2}([\Gamma])$ is the single-trace space of $\mathbb{H}^{-1/2}(\Gamma)$, defined by 
	\[H^{-1/2}([\Gamma]) \isdef \pi_n \left(H(\div,\R^3)\right)\,.\] 
	A well-defined bilinear form is obtained by \cite[eq. (5.2)]{claeys2013integral}
	\[\forall (u,v) \in \mathbb{H}^{1/2}(\Gamma) \times \mathbb{H}^{-1/2}(\Gamma)\,, \quad \dduality{u}{v} \isdef \sum_{j = 1}^J \int_{\Omega_j} \nabla f(x) \cdot \vec p(x) + f(x) \div \vec p(x)\,dx\,, \]
	where $f$ and $\vec p$ are arbitrary representatives of $u$ and $v$, i.e.,
	$u = \gamma(f)$, $v = \pi_n(\vec p)$. This realizes an isometric duality pairing in the sense that \cite[Prop. 5.1]{claeys2013integral}
	\begin{equation}
		\label{eq:isometricPairing}
		\norm{u}_{\mathbb{H}^{1/2}(\Gamma)} = \sup_{\varphi \in \mathbb{H}^{-1/2}(\Gamma)} \frac{\dduality{u}{\varphi}}{\norm{\varphi}_{\mathbb{H}^{-1/2}(\Gamma)}}\,, \quad \norm{v}_{\mathbb{H}^{-1/2}(\Gamma)} = \sup_{\psi \in \mathbb{H}^{1/2}(\Gamma)} \frac{\dduality{u}{\psi}}{\norm{\psi}_{\mathbb{H}^{1/2}(\Gamma)}}\,.
	\end{equation}
	Moreover, the single-trace spaces $H^{\pm 1/2}([\Gamma])$ are each other's {\em polar} under this bilinear form \cite[Proposition 6.3]{claeys2013integral}, i.e., 
	\begin{equation}
		\label{eq:polarity1}
		H^{1/2}([\Gamma]) = \enstq{u \in \mathbb{H}^{1/2}(\Gamma)}{\dduality{u}{v} = 0\,,\,\, \forall v \in H^{-1/2}([\Gamma])}\,,
	\end{equation}
	\begin{equation}
		\label{eq:polarity2}
		H^{-1/2}([\Gamma]) = \enstq{v \in 	\mathbb{H}^{-1/2}(\Gamma)}{\dduality{u}{v} = 0\,,\,\, \forall u \in H^{1/2}([\Gamma])}\,.
	\end{equation}
	We further introduce the Hilbert space
	$$H^1(\Delta,U) \isdef \enstq{u \in H^1(U)}{\nabla u \in H(\div,U)}\,,$$
	with the norm $\norm{u}_{H^1(\Delta,U)}^2 \isdef \norm{u}_{H^1(U)}^2 + \norm{\nabla u}_{H(\div,U)}^2$, and let $H^1_{\rm loc}(\Delta,U)$ be the space of functions $u$ such that $\chi u \in H^1(\Delta,U)$ for any smooth compactly supported functions $\chi$. 
%
%

	\subsection{Exterior Neumann boundary value problem}
	
 	We seek the solution $\mathcal{U}$ of the solution of the Boundary Value Problem (BVP)
	\begin{equation}
		\label{PDE}
		\Delta \mathcal{U} = 0 \quad \textup{ in } \R^3 \setminus {\Gamma}
	\end{equation} 
	with a prescribed normal derivative on $\Gamma$, and the decay condition
	\begin{equation}
		\label{decay}
		\mathcal{U}(x) = O\left(\frac{1}{\norm{x}}\right)\,,
	\end{equation}
	uniformly as $x \to \infty$, where $\norm{x}$ is the Euclidean norm of $x$. To prescribe the boundary condition, we supply a sufficiently regular vector field $\vec g$ on $\R^3$ such that the normal component of $\nabla \mathcal{U}$ agrees with $\vec g$ on $\Gamma$. More formally, we impose that 
	\begin{equation}
		\label{neumannCond}
		\pi_n(\nabla \mathcal{U}) = \pi_n(\vec g) \quad \textup{ in } \mathbb{H}^{-1/2}(\Gamma)\,,
	\end{equation}
	where $\vec g \in H(\div,\R^3)$. In particular, this requires the normal derivative of $\mathcal{U}$ to be ``continuous" across $\Gamma$, i.e., to be an element of the single-trace space $H^{-1/2}([\Gamma])$. 
	
	
	\subsection{Variational hypersingular boundary integral equation}
	We seek the solution $\mathcal{U}$ of the Neumann boundary value problem in the form of a double-layer potential. 
	\begin{definition}[Double-Layer potential]
		\label{def:DL}
		For $u \in \mathbb{H}^{1/2}(\Gamma)$, the {\em double-layer potential} $\DL u$ is defined by  
		\[\forall x \in \R^3 \setminus \Gamma\,, \quad \DL u(x) \isdef x \mapsto \dduality{u}{\pi_n (\nabla \mathcal{G}_{x})}\,,\]
		where
		\[\forall y \in \R^3\,, \quad \mathcal{G}_{x}(y) \isdef \begin{cases}
			0 & \textup{if } y = x,\\[0.5em]
			\dfrac{\chi_{x}(y)}{4\pi \norm{x - y}} & \textup{otherwise, }
		\end{cases}\]
		and $\chi_{x}$ is any smooth compactly supported function equal to $0$ in a neighborhood of $x$, and $1$ in a neighborhood of $\Gamma$. 
	\end{definition}
	The value of $\DL u(x)$ is independent of the particular choice of cutoff function, and \Cref{lemrepDL} gives a concrete integral representation of this operator which generalizes the commonly known formula. Furthermore, $\DL$ maps $\mathbb{H}^{1/2}(\Gamma)$ to $H^1_{\rm loc}(\Delta, \R^3\setminus \Gamma)$ continuously, satisfies the property
	\begin{equation}
		\label{kerDL}
		\quad \DL u = 0 \quad \forall u\in H^{1/2}([\Gamma])\,,
	\end{equation}
	and the {\em jump relation} \cite[Prop. 8.5]{claeys2013integral}
	\begin{equation}
		\label{jumpRel}
		 [\DL u]_\Gamma = [u]_\Gamma \quad \forall u \in \mathbb{H}^{1/2}(\Gamma)\,.
	\end{equation} 
	Finally, note that by the property \eqref{kerDL}, $\DL$ induces a linear continuous map on $\widetilde{H}^{1/2}([\Gamma])$, again denoted by $\DL$. For $u \in \widetilde{H}^{1/2}(\Gamma) $ and $U \subset \R^3 \setminus \Gamma$ an open set, $(\DL u)_{|U} \in C^\infty(U)$ and $\Delta(\DL u) = 0$ on $U$. Moreover, $\DL u$ satisfies the decay condition \eqref{decay}. 
	
	\begin{proposition}[{Hypersingular operator \cite[Section 8]{claeys2013integral}}]
		\label{def:Hypersingular}
		The {\em hypersingular operator} $\W \isdef \pi_n \circ \nabla \circ \DL$ is well-defined and continuous from $\mathbb{H}^{1/2}(\Gamma) \to {H}^{-1/2}([\Gamma])$. 
	\end{proposition}
	Let $b: \mathbb{H}^{1/2}(\Gamma) \times \mathbb{H}^{1/2}(\Gamma) \to \mathbb{R}$ be the bilinear form defined by 
	\begin{equation}
		b(u,v) \isdef \dduality{v}{\W u}\,, \quad \forall u,v \in \mathbb{H}^{1/2}(\Gamma)\,.
	\end{equation}	
	Then, by \Cref{idDLDL}, the bilinear form $b$ is symmetric and positive, but it is only semi-definite; due to the relation \eqref{kerDL}, it satisfies $b(u,\cdot) = 0$ for all $u \in H^{1/2}([\Gamma])$.
	However, we may define a new bilinear form $a: \widetilde{H}^{1/2}([\Gamma]) \times \widetilde{H}^{1/2}([\Gamma]) \to \R$ by quotienting $b$ with respect to $H^{1/2}(\Gamma)$, as follows.
	\begin{definition}[The hypersingular bilinear form]	
		\label{def:defa}
		For any $\varphi,\psi \in \widetilde{H}^{1/2}([\Gamma])$, we define
		\begin{equation}
			\label{eq:defa}
			a(\varphi,\psi) \isdef b(u,v) \quad \textup{where } u,v \in 	\mathbb{H}^{1/2}(\Gamma) \textup{ satisfy } [u]_\Gamma = \varphi\,,\,\, [v]_\Gamma = \psi\,.
		\end{equation}
	\end{definition}
	This definition is valid (i.e., it does not depend on the choice of $u$ and $v$) since $H^{1/2}([\Gamma])$, which is the kernel of $[\cdot]_\Gamma$ in $\mathbb{H}^{1/2}(\Gamma)$, is also in the kernel of $b$.
	From the mapping properties of $\W$ and \Cref{thmWpos}, we immediately obtain the following result.
	\begin{theorem}[Coercivity of the hypersingular bilinear form]
		\label{lemequivInner}
		The hypersingular bilinear form $a$ is continuous, positive definite and bounded from below. It induces an equivalent inner product on $\widetilde{H}^{1/2}(\Gamma)$, i.e. there exist constants $c_W,C_W >0$ such that 
		\[\forall \varphi \in \widetilde{H}^{1/2}([\Gamma])\,, \quad c_W \norm{\varphi}^2_{\widetilde{H}^{1/2}([\Gamma])} \leq a(\varphi,\varphi) \leq C_W \norm{\varphi}^2_{\widetilde{H}^{1/2}([\Gamma])}\,.\]
	\end{theorem}
	We introduce the linear form 
	\begin{equation}
		\label{defRhs}
		l_{g}: \widetilde{H}^{1/2}(\Gamma)\ni [u]_\Gamma \mapsto \dduality{\gamma u}{\pi_n(\vec g)}\,.
	\end{equation}
	This is a well-defined continuous linear form by the polarity property \eqref{eq:polarity1}.
	\begin{theorem}[Variational formulation of the the Laplace Neumann boundary value problem]
		\label{thmBVP}
		The variational problem
		\begin{equation}
			\label{varPb}
			\textup{Find } \varphi \in \widetilde{H}^{1/2}([\Gamma]) \textup{ such that } a(\varphi,\psi) = l_g(\psi) \textup{ for all } \psi \in \widetilde{H}^{1/2}([\Gamma])\,.
		\end{equation}
		has a unique solution $\varphi^*$, and $\mathcal{U} = \DL \varphi^*$ is the unique solution of the BVP
		\begin{equation}
			\label{BVP}
			\begin{cases}
				\Delta \mathcal{U} = 0 & \textup{in } \R^3 \setminus \Gamma\,,\\[0.3em]
				\mathcal{U} = O\left(\frac{1}{\norm{x}}\right) & \textup{uniformly for } x \to \infty\,,\\[0.3em]
				\pi_n(\nabla \mathcal{U}) = \pi_n(\vec g) & \textup{in } \mathbb{H}^{-1/2}(\Gamma)\,.
			\end{cases}
		\end{equation}
	\end{theorem}
	The proof is given in \Cref{sec:proofPropertiesW}.

	\subsection{Weakly singular integral representations}
	\label{sec:weakSingRep}
	
	We denote by $u_j \isdef u_{|\Omega_j}$ the restriction of $u$ to $\Omega_j$. Let $\vec n_j$ be the outward pointing unit normal vector of $\partial \Omega_j$ and $d\sigma_j$ the surface measure on $\partial \Omega_j$. Recall that $\gamma_j: H^1(\Omega_j) \to L^2(\partial \Omega_j)$ is the point trace operator. Let $\nabla_j$ be the tangential gradient on $\partial \Omega_j$, and $\curl_j \isdef \vec n_j \times \nabla_j$ the surface curl on $\partial \Omega_j$.
	
	On the boundaries $\partial \Omega_j$ of the Lipschitz domains $\Omega_j$, the spaces $H^1(\partial \Omega_j)$ are defined with the help of coordinate charts, see e.g. \cite[p. 96]{McLean}. The tangential gradient $\nabla_{j}$ on $\partial \Omega_j$ extends uniquely to a continuous map $H^1(\partial \Omega_j) \to (L^2(\partial \Omega_j))^3$. We denote by $\curl_j: H^1(\partial \Omega_j) \to (L^2(\partial \Omega_j))^3$ the operator defined by 
	$\curl_j u \isdef \vec n_j \times \nabla_j u$. 
	\begin{theorem}[Weakly singular representation of the hypersingular operator]
		\label{thmWweak}
		Let $u,v \in H^1(\R^3\setminus \Gamma)$ and suppose that $\gamma_j u_j$, $\gamma_j v_j$ belong to $H^1(\partial \Omega_j)$. Then there holds
		\[a\left([u]_\Gamma,[v]_\Gamma\right) = \sum_{j = 1}^{J}\sum_{k = 1}^J \iint_{\Gamma_j\times \Gamma_k} \frac{\vec  \curl_j\, \gamma_j u_j(x) \cdot \curl_k\, \gamma_k u_k(x')}{4\pi \norm{x - x'}}\,d\sigma_j(x) d\sigma_k(x')\,,\]
		with $\Gamma_j := \partial \Omega_j \cap \Gamma$ as in \Cref{def:polygonalMS}. 
	\end{theorem}
	The proof is given in \Cref{sec:proofPropertiesW}. In practice, it is useful to rewrite the above expression in terms of weakly singular integrals over pairs of triangles. To this end, a key ingredient is the so-called ``virtually inflated mesh" introduced in \cite[Section 4]{claeys2021quotient} and studied in more depth in \cite{averseng2022fractured}.
	
	\begin{definition}[Inflated mesh]
		\label{defMstar}
		Assume that $\mathcal{M}_{\Gamma,h}$ and $\mathcal{M}_{\Omega,h}$ are ``compatible" mesh refinements of $\mathcal{M}_{\Gamma}$ and $\mathcal{M}_\Omega$, in the sense that $\mathcal{M}_{\Gamma,h} \subset \mathcal{F}(\mathcal{M}_{\Omega,h}) \setminus \partial \mathcal{M}_{\Omega,h}$. The {\em inflated mesh} $\mathcal{M}^*_{\Gamma,h}$  is defined by 
		\begin{equation}
			\label{def:genMesh}
			\mathcal{M}^*_{\Gamma,h} \isdef \enstq{\mathbf{t} = (T,K) \in \mathcal{M}_{\Gamma,h} \times \mathcal{M}_{\Omega,h}}{T \in \mathcal{F}(K)}\,.
		\end{equation}
	\end{definition}
	The elements $\vec t = (T,K) \in \mathcal{M}^*_{\Gamma,h}$ model the triangles $T$ of $\mathcal{M}_\Gamma$ attached to a ``side" of the surface $\Gamma$ (the side determined by the position of the tetrahedron $K$). 
	The inflated mesh thus contains twice as many elements as $\mathcal{M}_{\Gamma,h}$: each triangle $T \in \mathcal{M}_{\Gamma}$ occurs exactly in two pairs $(T,K^+)$ and $(T,K^-)$. The inflated mesh $\mathcal{M}^*_{\Gamma,h}$ can be equivalently represented as a set of \emph{oriented triangles}, by associating to $\vec t = (T,K) \in \mathcal{M}^*_\Gamma$ the triangle $T$ oriented by the normal vector pointing inside $\abs{K}$ (recall that $\abs{S}$ is the convex hull of the simplex $S$). We denote this normal vector by $\vec n_{\vec t}$.  We also write $\abs{\vec t}$ as short for $\abs{T}$.	Let $\gamma_{\vec t}$ be the trace operator from the tetrahedron $\abs{K}$ to its face $\abs{T}$, $\nabla_{\vec t}$ the tangential gradient on $\abs{\vec t}$, $\curl_{\vec t} \isdef \vec n_{\vec t} \times \nabla_{\vec t}$, and $\sigma_{\vec t}$ the surface measure on $\abs{\vec t}$.
	
	\begin{corollary}
		\label{corW}
		Under the same assumptions as in \Cref{thmWweak}, there holds
		\[a([u]_\Gamma,[v]_\Gamma) = \sum_{\vec t, \vec t' \in \mathcal{M}^*_{\Gamma,h}} \iint_{\abs{\vec t} \times \abs{\vec t'}} \frac{\curl_{\vec t} \gamma_{\vec t} u (x) \cdot \curl_{\vec t'} \gamma_{\vec t'} v(x')}{4\pi \norm{x - x'}} d\sigma_{\vec t}(x)\, d\sigma_{\vec t'}(x')\,.\]
	\end{corollary} 
	From the proofs in \Cref{sec:proofPropertiesW}, we also record the following expression for the double-layer potential.
	\begin{lemma}[Representation of the double-layer potential]
		\label{lemrepDL}
		For all $u \in H^1(\R^3 \setminus \Gamma)$, 
		\begin{align}
			\label{formulaDL}
			\forall x \in \R^3 \setminus \Gamma\,, \quad \DL \gamma u\,(x) &= \sum_{j = 1}^J \int_{\Gamma_j}  \frac{\vec n_j(y) \cdot (y - x)}{4\pi \norm{x - y}^3}  \gamma_j u_j(y)\,d\sigma_j(y)\\
			&= \sum_{\vec t \in \mathcal{M}^*_{\Gamma,h}} \int_{\abs{\vec t}} \frac{\vec n_{\vec t} \cdot (y - x)}{4\pi \norm{x - y}^3} \gamma_{\vec t} u(y) d\sigma_{\vec t}(y)\,.\nonumber
		\end{align}
	\end{lemma}

	\section{Galerkin Boundary Element Method}
	\label{sec:GalerkinBEM}

	\Cref{thmBVP} immediately suggests a method for the numerical resolution of the BVP in \eqref{BVP}. Namely, we find an approximation $\varphi_h^*$ of the solution $\varphi^*$ of the variational problem \eqref{varPb} in a subspace $\widetilde{V}_h(\Gamma) \subset \widetilde{H}^{1/2}([\Gamma])$ by the Galerkin method. We then compute $\mathcal{U}_h = \DL \varphi^*_h$, with $\mathcal{U}_h$ being the proposed approximation of $\mathcal{U}$. 
	
	\subsection{Convergent Galerkin approximation}
	
	\begin{definition}[{Families of subspaces of $H^1(\R^3\setminus \Gamma)$, $\mathbb{H}^{1/2}(\Gamma)$ and  $\widetilde{H}^{1/2}([\Gamma])$}]
		Given a uniformly shape-regular family $(\mathcal{M}_{\Omega,h})_{h > 0}$ 
		of refinements of the mesh $\mathcal{M}_{\Omega}$ introduced in \Cref{sec:polyMS}, 
		such that, for each $h > 0$, the length of the longest edge in $\mathcal{M}_{\Omega,h}$ is bounded by $h$, and 
		\[\forall h > 0\,, \quad \mathcal{M}_{\Gamma,h} \subset \mathcal{F}(\mathcal{M}_{\Omega,h}) \setminus \partial \mathcal{M}_{\Omega,h}\,,\]
		let 
		\begin{equation}
			\label{defVh}
			V_h(\Omega \setminus \Gamma) \isdef \enstq{u \in H^1(\R^3 \setminus \Gamma)}{u = 0 \textup{ on } \Omega_0 \textup{ and } u_{|K} \textup{ is linear } \forall K \in \mathcal{M}_{\Omega,h}}\,,
		\end{equation}
		\begin{equation}
			\label{defDiscreteTraceJump}
			\mathbb{V}_h(\Gamma) \isdef \gamma(V_h(\Omega \setminus \Gamma)) \,, \quad \widetilde{V}_h([\Gamma]) \isdef [V_h(\Omega \setminus \Gamma)]_\Gamma\,.
		\end{equation}
		We call $\mathbb{V}_h(\Gamma) \subset \mathbb{H}^{1/2}(\Gamma)$ the \emph{discrete multi-trace space} and $\widetilde{V}_h([\Gamma]) \subset \widetilde{H}^{1/2}([\Gamma])$ the \emph{discrete jump space}. 
	\end{definition}
	We define the approximation $\varphi^*_h \in \widetilde{V}_h([\Gamma])$ of $\varphi^*$ as the solution of the variational problem
	\begin{equation}
		a(\varphi^*_h,\psi_h) = l_g(\psi_h) \quad \forall \psi_h \in \widetilde{V}_h([\Gamma])\,.
	\end{equation}
	Since the bilinear form $a$ is positive definite, by Céa's lemma, $\varphi^*_h$ is a quasi-optimal approximation of $\varphi^*$, in the sense that there exists a constant $C > 0$ such that for all $h >0$, 
	\begin{equation}
		\label{eq:qo}
		\norm{\varphi^* - \varphi^*_h}_{\widetilde{H}^{1/2}([\Gamma])} \leq C \inf_{\psi_h \in \widetilde{V}_h([\Gamma])}\norm{ \varphi^* - \psi_h}_{\widetilde{H}^{1/2}([\Gamma])}\,.
	\end{equation}
	
	\begin{theorem}[Convergence of $\varphi^*_h$ to $\varphi^*$]
		\label{thmConvPhiStar}
		The Galerkin method is convergent, i.e., the approximations $(\varphi^*_h)_{h > 0}$ of $\varphi^*$ satisfy
		\begin{equation*}
			\lim_{h \to 0} \varphi^*_h = \varphi^* \quad \textup{ in } \widetilde{H}^{1/2}([\Gamma])\,.
		\end{equation*}
	\end{theorem}
	The proof is given in \Cref{sec:proofGalerkConv}. In turn, let $\mathcal{U}_h \isdef \DL \varphi_h^*$. From the convergence of $\varphi^*_h$ and the mapping properties of $\DL$, we deduce immediately the convergence of $\mathcal{U}_h$ in an appropriate sense given below. 
	
	\begin{corollary}
		\label{corConvHh}
		For every compact set $K$ of $\R^3$, there holds
		\[\lim_{h \to 0}\norm{\mathcal{U}_h - \mathcal{U}}_{H^1(K \cap (\R^3 \setminus \Gamma))} = 0\,.\]
	\end{corollary}
	In the rest of this section, we address the practical computation of $\mathcal{U}_h$. We rely on the construction of a basis $\{\widetilde{\phi}_\nu\}_{1 \leq \nu \leq \widetilde{N}_h}$ of the space $\widetilde{V}_h(\Gamma)$, in such a way that the quantities
	\begin{equation}
		\label{requiredAlgo}
		a(\widetilde{\phi}_\nu,\widetilde{\phi}_{\nu'})\,, \quad l_g(\widetilde{\phi}_\nu)\,, \quad \textup{and} \quad \DL \widetilde{\phi}_\nu
	\end{equation}
	can be evaluated algorithmically. Thus, we can build the linear system 
	\[\mathbf{W} \Phi = L\,,\]
	where $\mathbf{W}$ is the $\widetilde{N}_h \times \widetilde{N}_h$ square matrix with entries $\mathbf{W}_{i,j} = a(\widetilde{\phi}_{\nu},\widetilde{\phi}_{\nu'})$, and $L$ is the column vector with entries $L_\nu= l_g(\widetilde{\phi}_\nu)$. The approximation $\mathcal{U}_{h}$ is then given by 
	\[\mathcal{U}_h = \DL \Bigg(\sum_{i = 1}^{\widetilde{N}_h} \Phi_i \widetilde{\phi}_{\nu}\Bigg)\,.\]
	The definition of the local shape functions $\{\widetilde{\phi}_\nu\}_{1 \leq \nu \leq \widetilde{N}_h}$ is also required for discussing the domain-decomposition preconditioner in Section \ref{sec:dd}.
	
	\subsection{Local shape functions}
	\label{sec:localShapeFunc}
	
	The overall idea is to define $\widetilde{\phi}_\nu \isdef [\psi_{\nu}]_\Gamma$, where $\{\psi_{\nu}\}_{1 \leq \nu \leq \widetilde{N}_h}$ is a basis of a subspace
	$\Psi_h(\Omega\setminus \Gamma) \subset V_h(\Omega \setminus \Gamma)$ chosen such that the jump operator 
	induces a bijection
	\begin{equation}
		[\cdot]_\Gamma : \Psi_h(\Omega \setminus \Gamma) \to \widetilde{V}_h([\Gamma])\,.
	\end{equation} 
	To proceed, let us denote by $\{\vec{x}_1\,,\ldots\,,\vec{x}_N\}$ the ordered set of vertices of $\mathcal{M}_{\Omega,h}$, with the common vertices of $\mathcal{M}_{\Omega,h}$ and $\mathcal{M}_{\Gamma,h}$ given by $\vec{x}_1,\ldots,\vec{x}_M$, $M < N$. Let $V_h(\Omega)$ be the space of functions in $V_h(\Omega \setminus \Gamma)$ which are continuous across $\Gamma$; notice that
	\[V_h(\Omega) = H^1(\R^3) \cap V_h(\Omega \setminus \Gamma)\,.\]
	Let $\{\phi_i\}_{1 \leq i \leq N}$ be the nodal basis of $V_h(\Omega)$, that is, the set of elements of $V_h(\Omega)$ defined by
	\[\phi_i(\vec{x}_{i'}) = \delta_{i,i'}\,, \quad 1 \leq i,i' \leq N\,.\] 
	For each $i \in \{1,\ldots,N\}$, the {\em star} of $\vec{x}_i$, denoted by $\textup{st}(\vec{x}_i,\mathcal{M}_{\Omega,h})$, is the set of tetrahedra $K \in \mathcal{M}_{\Omega,h}$ containing $\vec{x}_i$ as a vertex. We define a graph $\mathcal{G}(\vec{x}_i)$ with 
	\begin{itemize}
		\item {\bf Nodes:} The elements of $\textup{st}(\vec{x}_i,\mathcal{M}_{\Omega,h})$
		\item {\bf Edges:} The pairs $\{K,K'\} \subset \textup{st}(\vec x_i,\mathcal{M}_\Omega)$ such that $K\cap K' \in \mathcal{F}(\mathcal{M}_{\Omega,h}) \setminus \mathcal{M}_{\Gamma,h}$. 
	\end{itemize}
	We denote by $\Delta_{i,1}, \ldots, \Delta_{i,q_i}$ the connected components of $\mathcal{G}(\vec{x}_i)$. Each $\Delta_{i,j}$ is thus a group of tetrahedra that can be linked by face-connected paths avoiding the faces in $\mathcal{M}_{\Gamma}$. The connected components of $\mathcal{G}(\vec x_i)$ model the different connected sectors locally near $\vec x_i$ in $\R^3 \setminus \Gamma$. Define the sets 
	\begin{equation}
		\label{defCurlyHOmega}
		\mathcal{H}(\Omega \setminus \Gamma) \isdef \enstq{(i,j) \in \N^2}{1 \leq i \leq N\,,\,\, 1 \leq j \leq q_i}\,, \quad \textup{and }
	\end{equation}
	\begin{equation}
		\label{defCurlyHGamma}
		\mathcal{H}(\Gamma) \isdef \enstq{(i,j) \in \mathcal{H}(\Omega \setminus \Gamma)}{i \leq M}\,,
	\end{equation}
	For all $(i,j) \in \mathcal{H}(\Omega \setminus \Gamma)$, we define $\vec x_{i,j}$ as the pair $(\vec x_i, \Delta_{i,j})$, i.e. a vertex $\vec x_i$ ``labeled" by one of the connected components of $\mathcal{G}({\vec x}_i)$. We call $\{\vec x_{i,j}\}_{1 \leq j \leq q_i}$ the set of \emph{generalized vertices} attached to $\vec x_i$ (see also \cite[Definition 2.11]{averseng2022fractured}). For $(i,j) \in \mathcal{H}(\Omega \setminus \Gamma)$, let 
	\[\abs{\Delta_{i,j}} \isdef \bigcup_{K \in \Delta_{i,j}} \overline{K}\]
	and denote by $\phi_{i,j}$ the {\em split basis function} of $V_h(\Omega \setminus \Gamma)$ associated to the generalized vertex $\vec x_{i,j}$, which is defined by
	\[\phi_{i,j}(\vec{x}) \isdef \begin{cases}\phi_i( \vec x) & \textup{for } \vec x \in \textup{int}(\abs{\Delta_{i,j}})\,, \\
		0 & \textup{otherwise}.	
	\end{cases}\]
	Split basis functions span $V_h(\Omega \setminus \Gamma)$, as seen with the following result.
	\begin{lemma}[{See \cite[Lemma 4.1]{averseng2022stable}}]
		\label{basisVh}
		The split basis functions $\{\phi_{i,j}\}_{(i,j) \in \mathcal{H}(\Omega\setminus\Gamma)}$ form a basis of $V_h(\Omega \setminus \Gamma)$.
	\end{lemma}
	For $u_h \in V_h(\Omega \setminus \Gamma)$, we denote by $u_h(\vec x_{i,j})$ the coefficient of $u_h$ on the split basis function $\phi_{i,j}$, so that
	\begin{equation}
		\label{eq:defCoeffUhXij}
		\forall u_h \in V_h(\Omega \setminus \Gamma)\,, \quad \exists u_n(\vec x_{i,j}) \in \R\,: \qquad u_h = \sum_{(i,j) \in \mathcal{H}(\Omega\setminus\Gamma)} u_h(\vec x_{i,j}) \phi_{i,j}\,.
	\end{equation}
	We introduce the following coefficient-wise scalar product
	\begin{equation*}
	\forall (u_h,v_h) \in {V}_h(\Omega \setminus \Gamma) \times {V}_h(\Omega \setminus \Gamma)\,, \quad [u_h,v_h]_{l^2} \isdef \sum_{(i,j) \in \mathcal{H}(\Omega\setminus\Gamma)} u_h(\vec x_{i,j}) v_h(\vec x_{i,j})\,,
	\end{equation*}
	The space  $\Psi_h(\Omega \setminus \Gamma)$ is then defined as the $[\cdot,\cdot]_{l^2}$ orthogonal complement of $V_h(\Omega)$ in $V_h(\Omega \setminus \Gamma)$, so that
	\begin{equation}
		\label{eq:orthogonalityVhYh}
		V_h(\Omega \setminus \Gamma) = V_h(\Omega) \overset{\perp}{\oplus} \Psi_h(\Omega \setminus \Gamma)\,.
	\end{equation}
	Those definitions readily imply:
	\begin{corollary}[Parametrization of the jump space]
		The jump operator $[\cdot]_\Gamma$ induces a bijection
		\[[\cdot]_\Gamma: \Psi_h(\Omega \setminus \Gamma) \to \widetilde{V}_h([\Gamma])\,.\]
	\end{corollary} 
	Let
	\begin{equation}
		\label{defCurlyHtildeGamma}
		\widetilde{\mathcal{H}}([\Gamma])\isdef \enstq{(i,j) \in \mathcal{H}(\Omega\setminus\Gamma)}{q_i > 1 \textup{ and } j \leq q_i - 1}\,,
	\end{equation}
	and for $(i,j) \in \widetilde{\mathcal{H}}([\Gamma])$ define 
	\begin{equation}
		\label{defYij}
		\psi_{i,j} \isdef \phi_{i,j} - \phi_{i,q_i}\,.
	\end{equation}
	Using that $\{\phi_i\}_{1 \leq i \leq N}$ is a basis of $V_h(\Omega)$, together with the property
	\[\phi_{i} = \sum_{j = 1}^{q_i} \phi_{i,j}\,,\]
	and \Cref{basisVh}, a simple algebraic reasoning shows that $\{\psi_{i,j}\}_{(i,j) \in \widetilde{\mathcal{H}}([\Gamma])}$ is a basis of $\Psi_h(\Omega \setminus \Gamma)$. 
	Defining
	\begin{equation}
		\label{eq:defPhiTildeij}
		\forall (i,j) \in \widetilde{\mathcal{H}}([\Gamma])\,, \quad \widetilde{\phi}_{i,j} \isdef [\psi_{i,j}]_\Gamma
	\end{equation}
	we deduce the following result.
	\begin{corollary}[Basis of the jump space] 
		\label{lemparamJumps}
		The set $\{\widetilde{\phi}_{i,j}\}_{(i,j) \in \widetilde{\mathcal{H}}([\Gamma])}$ is a basis of $\widetilde{V}_h([\Gamma])$. 	
	\end{corollary}

	\subsection{Algorithm for the computation of the quantities in \eqref{requiredAlgo}}
	
	We first remark that computing  $a(\widetilde{\phi}_{i,j},\widetilde{\phi}_{i',j'})$, for $(i,j)$ and $(i',j')$ in $\widetilde{H}([\Gamma])$, only requires the evaluation of a linear combination of the quantities $a([\phi_{k,\ell}]_\Gamma,[\phi_{k',\ell'}]_\Gamma)$, for $(k,\ell),(k',\ell') \in \mathcal{H}(\Gamma)$. Those quantities are given by \Cref{thmWweak} in terms of the traces of $\phi_{k,\ell}$ and $\phi_{k',\ell'}$ on the boundaries $\partial \Omega_j$, but since those functions are defined in terms of connected components of the vertex stars of $\mathcal{M}_{\Omega,h}$, it is not a priori obvious how to perform those computations without relying on the external, tetrahedral mesh. To avoid this, the key idea is to introduce, for each $(k,\ell) \in \mathcal{H}(\Gamma)$, the set
	\begin{equation}
		\label{defAlphaij}
		\alpha_{k,\ell} \isdef \enstq{\vec {t} = (T,K) \in \mathcal{M}^*_{\Gamma,h}}{K \in \Delta_{k,\ell}}
	\end{equation}
	with $\mathcal{M}^*_{\Gamma,h}$ as in Definition \ref{defMstar}. Upon viewing the pairs $(T,K)$ as oriented triangles, the sets $\alpha_{k,\ell}$ can be computed without the need for the external mesh, using the so-called {\em intrinsic inflation algorithm} \cite[Def. 4.1]{averseng2022fractured}. Once $\mathcal{M}^*_{\Gamma,h}$ is endowed with a ``generalized mesh" structure, and if we assume that $\Gamma$ has no ``point contacts", those sets are immediately obtained from the {\em generalized vertices} of $\mathcal{M}^*_{\Gamma}$, computed via \cite[Algorithms 1 and 2]{averseng2022fractured}.\footnote{The (surface) generalized vertices of $\mathcal{M}^*_{\Gamma,h}$ should not be confused with the (volume) generalized vertices of the ``fractured mesh" $\mathcal{M}^*_{\Omega \setminus \Gamma,h}$, of which $\mathcal{M}^*_{\Gamma,h}$ is the boundary (see \cite[Section 4]{averseng2022fractured}). The generalized vertices $\vec x_{i,j}$ defined here  in \Cref{sec:localShapeFunc}, correspond to the (volume) generalized vertices of $\mathcal{M}^*_{\Omega \setminus \Gamma,h}$, while the $\alpha_{i,j}$ correspond to the (surface) generalized vertices of $\mathcal{M}^*_{\Gamma,h}$. Only the latter can be computed from the mesh $\mathcal{M}_{\Gamma,h}$ alone. For the definition of point contacts, see \cite[Section 5.2]{averseng2022fractured}. In the presence of point contacts, additional continuity conditions must be enforced.} 
	
	One has, by \Cref{corW}:
	\begin{equation}
		\label{eq:eval_a_phiij}
		a([\phi_{k,\ell}]_\Gamma,[\phi_{k',\ell'}]_\Gamma) = \sum_{\vec t,\vec t' \in \mathcal{M}^*_{\Gamma,h}} \iint_{\abs{\vec{t}}\times \abs{\vec t'}} \frac{\curl_{\vec t} \varphi_{k,\ell}^{\vec t}(x)\cdot\curl_{\vec t'} \varphi_{k',\ell'}^{\vec t'}(x') }{4\pi \norm{x - x'}} d\sigma_{\vec t}(x) d\sigma_{\vec t'}(x')\,,
	\end{equation}
	where $\{\varphi_{k,\ell}^{\vec t}\}_{\vec t \in \mathcal{M}^*_{\Gamma,h}}$ is defined by
	\begin{equation}
		\label{eq:defvarphikij}
		\forall \vec t \in \mathcal{M}^*_{\Gamma,h}\,, \,\,\forall x \in \abs{\vec t}\,, \quad \varphi_{k,\ell}^{\vec t}(x)\isdef 
		\begin{cases}
			\phi_k(x) & \textup{if } \vec t \in \alpha_{k,\ell}\,,\\
			0& \textup{otherwise}\,,
		\end{cases}
	\end{equation} 
	for all $\vec t \in \mathcal{M}^*_{\Gamma}$ and $(k,\ell) \in \mathcal{H}(\Gamma)$, where $\gamma_{\vec t}$ is defined in the paragraph below \Cref{defMstar}. Using \eqref{eq:defvarphikij}, the right-hand side of \eqref{eq:eval_a_phiij} can be evaluated without resorting to the external mesh $\mathcal{M}_{\Omega,h}$. The computation of $l_g(\widetilde{\varphi}_{i,j})$ is performed similarly, using that 
	\[l_g([\phi_{k,\ell}]_\Gamma) = \dduality{\gamma(\phi_{k,\ell})}{\pi_n(\vec g)} = \sum_{\vec t \in \mathcal{M}^*_{\Gamma,h}} \int_{\abs{\vec t}} \vec g(x) \cdot \vec n_{\vec t}(x) \varphi_{k,\ell}^{\vec t}(x)d\sigma_{\vec t}(x)\] 
	by definition of $l_g$, of the duality product $\dduality{\cdot}{\cdot}$ and integration by parts on each domain $\Omega_j$. The same ideas apply straightforwardly to $\DL \widetilde{\phi}_{i,j}$.

	\section{Induced decompositions of quotient spaces}	
	\label{sec:quotientSplit}
	
	We now recall a standard condition number estimate from the theory of additive Schwarz preconditioners involving the concept of ``stable subspace decompositions", and, based on ideas from \cite{hiptmair2012stable}, we identify two abstract conditions under which an initially stable splitting remains stable after passing to the quotient. 

	\subsection{Abstract condition number estimate for subspace splitting}
	
	Let $(\mathbb{V},\norm{\cdot}_{\mathbb{V}})$ be a Hilbert space, and let $V_{h,1}\,,\ldots\,,V_{h,n} \subset \mathbb{V}$ be finite dimensional subspaces of $\mathbb{V}$ and
	\[V_h \isdef V_{h,1} + \ldots + V_{h,n}\,.\]
	Let $P_i: V_{h} \to V_{h}$ be the $(\cdot,\cdot)_{\mathbb{V}}$-orthogonal projection onto $V_{h,i}$ and $P_{\rm ad} := \sum_{i=1}^n P_i$. Introduce the norm
	\[|||u_h|||_{\mathbb{V}}^2 := \inf \enstq{\sum_{i = 1}^n \norm{u_{h,i} }^2_{\mathbb{V}}}{\sum_{i = 1}^n u_{h,i} = u_h\,, \,\,\, u_{h,i} \in V_{h,i}}\,.\]
	\begin{theorem}[{cf. \cite[Thm. 16]{oswald1994multilevel}}]
		\label{thmCond}
		Suppose that there exists $\lambda_h, \Lambda_h > 0$ such that 
		\[\forall u_h \in V_h\,, \quad \lambda_h \norm{u_h}^2_{\mathbb{V}} \leq |||u_h|||_{\mathbb{V}}^2 \leq  \Lambda_h \norm{u_h}_{\mathbb{V}}^2\,.\]
		Then the spectral condition number of $P_{\rm ad}$ satisfies $\kappa(P_{\rm ad}) \leq \dfrac{\Lambda_h}{\lambda_h}$.
	\end{theorem}
	
	\subsection{Stability of induced quotient splitting}
	Suppose that $\mathbb{V}_0 \subset \mathbb{V}$ is a closed subspace and let $\mathbb{X}$ be the quotient space 
	$$\mathbb{X} = \mathbb{V} / \mathbb{V}_0\,.$$ 
	Let $T: \mathbb{V} \to \mathbb{X}$ be the canonical surjection associated to this quotient, and $\norm{\cdot}_\mathbb{X}$ the quotient norm
	\[\norm{x}_{\mathbb{X}} \isdef \inf_{v \in T^{-1}(x)} \norm{v}_{\mathbb{V}}\,.\] 
	Recall that this makes $\mathbb{X}$ a Hilbert space with the inner product 
	\[\inner{x}{x'}_{\mathbb{X}} := \inner{(\textup{Id}-P)v}{(\textup{Id}-P)v'}_{\mathbb{V}}\,,\]
	where $P: \mathbb{V} \to \mathbb{V}$ is the $\inner{\cdot}{\cdot}_{\mathbb{V}}$-orthogonal projection onto $\mathbb{V}_0$ and $v$ (resp. $v'$) is an arbitrary element of $T^{-1}(x)$ (resp. $T^{-1}(x')$). 
	With $X_h \isdef T(V_h)$, we write
	\[X_h = X_{h,1} + \ldots + X_{h,n}\,, \quad X_{h,i} \isdef T(V_{h,i})\,.\]
	\[\forall f_h \in X_h\,, \quad |||f_h|||_{\mathbb{X}}^2 := \inf \enstq{\sum_{i = 1}^n \norm{f_{h,i}}_{\mathbb{X}}^2}{\sum_{i = 1}^n f_{h,i} = f_h\,, \,\, f_{h,i} \in X_{h,i}}\,.\]
	Introduce the following assumptions\\
	
	\begin{itemize}
		\item[\textbf{(A)}] There exists an operator 
		\[\Pi_h: \mathbb{V} \to \mathbb{V}\,,\]
		which is a projection onto $V_h$ (i.e., satisfying $\Pi_h u_h = u_h$ for $u_h \in V_h$) and which {\em preserves} $\mathbb{V}_0$ (in the sense that if $v \in \mathbb{V}_0$, then $\Pi_h v \in \mathbb{V}_0$). Denote by $\norm{\Pi_h}_{\mathcal{L}(\mathbb{V})}$ its operator norm.\\[1em]
		\item[\textbf{(B)}] There exist constants $\kappa_1,\ldots,\kappa_n >0$ such that for all $u_{h,i} \in V_{h,i}$, 
		\begin{equation*}
			\min \enstq{\norm{u_{h,i} - u_{h,i,0}}_{\mathbb{V}}}{u_{i,h,0} \in \mathbb{V}_0 \cap V_{h,i}} \leq \kappa_i \min \enstq{\norm{u_{h,i} - u_{h,0}}_{\mathbb{V}}}{u_{h,0} \in \mathbb{V}_0 \cap V_{h}}\,.
		\end{equation*}
	\end{itemize}

	\begin{theorem}[Stability of the quotient splitting]
		\label{mainthm}
		Let $\alpha_h, A_h > 0$ be such that 
		\[\forall u_h \in V_h\,, \quad \alpha_h\norm{u_h}_{\mathbb{V}}^2 \leq |||u_h|||_{\mathbb{V}}^2 \leq A_h \norm{u_h}_{\mathbb{V}}^2\,,\]
		and assume \textup{\textbf{(A)}} and \textup{\textbf{(B)}}. Then,
		\begin{equation}
			\label{eq:estim_mainthm}
			\forall f_h \in X_h\,, \quad \beta_h\norm{f_h}_{\mathbb{X}}^2 \leq |||f_h|||_{\mathbb{X}}^2 \leq B_h \norm{f_h}_{\mathbb{X}}^2\\[1em]
		\end{equation}
		\[\textup{with } \,\, \beta_h \isdef \frac{\alpha_h}{(\max_{1 \leq i \leq n} \kappa_i)^2\norm{\Pi_h}^2_{\mathcal{L}(\mathbb{V})}}\,, \quad B_h \isdef  \norm{\Pi_h}_{\mathcal{L}(\mathbb{V})}^2 A_h\,.\]
	\end{theorem}
	We use a variant of the previous result in which the condition {\bf (B)} is relaxed for one of the subspaces (i.e., no estimate is required for one of the constants $\kappa_i$):
	\begin{lemma}[Weakening condition {\bf (B)}]
		\label{lem:weakening}
		With the same assumptions as in \Cref{mainthm}, the estimate \eqref{eq:estim_mainthm} also holds with 
		\[\beta_h = \frac{\alpha_h}{2 \max\{\norm{\Pi_h}_{\mathcal{L}(\mathbb{V})},\kappa_2,\ldots,\kappa_n\}^2 \norm{\Pi_h}^2_{\mathcal{L}(\mathbb{V})}}\,.\]
	\end{lemma}
The proofs are given in \Cref{sec:proofHiptmairMao}.

\section{Stable splitting of boundary element spaces on multiscreens}
	\label{sec:dd}
	
We now construct the splitting used to define our additive Schwarz operator $P_{\rm ad}(H;h)$ introduced in Section \ref{sec:main}. We start by defining spaces of so-called {\em discrete-harmonics} (see e.g. \cite[Section 4.4]{toselli2004domain}). We then define a splitting of $V_h(\Omega \setminus \Gamma)$, and deduce a splitting of $\widetilde{V}_h([\Gamma])$ by application of the jump operator, for which we estimate the stability constants to prove \Cref{thm:main}.
	
\subsection{Coarse mesh}

From now on, we assume that the sets $\Omega_j$ defined in \Cref{sec:polyMS} are (coarse) tetrahedra, providing a quasi-uniform and shape-regular coarse triangulation of $\Omega$ (and in turn, of $\Gamma$), with diameters bounded by a coarse mesh parameter $H$, such that $0 < h < H < C$ where $C > 0$ is a constant depending only on $\Gamma$. In the following analysis, we also assume that for each $h > 0$, the mesh $\mathcal{M}_{\Omega,h}$ is uniformly shape-regular and quasi-uniform, with elements of diameter uniformly comparable to $h$.

\subsection{Discrete harmonic functions in the volume}

For each $j$, let us introduce  the subspace of discrete functions that are localised
in the subdomain $\Omega_j$ and vanishes beyond the boundary of $\Omega_j$, which we denote
\begin{equation*}
	V_{h,0}(\Omega_j)\isdef\enstq{u \in V_h(\Omega\setminus\Gamma)}{ u = 0 \,\, \textup{on}\,\, \Omega\setminus\Omega_j}.
\end{equation*}
For $j=1,\dots, J$, this forms a collection of subspaces of $V_h(\Omega\setminus\Gamma)$ that are pairwise orthogonal
in the $H^1(\Omega \setminus \Gamma)$-scalar product. Then we can define $\mathbf{V}_h(\Omega \setminus \Gamma)$
as the orthogonal complement to $V_{h,0}(\Omega_1)\oplus\dots\oplus V_{h,0}(\Omega_J)$ with respect to this scalar product.
As a consequence we have
\begin{equation}\label{OrthogonalDecomposition}
	V_h(\Omega\setminus\Gamma) = \mathbf{V}_h(\Omega \setminus \Gamma)\oplus V_{h,0}(\Omega_1)\oplus\dots\oplus V_{h,0}(\Omega_J)
\end{equation}
and this sum is $H^1(\Omega \setminus \Gamma)$-orthogonal by
construction. In words, $\mathbf{V}_h(\Omega \setminus \Gamma)$ is the
set of elements of $V_h(\Omega \setminus \Gamma)$ which are discrete
harmonics in each $\Omega_j$. The space $\mathbf{V}_h(\Omega
\setminus \Gamma)$ is {\bf not} the $H^1(\Omega \setminus
\Gamma)$-orthogonal complement of $V_h(\Omega \setminus \Gamma) \cap
H^1_{0,\Gamma}(\Omega)$ in $V_h(\Omega \setminus \Gamma)$, i.e. it
is not a set of global discrete harmonics. The motivation for
choosing piecewise harmonics instead of global harmonics is to make
it easier to define a decomposition satisfying an explicit
strengthened Cauchy-Schwarz inequalities in the additive Schwarz
framework. 

%

\begin{definition}[Basis of the discrete harmonic space]
	\label{def:omegaij}
	Let 
	\begin{equation}
		\label{defcurlyHsigma}
		\mathcal{H}(\Sigma) \isdef \enstq{(i,j) \in \mathcal{H}(\Omega \setminus \Gamma)}{\vec x_i \in \Sigma}\,,
	\end{equation}
	where $\mathcal{H}(\Omega \setminus \Gamma)$ is defined in eq.~\eqref{defCurlyHOmega} and $\Sigma = \partial \Omega_0\cap \ldots \cap \partial \Omega_J$ is the skeleton of the Lipschitz partition.
	For $(i,j) \in \mathcal{H}(\Sigma)$ and $k \in \{1\,\ldots,J\}$, we define $\varepsilon_{i,j}^k \in V_{h,0}(\Omega_k)$ by the variational problem 
	\[\int_{\Omega_k}\varepsilon_{i,j}^k(x) v_h(x) + \nabla \varepsilon_{i,j}^k(x) \cdot \nabla v_h(x) dx= \int_{\Omega_k} \phi_{i,j} (x)v_h(x) + \nabla \phi_{i,j}(x) \cdot \nabla v_h(x)dx\,,\]
	for all $v_h \in V_{h,0}(\Omega_k)$. Let 
	\begin{equation}
		\label{eq:defomegaij}
		\omega_{i,j} \isdef \phi_{i,j} - \sum_{k = 1}^J \varepsilon_{i,j}^k\,.
	\end{equation}
\end{definition}
This definition readily implies that $\omega_{i,j} \in \mathbf{V}_h(\Omega \setminus \Gamma)$ and satisfies $\gamma(\omega_{i,j}) = \gamma(\phi_{i,j})$.

\begin{lemma}
	\label{lemdof}
	The family $\{\omega_{i,j}\}_{(i,j) \in \mathcal{H}(\Sigma)}$ is a basis of $\mathbf{V}_h(\Omega \setminus \Gamma)$ and 
	\[\forall u_h \in \mathbf{V}_h(\Omega \setminus \Gamma)\,, \quad u_h = \sum_{(i,j) \in \mathcal{H}(\Sigma)} u_h(\vec x_{i,j}) \omega_{i,j}\,.\] 
\end{lemma}
\begin{proof}
	Assume that there exist coefficients $\{\lambda_{i,j}\}_{(i,j) \in \mathcal{H}(\Sigma)}$ such that 
	\begin{equation}
		\label{linearComb}
		\sum_{(i,j) \in \mathcal{H}(\Sigma)} \lambda_{i,j} \omega_{i,j} = 0\,.
	\end{equation}
	Let $(i_0,j_0) \in \mathcal{H}(\Sigma)$ and, given $K \in \Delta_{i_0,j_0}$, consider a sequence $(y_n)_{n \in \N}$ of points in $\textup{int}(\abs{K})$ converging to $\vec x_{i_0}$. Note that since $\epsilon_{i,j}^k \in V_{h,0}(\Omega_k)$, $\omega_{i,j}$ differs from $\phi_{i,j}$ only by a continuous function which vanishes at every node of $\Sigma$. Therefore, one has 
	\[\lim_{n \to \infty} \omega_{i,j}(y_n) = \lim_{n\to \infty} \phi_{i,j}(y_n)\,.\]
	Furthermore, one can check (see e.g. the proof of \cite[Lemma 4.1]{averseng2022stable}) that 
	\[\lim_{n \to \infty} \phi_{i,j}(y_n) = \delta_{i,i_0} \delta_{j,j_0}\,.\]
	Combining this with eq.~\eqref{linearComb}, 
	it follows that $\lambda_{i_0,j_0} = 0$. Hence $\{\omega_{i,j}\}_{(i,j) \in \mathcal{H}(\Sigma)}$ is a free family. 
	
	On the other hand, let $u_h \in \mathbf{V}_h(\Omega \setminus \Gamma)$ and set
	$v_h = \sum_{(i,j) \in \mathcal{H}(\Sigma)} u_h(\vec x_{i,j}) \omega_{i,j}$.
	Each $\omega_{i,j}$ belongs to $\mathbf{V}_h(\Omega \setminus \Gamma)$ so we conclude that
	$u_h-v_h\in \mathbf{V}_h(\Omega \setminus \Gamma)$. Next, by definition of $\omega_{i,j}$, we have 
	\begin{equation*}
		u_h - v_h = \sum_{(i,j) \in \mathcal{H}(\Sigma)} u_h(\vec x_{i,j})\sum_{k = 1}^J \varepsilon_{i,j}^k\in
		V_{h,0}(\Omega_1)\oplus\dots\oplus V_{h,0}(\Omega_J).
	\end{equation*}
	Since $( V_{h,0}(\Omega_1)\oplus\dots\oplus V_{h,0}(\Omega_J) )\cap \mathbf{V}_h(\Omega \setminus \Gamma) = \{0\}$
	according to \eqref{OrthogonalDecomposition}, we deduce that $u_h-v_h = 0$, which concludes the proof. 
\end{proof}

The main property of discrete harmonics is that their $H^1$ norm can be estimated by a suitable norm of their boundary values, (this is classical, see also see e.g. \cite[Lemma 4.10]{toselli2004domain}). For a Lipschitz domain $\mathcal{D}$, let $\gamma_{\mathcal{D}}: H^1(\mathcal{D}) \to L^2(\partial \mathcal{D})$ be the pointwise trace operator, and let $H^{1/2}(\partial \mathcal{D}) \subset L^2(\partial \mathcal{D})$ be the image of $H^1(\mathcal{D})$ by the trace operator $\gamma_{\mathcal{D}}$. 
\begin{definition}[Quotient $H^{1/2}$ semi-norm]
	\label{defH12norms}
	We define the $H^{1/2}$ semi-norm as
	\[\abs{u}_{H^{1/2}(\partial \mathcal{D})} \isdef \min_{\gamma_{\mathcal{D}} U = u} \norm{\nabla u}_{L^2(\mathcal{D})}\,.\]
\end{definition}
Other equivalent norms are often used for this space, see e.g. \cite[Chap. 3]{McLean}, but this definition is convenient here because of scale-invariance results from \cite{pechstein2013shape}, which play a role in the proof of \Cref{lem:satB} below.

\subsection{Stable volume splitting}

We now define a subspace splitting of $V_h(\Omega \setminus \Gamma)$ which relies on a partition of the generalized vertices $\vec x_{i,j}$ with $(i,j) \in \mathcal{H}(\Sigma)$, according to the coarse mesh $\Omega_1,\ldots,\Omega_J$. After removing the interior degrees of freedom (i.e. the spaces $V_{h,0}(\Omega_j)$), it remains to decompose the space $\mathbf{V}_h(\Omega \setminus \Gamma)$. 
The goal is to construct a decomposition of the type of \cite[Algorithm 5.5]{toselli2004domain}, but the difference is that here, we need to account for the jumps of the functions in $\mathbf{V}_h(\Omega \setminus \Gamma)$ across $\Gamma$, due to the different ``sides"  of the faces, edges and vertices located on $\Gamma$.  

\begin{definition}[Subspace generated by an index subset of $\mathcal{H}(\Sigma)$]
	\label{def:SpaceCorresponding}
	Given a subset $\mathcal{H}_k \subset \mathcal{H}(\Sigma)$, the subspace $\mathbf{V}_h(\mathcal{H}_k) \subset \mathbf{V}_h(\Omega \setminus \Gamma)$ {\em generated by} $\mathcal{H}_k$ is defined by 
	\[\mathbf{V}_h(\mathcal{H}_k) \isdef \textup{Span}(\{\omega_{i,j}\}_{(i,j) \in \mathcal{H}_k})\,, \quad \omega_{i,j} \text{ defined as in \eqref{eq:defomegaij}}\,.\]
\end{definition}
We denote by $\mathcal{F}^k$ a face of the coarse triangulation (i.e., a triangular face of one of the tetrahedra $\Omega_k$), and first assume that $\mathcal{F}^k \notin \mathcal{M}_{\Gamma}$. In this case, define
$\mathcal{F}^k_h$ as the set of pairs $(i,j)$ such that $\vec x_i$ belongs to the relative interior of $\mathcal{F}^k$ (note that, since $\mathcal{F}^k$ is not in $\Gamma$, $\mathcal{F}^k_h$ only contains pairs of the form $(i,1)$) and let $\mathbf{V}_{\mathcal{F}^k} \isdef \mathbf{V}_h(\mathcal{F}^k_h)$. 

On the other hand, if the face $\mathcal{F}^k$ shared by the coarse tetrahedra $\Omega_{\ell}$ and $\Omega_m$ belongs to $\mathcal{M}_\Gamma$, then for each vertex $\vec x_i \in \textup{int}(\mathcal{F}^k)$, there are two generalized vertices $\vec x_{i,1}$ and $\vec x_{i,2}$ associated to $\vec x_i$, corresponding to either $\Omega_\ell$ or $\Omega_m$. We then define two spaces $\mathbf{V}_{\mathcal{F}^k,\nu} = \mathbf{V}_h(\mathcal{F}^k_{h,\nu})$, $\nu = 1,2$, where 
\[\mathcal{F}^k_{h,1} \isdef\enstq{(i,j) \in \mathcal{H}(\Omega \setminus \Gamma)}{\vec x_{i} \in \textup{int}(\mathcal{F}^k)\,,\,\, \Omega_\ell \cap \abs{\Delta_{i,j}} \neq \emptyset}\,,\]
\[\mathcal{F}^k_{h,2} \isdef\enstq{(i,j) \in \mathcal{H}(\Omega \setminus \Gamma)}{\vec x_{i} \in \textup{int}(\mathcal{F}^k)\,,\,\, \Omega_m \cap \abs{\Delta_{i,j}} \neq \emptyset}\,.\]

The set of remaining pairs $(i,j)$, i.e. those such that $\vec x_i$ belongs to the boundary of a coarse face $\mathcal{F}^k$, is denoted by $\mathcal{W}_h$, and we define the {\em wire-basket space} $\mathbf{V}_{\mathcal{W}} \isdef \mathbf{V}_h(\mathcal{W}_h)$. The proposed splitting of $V_h(\Omega \setminus \Gamma)$ is then 
\begin{equation}
	\label{eq:proposedSplit}
	V_{h}(\Omega \setminus \Gamma) = \sum_{j=1}^J V_{h,0}(\Omega_j) + \sum_{\mathcal{F}^k \cap \Gamma = \emptyset }\mathbf{V}_{\mathcal{F}^k} + \sum_{\mathcal{F}^k \subset \Gamma} \sum_{\nu = 1}^2 \mathbf{V}_{\mathcal{F}^k,\nu} + \mathbf{V}_{\mathcal{W}} + V_H(\Omega \setminus \Gamma)\,,
\end{equation}
where $V_H(\Omega \setminus \Gamma)$ is the ``coarse space" of the decomposition, which is the set of elements of $V_h(\Omega \setminus \Gamma)$ whose restriction to each $\Omega_j$ is affine. For notational convenience, we label the subsets of the splitting as ${V}_0,\ldots,{V}_L$ with ${V}_0 = {V}_H(\Omega \setminus \Gamma)$, ${V}_1 = \mathbf{V}_{\mathcal{W}}$, and ${V}_2,\ldots,{V}_L$ equal to the remaining spaces, in some arbitrary order. For $u_h \in V_h(\Omega \setminus \Gamma)$, define
\begin{equation}
	\label{eq:defNormSplitVol}
	|||u_h|||_{\rm split}^2 \isdef \inf \enstq{\sum_{i = 0}^L \norm{u_{h,i}}^2_{H^1(\R^3\setminus \Gamma)}}{\sum_{i = 0}^L u_{h,i} = u_h\,,\,\, u_{h,i} \in {V}_{i}}\,.
\end{equation}

\begin{theorem}
	\label{thm:splitVol1}
	The splitting in eq.~\eqref{eq:proposedSplit} is $(\alpha_h,A_h)$-stable with respect to the $H^1(\R^3 \setminus \Gamma)$ norm, in the sense that 
	\begin{equation}
		\label{eq:stableSplitVol}
		\forall u_h \in V_h(\Omega \setminus \Gamma)\,, \quad \alpha_h \norm{u_h}^2_{H^1(\R^3 \setminus \Gamma)} \leq |||u_h|||^2_{\rm split} \leq A_h \norm{u_h}^2_{H^1(\R^3 \setminus \Gamma)}\,,
	\end{equation} 
	where, 
	\[\alpha_h \geq c(1 + \log H/h)^{-2}\,, \quad A_h \leq C \,.\]
	with $c,C>0$ independent of $h$ and $H$. 
\end{theorem}
The proof is given in \Cref{sec:proofMain}.

\subsection{Stable splitting of the jump space}

We now define a splitting of $\widetilde{V}_h([\Gamma])$ by applying the operator $[\cdot]_{\Gamma}$ on both sides of eq.~\eqref{eq:proposedSplit}. Note that $[V_{h,0}(\Omega_j)] =\{0\}$ for all $j = 1,\ldots,J$. Moreover,
\[[V_H(\Omega \setminus \Gamma)]_\Gamma = \widetilde{V}_H([\Gamma])\,,\]
where $\widetilde{V}_H([\Gamma])$ is defined just as $\widetilde{V}_h([\Gamma])$ but using the coarse mesh instead of the fine mesh. 

\begin{definition}[Proposed splitting of the jump space]
	\label{def:proposedJumpSplit}
	We define the additive Schwarz operator $P_{\rm ad}(H;h)$ according to the splitting
	\begin{equation}
		\label{eq:jumpSplit}
		\boxed{\widetilde{V}_h([\Gamma]) = \sum_{\mathcal{F}^k \subset \Gamma} \sum_{\nu = 1}^2 \widetilde{V}_{\mathcal{F}^k,\nu} +  \widetilde{V}_{\mathcal{W}} + \widetilde{V}_H([\Gamma])\,,}
		\end{equation}
where $\widetilde{V}_{\mathcal{F}^k,\nu} \isdef [\mathbf{V}_{\mathcal{F}^k,\nu}]_\Gamma$ and $\widetilde{V}_{\mathcal{W}} := [\mathbf{V}_{\mathcal{W}}]_\Gamma$.
\end{definition}
\begin{remark}
	We point out that in \eqref{eq:jumpSplit}, we have $\widetilde{{V}}_{\mathcal{F}^k,1} = \widetilde{V}_{\mathcal{F}^k,2}$. Therefore, one of the two copies can be removed from the decomposition, without worsening the final stability result. Indeed, one can see that if \Cref{thmCond} holds for the full decomposition, then it holds for the decomposition with just one copy of each face jump spaces with the same $\lambda_h$, and with $\Lambda'_h = 2\Lambda_h$. 	
\end{remark}

\subsection{Proof of \Cref{thm:main}}
 	We now complete the proof of \Cref{thm:main}. We bound the condition number $\kappa(P_{\rm ad})$ using \Cref{thmCond}, where the stability constants $\lambda_h,\Lambda_h$ are estimated using the concept of induced splittings discussed in \Cref{sec:quotientSplit}. The initial splitting is given by eq.~\eqref{eq:proposedSplit}, with stability constants given by \Cref{thm:splitVol1}, and the operator $T$ mapping this initial splitting to the one given in eq.~\eqref{eq:proposedSplit} is given by the jump operator $[\cdot]_\Gamma$. According to \Cref{thm:splitVol1}, it remains to check that the stability conditions of \Cref{lem:weakening} hold. In this context, they take the following form 
	\begin{itemize}
		\item[\textbf{(A)}] There exists an $h$-uniformly bounded projection $\Pi_h: H^1(\R^3 \setminus \Gamma) \to H^1(\R^3 \setminus \Gamma)$ onto  ${V}_h(\Omega \setminus \Gamma)$ preserving $H^1(\R^3)$. 
		\item[\textbf{(B)}] For all but one subspace ${V}_n$ in the decomposition \eqref{eq:proposedSplit}, there exists $\kappa_n >0$ independent of $h$ such that 
		\[\forall u \in {V}_n\,, \quad \min_{v_{0,n} \in {V}_n \cap H^1(\R^3)} \norm{u - v_{0,n}}_{H^1(\R^3 \setminus \Gamma)} \leq \kappa_n \min_{v_0 \in {V}_h(\Omega \setminus \Gamma) \cap H^1(\R^3)} \norm{u - v_0}_{H^1(\R^3 \setminus \Gamma)}\,.\]
	\end{itemize}
	By \Cref{thmQI} below, {\bf (A)} holds with $\norm{\Pi_h}_{\mathcal{L}(H^1(\R^3 \setminus \Gamma)} = O(1)$ (with respect to the parameters $h$ and $H$). On the other hand, by \Cref{lem:satB} below, {\bf (B)} holds with ${\kappa_n} = O(1)$ for each subspace $V_n$, with the exception of the wire-basket (this space being the exception permitted by \Cref{lem:weakening}).~$\square$

\paragraph{Proof of {\bf (A)}}

We construct the operator $\Pi_h$ by combining two quasi-interpolants. The first one is (up to minor modifications) given by the classical Scott-Zhang quasi-interpolant acting on functions in the volume. The second one is the analog of that operator but acting on multi-traces i.e. on the surface of the multi-screen rather than the volume. Note that the extra property that the operator $\Pi_h$ of \Cref{thmQI} has compared to the classical Scott-Zhang interpolant $\Psi_h$ of \Cref{SZ} is that it preserves a larger space (namely $H^1(\R^3)$, instead of just $H^1_{0,\Gamma}(\R^3)$).

\begin{proposition}[Scott-Zhang quasi-interpolant on $H^1(\R^3 \setminus \Gamma)$]
	\label{SZ}
	There exists a constant $C > 0$ such that, for each $h > 0$, there exists a projection $\mathscr{Z}_h: H^1(\R^3 \setminus \Gamma) \to H^1(\R^3 \setminus \Gamma)$ onto ${V}_h(\Omega \setminus \Gamma)$ which preserves $H^1_{0,\Gamma}(\R^3)$ and satisfies
	\[\norm{\mathscr{Z}_h u}_{H^1(\R^3 \setminus \Gamma)} \leq C \norm{u}_{H^1(\R^3 \setminus \Gamma)} \quad \forall u \in H^1(\R^3 \setminus \Gamma)\,.\]
\end{proposition}
One can construct $\mathscr{Z}_h$ as in \cite{scott1990finite}. The analysis extends with only minor adaptations to deal with the more complex domain $\R^3 \setminus \Gamma$. Note that, by applying to $\mathscr{Z}_h$ the same reasoning as to $\Pi_h$ in the proof of \Cref{thmQI} below, one can show that the combination of the projection property and the stability of the space $H^1_{0,\Gamma}(\R^3)$ implies that $\mathscr{Z}_h$ preserves piecewise linear point traces, in other words, $\gamma \mathscr{Z}_h u = \gamma u$ if $\gamma u \in \mathbb{V}_h(\Gamma)$.
\begin{proposition}[{Jump aware quasi-interpolant on $\mathbb{H}^{1/2}(\Gamma)$}]
	There exists a constant $C > 0$ such that, for each $h > 0$, there exists a projection $\Phi_h: \mathbb{H}^{1/2}(\Gamma) \to \mathbb{H}^{1/2}(\Gamma)$ onto $\mathbb{V}_h(\Gamma)$ which preserves $H^{1/2}([\Gamma])$ and satisfies
	\[\norm{\Phi_h u}_{\mathbb{H}^{1/2}(\Gamma)} \leq C \norm{u}_{\mathbb{H}^{1/2}(\Gamma)} \quad \forall u \in \mathbb{H}^{1/2}(\Gamma)\,.\]
\end{proposition}
The construction can be found in \cite{averseng2022stable}.

\begin{corollary}[Condition {\bf (A)}]
	\label{thmQI}
	There exists a constant $C > 0$ such that, for each $h > 0$, there exists a projection $\Pi_h: H^1(\R^3 \setminus \Gamma) \to H^1(\R^3 \setminus \Gamma)$ onto  ${V}_h(\Omega \setminus \Gamma)$ which preserves $H^1(\R^3)$ and satisfies
	\[\norm{\Pi_h u}_{H^1(\R^3 \setminus \Gamma)} \leq C \norm{u}_{H^1(\R^3 \setminus \Gamma)} \quad \forall u \in H^1(\R^3 \setminus \Gamma)\,.\]
\end{corollary}
\begin{proof}[Proof of \Cref{thmQI}]
	Given $u \in H^1(\R^3 \setminus \Gamma)$, we define $\Pi_h u$ as follows. Let $v_h = \Phi_h (\gamma u)$, and let $V$ be the harmonic lifting of $v_h$, i.e. the element of $H^1(\R^3 \setminus \Gamma)$ with minimal $H^1(\R^3 \setminus \Gamma)$ norm such that $\gamma V = v_h$. Finally, let 
	\[\Pi_h u \isdef \mathscr{Z}_h (V)\,.\]
	It is clear that $\Pi_h$ satisfies the required stability, since all operations used to define it are continuous uniformly in $h$.  To prove that it is a projection, it is convenient to write
	\[\Pi_h = \mathscr{Z}_h \circ \mathcal{L} \circ \Phi_h \circ \gamma\]
	where $\mathcal{L}: \mathbb{H}^{1/2}(\Gamma) \to H^1(\R^3 \setminus \Gamma)$ is the harmonic lifting operator. We claim that
	\begin{equation}
		\label{claimProj}
		\gamma \circ \mathscr{Z}_h \circ \mathcal{L} \circ \Phi_h = \Phi_h
	\end{equation}
	If this holds, then one deduces easily that $\Pi_h$ is a projection writing
	\begin{equation*}
		\Pi_h^2 = \mathscr{Z}_h \circ \mathcal{L} \circ \Phi_h \circ (	\gamma \circ \mathscr{Z}_h \circ \mathcal{L} \circ \Phi_h)  \circ \gamma
		 = \mathscr{Z}_h \circ \mathcal{L} \circ \Phi_h \circ \Phi_h  \circ \gamma
		 = \Pi_h\,,
	\end{equation*}
	since $\Phi_h$ is a projection. To prove eq.~\eqref{claimProj}, we fix $\varphi \in \mathbb{H}^{1/2}(\Gamma)$ and let $\varphi_h = \Phi_h(\varphi)$. Then, by \Cref{lem:StabDisc}, one has 
	\[\gamma (\mathscr{Z}_h \mathcal{L} (\varphi_h)) = \varphi_h\,.\]
	Since this holds for all $\varphi \in \mathbb{H}^{1/2}(\Gamma)$, the claim in eq.~\eqref{claimProj} follows.

	Finally, let us show that $\Pi_h$ preserves $H^1(\R^3)$. Assume that $u \in H^1(\R^3)$, let $v_h = \Phi_h(\gamma u)$ and $V = \mathcal{L}v_h$ the harmonic lifting of $v_h$ as before. Note that $v_h \in H^{1/2}([\Gamma])$ since $\Phi_h$ preserves this space and $\gamma u \in H^{1/2}([\Gamma])$. Let $w_h \in V_h(\Omega \setminus \Gamma)$ be such that $\gamma w_h = v_h$ (the existence of $w_h$ is guaranteed by the fact that $v_h \in \mathbb{V}_h(\Gamma) = \gamma (V_h(\Omega \setminus \Gamma))$. Then we have 
	\[\gamma(V - w_h) = v_h - v_h = 0 \,,\]
	i.e. $V - w_h \in H^1_{0,\Gamma}(\R^3)$. Therefore, since $\mathscr{Z}_h$ preserves $H^1_{0,\Gamma}(\R^3)$, we conclude
	\[ 0 = \gamma(\mathscr{Z}_h (V - w_h)) = \gamma \Pi_h u - \gamma \mathscr{Z}_h w_h = \gamma\Pi_h u - \gamma w_h = \gamma \Pi_h u - v_h\,,\]
	recalling that $\mathscr{Z}_h V = \mathscr{Z}_h \mathcal{L} \Phi_h \gamma  u = \Pi_h u$ by definition of $\Pi_h$, and using that $\mathscr{Z}_h w_h = w_h$. In other words, $\gamma \Pi_h u = v_h \in H^{1/2}([\Gamma])$, i.e. $\Pi_h u \in H^1(\R^3)$. This concludes the proof.
\end{proof}

\paragraph{Proof of {\bf (B)}} We show that the condition {\bf (B)} holds in the following lemma.
\begin{lemma}[Condition {\bf (B)}]
	\label{lem:satB}
	The condition $\textup{\textbf{(B)}}$ is satisfied by every space $V_n$ except the wire-basket space $\mathbf{V}_{\mathcal{W}}$ in the splitting \eqref{eq:proposedSplit}, with a constant $\kappa_n = O(1)$.
\end{lemma}
\begin{proof}
	The condition is vacuous for the subspaces $\mathbf{V}_{\mathcal{F}^k}$ with $\mathcal{F}^k \cap \Gamma = \emptyset$, because in this case, $\mathbf{V}_{\mathcal{F}^k} \subset H^1(\R^3)$, hence the left-hand side of the inequality is $0$. For the same reason, the condition is satisfied for the spaces $V_{h,0}(\Omega_j)$. 
	
	\begin{itemize}
		\item For the space $V_H(\Omega \setminus \Gamma)$, we have by  \Cref{lemDiscreteTraceNorm} and \Cref{thmQI} 
		\[\min_{v_{H,0} \in V_H(\Omega \setminus \Gamma) \cap H^1(\Omega)} \norm{u - v_{H,0}}_{H^1(\R^3 \setminus \Gamma)} \leq \norm{\Pi_H}_{\mathcal{L}(H^1(\Omega \setminus \Gamma))} \norm{[u]_\Gamma}_{\widetilde{H}^{1/2}([\Gamma])} \quad \forall u \in \mathbb{H}^{1/2}(\Gamma)\,.\]
		This implies condition \textbf{(B)} for this space by the quotient definition of the $\widetilde{H}^{1/2}([\Gamma])$ norm. 
		\item If $\mathcal{F}^k \subset \Gamma$, then we have $\mathbf{V}_{\mathcal{F}^k,\nu} \cap H^1(\R^3) = \{0\},$ for $\nu = 1,2$, so it suffices to show
		\[\norm{u_h}_{H^1(\R^3 \setminus \Gamma)} \leq C \norm{[u_h]_\Gamma}_{\widetilde{H}^{1/2}([\Gamma])}\,, \quad \forall u_h \in \mathbf{V}_{\mathcal{F}^k,\nu}\,. \]
		Let $u_h \in \mathbf{V}_{\mathcal{F}^k,\nu}$ and $\varphi_h = \gamma_\ell u_h$, where $\ell$ is the index such that $\textup{supp}\,u_h \subset \Omega_\ell$. Let $U \in H^1(\Omega_\ell)$ be the unique element of $H^1(\Omega_\ell)$ with $\gamma_\ell U = \varphi_h$ and 
		\[\norm{\nabla U}^2_{L^2(\Omega_\ell)} = \abs{\varphi_h}^2_{H^{1/2}(\partial \Omega_\ell)}\,,\]       
		in view of \Cref{defH12norms}. With $\mathscr{Z}_\ell$ a Scott-Zhang interpolant onto $V_h(\Omega_\ell)$, let $U_h = \mathscr{Z}_\ell U$. One has, on the one hand,
		\[\norm{U_h}^2_{H^1(\Omega_\ell)} \leq C \norm{U}^2_{H^1(\Omega_\ell)}\,,\]
		by the mapping properties of $\mathscr{Z}_\ell$, and, on the other hand,
		\[\norm{u_h}^2_{H^1(\R^3 \setminus \Gamma)}  = \norm{u_h}^2_{H^1(\Omega_\ell)} \leq \norm{U_h}^2_{H^1(\Omega_\ell)}\,.\] 
		Here we used the minimizing property of discrete harmonics, and the fact that $\gamma_\ell(U_h) = \varphi_h$, since $\mathscr{Z}_\ell$ preserves piecewise linear boundary values. Since $\varphi_h$ vanishes on the faces of $\Omega_\ell$ distinct from $\mathcal{F}^k$, we can find a face $\mathcal{F}$ of $\Omega_\ell$ such that $U_h$ vanishes on $\mathcal{F}$, hence, by \cite[Lemma 2.57]{pechstein2013finite},
		\[\norm{U}^2_{L^2(\Omega_\ell)} \leq C H^2 \norm{\nabla U}^2_{L^2(\Omega_\ell)} \leq C \norm{\nabla U}^2_{L^2(\Omega_\ell)}\,,\]
		since $H$ is bounded. Combining these estimates, we arrive at 
		\[\norm{u_h}^2_{H^1(\R^3 \setminus \Gamma)} \leq C \abs{\varphi_h}^2_{H^{1/2}(\Gamma)}\,.\]
		By \cite[Lemma 6.5]{pechstein2013shape}, there holds
		\[\abs{\varphi_h}^2_{H^{1/2}(\partial \Omega_\ell)}\leq c_\ell (W_{\ell} \varphi_h,\varphi_h)_{L^2(\partial \Omega_\ell)}\]
		where $W_{\ell}$ is the classical hypersingular operator on $\partial \Omega_\ell$, and where the constant $c_\ell$ is uniformly bounded, due the shape-regularity of the coarse mesh. Since $u_h$ vanishes on all $\Omega_m$ for $m \neq \ell$, by \Cref{thmWweak}, we have 
		\[a([u_h]_\Gamma,[u_h]_\Gamma) = \iint_{\partial \Omega_\ell \times \partial \Omega_\ell} \frac{\curl_\ell \varphi_h(x) \cdot \curl_\ell \varphi_h(y) }{4\pi \norm{x - y}} = (W_\ell \varphi_h,\varphi_h)_{L^2(\partial \Omega_\ell)}\,.\]
		Therefore, 
		\[\norm{u_h}^2_{H^1(\R^3 \setminus \Gamma)} \leq C a([u_h]_\Gamma,[u_h]_\Gamma) \leq C_W \norm{[u_h]_\Gamma}^2_{\widetilde{H}^{1/2}([\Gamma])}\,,\]
		which implies condition {\bf (B)} for the face space $\mathbf{V}_{\mathcal{F}^k,\nu}$. 
	\end{itemize}
\end{proof}

	\bibliographystyle{plain}
	\bibliography{biblio.bib}

	\appendix

	\section{Proof of \Cref{thmBVP}}
	\label{sec:appPoincare}
\begin{proof}
	The existence and uniqueness of $\varphi^*$ follows from the Riesz representation theorem since $a(\cdot,\cdot)$ defines a scalar product on $\widetilde{H}^{1/2}(\Gamma)$. In turn, the properties of the double-layer potential stated below \Cref{def:DL} imply that $\mathcal{U} = \DL \varphi^*$ satisfies $\Delta \mathcal{U} = 0$ on $\R^3 \setminus \Gamma$ and $\mathcal{U} = O\left(\frac{1}{\norm{x}}\right)$. Moreover, by definition of the bilinear form $a(\cdot,\cdot)$ and using the jump relation $[\DL \varphi^*] = [\varphi^*]$,  
	\[a([v]_\Gamma,\varphi^*) = \dduality{v}{\pi_n(\nabla \mathcal{U})}\qquad \forall v \in \mathbb{H}^{1/2}(\Gamma)\,.\]
	Notice that for all $v \in \mathbb{H}^{1/2}(\Gamma)$, one has $a(\varphi^*,[v]_\Gamma) = l_g([v]_\Gamma) = \dduality{v}{\mathbf{g}}$ by definition of $\varphi^*$ and $l_g$ and using that $a(\cdot,\cdot)$ is symmetric. Therefore, 
	\[\dduality{v}{\pi_n (\mathbf{g}) - \pi_n(\nabla \mathcal{U})} = 0\, \quad \forall v \in \mathbb{H}^{1/2}(\Gamma)\,,\]
	implying that $\pi_n(\mathbf{g}) = \pi_n(\nabla \mathcal{U})$ in $\mathbb{H}^{-1/2}(\Gamma)$ by eq.~\eqref{eq:isometricPairing}. It remains to prove the uniqueness of the solution $\mathcal{U}$. Let $\mathcal{U}$ be a solution of the PDE with $\mathbf{g} = 0$. The boundary condition $\pi_n (\nabla \mathcal{U})$ can then be rephrased as $\nabla \mathcal{U} \in H_{0,\Gamma}(\div,\R^3)$. Let $\rho_0 >0$ be sufficiently large so that $\Gamma$ is contained in a ball $\subset B_{\rho_0} \isdef \enstq{x \in \R^3}{\norm{x} < \rho_0}$. Observe that by the representation theorem, there holds 
	\begin{equation}
		\mathcal{U}(x) = \int_{\partial B_\rho} \frac{\vec n_\rho(y) \cdot (x - y)}{4\pi\norm{x - y}^3}\mu(y)dy-\int_{\partial B_\rho} \frac{1}{4\pi\norm{x - y}}\lambda(y)dy \quad \forall x \in (\overline{B_\rho})^c
	\end{equation}
	where $\mu(y) = \mathcal{U}(y) \in H^{1/2}(\partial B_\rho)$ and $\lambda(y) = \frac{y}{\norm{y}} \cdot \nabla \mathcal{U}(y) \in H^{-1/2}(\partial B_\rho)$. In particular, one has $\nabla \mathcal{U} = O\left(\frac{1}{\norm{x}^2}\right)$ uniformly as $x \to \infty$. Given $\rho \geq \rho_0$, integrating by parts on $B_\rho \cap (\R^3 \setminus \Gamma)$ and using that $\Delta \mathcal{U} = 0$ on each $\R^3 \setminus \Gamma$ and $\nabla \mathcal{U} \in H_{0,\Gamma}(\div,\R^3)$, we obtain
	\begin{equation}
		\label{eq:IBPuniqueness}
		\int_{B_\rho \cap (\R^3 \setminus \Gamma)} \abs{\nabla \mathcal{U}}^2 = \int_{\partial B_\rho} \mathcal{U} \left(\frac{x}{\norm{x}} \cdot \nabla \mathcal{U}\right)dx\,.
	\end{equation}
	The decay conditions for $\mathcal{U}$ and $\nabla \mathcal{U}$ imply that the right-hand in eq.~\eqref{eq:IBPuniqueness} tends to $0$ as $\rho \to \infty$. We conclude that $\nabla \mathcal{U} = 0$ on $\R^3 \setminus \Gamma$ hence $\mathcal{U} = 0$ since $\R^3\setminus \Gamma$ is connected. 
\end{proof}
	
	\section{A dense subspace of $H^1(\R^3 \setminus \Gamma)$}
	\label{sec:proofDensity}
	\begin{definition}[Space $X^\infty$, see also \cite{claeys2021quotient}]
		Let $X^\infty$ be the space defined by 
		\[X^\infty \isdef \enstq{u \in H^1(\R^3\setminus \Gamma) \cap C^\infty(\R^3 \setminus \Gamma)}{u_j \in C^\infty(\overline{\Omega_j}) \textup{ for all } j = 0,\ldots, J}\,,\]
		where, for any open set $U \subset \R^3$, the set $C^\infty(\overline{U})$ is the set of restrictions to $U$ of elements of $C^\infty(\R^3)$. 
	\end{definition}
	The goal of this section is to prove the following result, stated in \cite{claeys2021quotient}:
	\begin{theorem}[Density of $X^\infty$]
		\label{thm:densityXinfty}
		The set $X^\infty$ is dense in $H^1(\R^d \setminus \Gamma)$. 
	\end{theorem}
	In what follows, for $j = 0,\ldots,J$, let
	\[\Sigma_j \isdef \partial \Omega_j \setminus \overline{\Gamma_j}\]
	where $\Gamma_j = \Gamma \cap \partial \Omega_j$. Recall that both $\Gamma_j$ and $\Sigma_j$ are simple Lipschitz screens with $\partial \Gamma_j = \partial \Sigma_j$. We recall a result from \cite{claeys2013integral} (the statement therein is weaker than the one below, but the proof given in that reference actually proves the stronger statement).
	\begin{proposition}[{\cite[Prop. 8.11]{claeys2013integral}}]
		\label{prop:density2013}
		The set of functions $u \in H^1(\R^d \setminus \Gamma)$ which vanish in a neighborhood of $\cup_{j = 0}^J \partial \Gamma_j$ is dense in $H^1(\R^d \setminus \Gamma)$. 
	\end{proposition}
	
	\begin{lemma}
		Let $K$ be a compact set of $\R^n$ and $V \subset \R^n$ such that $d(K,V) > 0$. Then there exists a function $f \in W^{1,\infty}(\R^d)$ such that $f = 0$ on $K$ and $f = 1$ on $V$. 
	\end{lemma}
	\begin{proof}
		It suffices to put $f(x) = \frac{d(x,K)}{d(x,K) + d(x,V)}$. First notice that $$g(x) \isdef d(x,K) + d(x,V) \geq d(K,V) =: \alpha > 0$$
		since $K$ is compact. Moreover,
		\begin{align*}
			\abs{f(x) - f(y)} &= \abs{\frac{d(y,V) \left[d(x,K) - d(y,K)\right] + d(y,K)\left[d(y,V) - d(x,V)\right]}{g(x) g(y)}}\\
			& \leq \frac{d(y,V)}{g(y)} \frac{d(x,y)}{\alpha} + \frac{d(y,K)}{g(y)} \frac{d(x,y)}{\alpha}\\
			& = \frac{d(x,y)}{\alpha} \quad \forall x,y \in \R^n\,.
		\end{align*} 
	\end{proof}
	\begin{corollary}
		\label{petitCor}
		In the previous result, one can also choose $f$ such that $f = 1$ in a neighborhood of $K$ and $0$ in a neighborhood of $V$.
	\end{corollary}
	\begin{proof}
		Take $f$ as before. Note that $K \subset f^{-1}({0})$ and $V \subset f^{-1}({1})$. Therefore, for any neighborhoods $U_0$ and $U_1$ of $0$ and $1$ respectively in $\R$, $f^{-1}(U_0)$ and $f^{-1}(U_1)$ are neighborhoods of $K$ and $V$ respectively in $\R^n$. Based on this idea, for some $\eta \in (\frac{1}{2},1)$, let $f_\eta$ be defined by 
		\[f_\eta(x) = \begin{cases}
			\eta & \textup{if } f(x) \geq \eta\\
			f(x) & \textup{if } 1 - \eta \leq f \leq \eta\\
			1 - \eta & \textup{if } f(x) \leq 1 - \eta
		\end{cases}\]
		Set $g = \frac{f_\eta - (1 - \eta)}{2\eta - 1}$. Then $g$ is equal to $0$ in the set $f^{-1}(]-\infty,1-\eta[)$ and $g = 1$ on $f^{-1}(]\eta,+\infty[)$, so $g$ satisfies the required property. 
	\end{proof}
	
	\begin{lemma}
		\label{decomp}
		Let $u \in H^1(\R^d \setminus \Gamma)$ be such that $u$ vanishes in the neighborhood of $\cup_{j}\partial \Gamma_j$. Then 
		\[u = u_1 + u_2\]
		where $u_1 \in H^1(\R^d \setminus \Gamma)$ vanishes in a neighborhood of $\overline{\Sigma}_j$ for all $j$ and $u_2$ is the restriction to $\R^d \setminus \Gamma$ of a function $U_2 \in H^1(\R^d)$. 
	\end{lemma}
	\begin{proof}
		Let $U$ be an open neighborhood of $\cup_{j = 1}^J \partial \Gamma_j$ in which $u$ vanishes. Define the compact set $K = \Gamma \setminus U$, and let $V = \cup_{j} \overline{\Sigma_j}$. Note that $V$ is closed and disjoint from $K$, so since $K$ is compact, it is at a positive distance of $K$. Thus we can apply \autoref{petitCor} to fix a function $\chi \in W^{1,\infty}(\R^d)$ such that $\chi = 0$ in a neighborhood of $K$ and $\chi = 1$ in a neighborhood of $V$. Then we define $u_1 = (1 - \chi) u$ and $u_2 = \chi u$. By the chain rule, it is immediate that $u_1$ and $u_2$ are in $H^1(\R^d\setminus \Gamma)$ since $\chi \in W^{1, \infty}(\R^d \setminus \Gamma)$. Clearly, $u_1$ satisfies the required property by definition of $\chi$. It remains to show that $u_2$ can be extended to a $H^1$ function on $\R^d$. For this, we set $U_2(x) = 0$ for $x \in \Gamma$ and $U_2(x) = u_2(x)$ for $x \in \R^d \setminus \Gamma$. Let $U_K$ be a neighborhood of $K$ on which $\chi$ vanishes. Then $U_K \cup U$ is a neighborhood of $\Gamma$ where $U_2$ vanishes. Thus $U_2$ is in $H^1(\R^d)$ since
		\begin{align*}
			\int_{\R^d} |U_2(x)|^2 + |\nabla U_2(x)|^2 &= \int_{\R^d \setminus (U_K \cup U)}|U_2(x)|^2 + |\nabla U_2(x)|^2 dx\\
			& = \int_{\R^d \setminus (U_k \cup U)}|u(x)|^2 + |\nabla u(x)|^2dx
			 \leq \norm{u}^2_{H^1(\R^d \setminus \Gamma)}\,.
		\end{align*}
	\end{proof}
	
	\begin{lemma}
		\label{lemtmpDensityXinfty}
		Let $\Omega$ be a Lipschitz domain and let $\Sigma \subset \partial \Omega$ be a Lipschitz screen. If $u \in H^1(\Omega)$ vanishes in a neighborhood of $\Sigma$, then for all $\varepsilon > 0$ there exists $\varphi_\varepsilon \in C^\infty(\overline{\Omega}) \cap H^1(\Omega)$ such that $\varphi = 0$ in a neighborhood of $\Sigma$ and 
		\[\norm{u - \varphi_\varepsilon}_{H^1(\Omega)} \leq \varepsilon\,.\]
	\end{lemma}
	\begin{proof}
		First let $U$ be a neighborhood of $\Sigma$ where $u = 0$. Introduce a smooth cutoff function $\chi$ which is identically $1$ outside $U$ and vanishes in a smaller neighborhood $V$ of $\Sigma$. This is possible since $\Sigma$ is closed \cite[Cor 16.4]{treves2006topological}. Let $M = \norm{\chi}_{W^{1,\infty}(\Omega)}$. Fix $\varepsilon > 0$. Since $C^\infty(\overline{\Omega}) \cap H^1(\Omega)$ is dense in $H^1(\Omega)$ (because $\Omega$ is a Lipschitz domain, see \cite[Thm. 3.29]{McLean}), we can find $\varphi \in C^\infty(\overline{\Omega})$ such that $\norm{u - \varphi}_{H^1(\Omega)} \leq \frac{\varepsilon}{M}$. Then $\chi \varphi$ vanishes in a neighborhood of $\Sigma$ and 
		\[\norm{\chi \varphi - u}_{H^1(\Omega)} = \norm{\chi\varphi - \chi u} \leq M \norm{\varphi - u} \leq \varepsilon\]
		concluding the proof. 
	\end{proof}

	\begin{proof}[Proof of \Cref{thm:densityXinfty}]
		By \Cref{prop:density2013}, we can first assume that $u$ vanishes in a neighborhood of $\cup_{j = 1}^J\partial \Gamma_j$ and therefore represent it as $u_1 + u_2$ as in \Cref{decomp}. Fix $\varepsilon > 0$. By density of $C^\infty_c(\R^d)$ in $H^1(\R^d)$ (\cite[Lemma 3.24]{McLean}), there is $\varphi_2 \in C^\infty_c(\R^d)$ such that $\norm{u_2 - \varphi_2}_{H^1(\R^d)} \leq \varepsilon$. On the other hand, by \Cref{lemtmpDensityXinfty} for each $j$, there exists $\psi_j \in C^\infty(\overline{\Omega_j}) \cap H^1(\Omega_j)$ such that $\psi_j$ vanishes in a neighborhood of $\overline{\Sigma}_j$ and 
		\[\norm{u_{1|\Omega_j} - \psi_j}_{H^1(\Omega_j)} \leq \varepsilon\,.\] 
		Let $\psi$ be defined by 
		\[\psi(x) = \begin{cases}
			\psi_j(x) & \textup{if } x \in \Omega_j\,,\\
			0 & \textup{if } x \cup_{j} \Sigma_j\,.
		\end{cases}\]
		We claim that $\psi \in X^\infty$. Indeed, $\psi$ is $C^\infty$ at any $x \in \Omega_j$, and if $x \in \Sigma_j$ for some $j$, then $\psi$ is identically $0$ in a neighborhood of $x$. We furthermore have 
		\[\norm{u_1 - \psi}_{H^1(\R^d \setminus \Gamma)} \leq J\varepsilon\,.\]
		In conclusion, letting $\varphi = \varphi_1 + \varphi_2$, we can write
		\[\norm{u - \varphi}_{H^1(\R^d\setminus \Gamma)} \leq (J+1)\varepsilon\,,\]
		concluding the proof.
	\end{proof}

	\section{Properties of the Hypersingular bilinear form}
	
	\begin{lemma}[Poincaré-type inequality]
		\label{lem:poincaré}
		Let $\Gamma$ be a polygonal multiscreen such that $\R^3 \setminus \Gamma$ is connected and let $\Omega_0,\ldots,\Omega_j$ be as in \Cref{def:polygonalMS}. There exists a positive constant $C = C(\Gamma,\Omega_0,\ldots,\Omega_j) > 0$ such that
		\[\norm{u}^2_{H^1(\R^3 \setminus \Gamma)} \leq C \Big(\norm{u}^2_{H^1(\Omega_0)} + \sum_{j = 1}^J \norm{\nabla u}^2_{L^2(\Omega_j)}\Big) \qquad \forall u \in H^1(\R^3 \setminus \Gamma)\,.\]
	\end{lemma}
	\begin{proof}
		It suffices to show that there exists $C > 0$ such that 
		\[\sum_{j = 1}^J \norm{u}^2_{L^2(\Omega_j)} \leq C \Big(\norm{u}^2_{H^1(\Omega_0)} + \sum_{j = 1}^N \norm{\nabla u}^2_{L^2(\Omega_j)}\Big)\,.\]
		Assuming that it is not true, one may construct a sequence $(u_n)_{n \in \N}$ of functions in $H^1(\R^3 \setminus \Gamma)$ such that 
		\begin{equation}
			\label{eq:contra1}
			\sum_{j = 1}^J \norm{u_n}^2_{L^2(\Omega_j)} = 1 \qquad \forall n \in \N
		\end{equation}
		\begin{equation}
			\label{eq:contra2}
			\lim_{n \to 0} \Big(\norm{u_n}^2_{H^1(\Omega_0)} + \sum_{j = 1}^N \norm{\nabla u_n}^2_{L^2(\Omega_j)}\Big) = 0\,.
		\end{equation}
		Extracting a subsequence, we can assume that $u_n$ converges weakly in $H^1(\R^3 \setminus \Gamma)$ to some $u_\infty \in H^1(\R^3 \setminus \Gamma)$. By eq.~\eqref{eq:contra2}, $u_\infty = 0$ on $\Omega_0$. 
		Moreover, the conditions in eqs.~\eqref{eq:contra1} and \eqref{eq:contra2} imply that the sequences $(\norm{u_n}_{H^1(\Omega_j)})_{n\in \N}$ are bounded bounded for $j = 1,\ldots,J$. Using the compact embedding $H^1(\Omega_j) \subset\subset L^2(\Omega)$ (since $\Omega_j$ is bounded for $j = 1,\ldots,J$) and extracting a new subsequence, one can further assume that 
		\begin{equation}
			\label{eq:strongConvL2}
			\lim_{n \to \infty}\norm{u_n - u_\infty}_{L^2(\Omega_j)} = 0\,.
		\end{equation}
		We now show that $u_\infty$ is locally constant by computing the quantity $\ell:=\lim_{n \to \infty}\int_{\R^3 \setminus \Gamma} u_n u_\infty + \nabla u_n \cdot \nabla u_\infty$ in two different ways. On the one hand, by weak convergence, $\ell = \norm{u_\infty}^2_{H^1(\R^3 \setminus \Gamma)}$. On the other hand, using eqs.~\eqref{eq:contra2} and \eqref{eq:strongConvL2},
		$\ell = \norm{u_\infty}^2_{L^2(\R^3 \setminus \Gamma)}$. Thus, 
		$$\norm{\nabla u_\infty}^2_{L^2(\R^3 \setminus \Gamma)} = \norm{u}^2_{H^1(\R^3 \setminus \Gamma)} - \norm{u}^2_{L^2(\R^3 \setminus \Gamma)} = \ell - \ell = 0\,,$$ i.e., $u_\infty$ is locally constant. Since $\R^3 \setminus \Gamma$ is connected, it follows that $u_\infty = 0$, which contradicts eqs.~\eqref{eq:contra1} and \eqref{eq:strongConvL2}.
	\end{proof}
	\label{sec:proofPropertiesW}
	\begin{lemma}
		\label{idDLDL}
		For all $u,v \in \mathbb{H}^{1/2}([\Gamma])$, one has the identity
		\begin{equation}
			\dduality{u}{\W v} = \sum_{j = 0}^{J} \int_{\Omega_j} \nabla \DL u \cdot \nabla \DL v\,dx\,.
		\end{equation}
	\end{lemma}
	\begin{proof}
		We first notice that 
		\[\dduality{\gamma(\DL u)}{\W v} = \dduality{u}{\W v}\]
		due to the jump relation \eqref{jumpRel}, the polarity of the single trace spaces $H^{\pm 1/2}([\Gamma])$, and the fact that $\W v \in H^{-1/2}([\Gamma])$. Moreover, by definition 
		\[\dduality{\gamma(\DL u)}{\W v} = \sum_{j = 0}^{J} \int_{\Omega_j} \nabla \DL u(x) \cdot \nabla \DL v(x) + \DL u(x) \Delta(\DL v)(x)\,dx\,.\]
		The second term vanishes since $\DL u$ is harmonic in $\R^3\setminus \Gamma$, proving the result. 
	\end{proof}
	
	\begin{theorem}
		\label{thmWpos}
		There exists a constant $c_W > 0$ such that
		\begin{equation}
			\forall u \in \mathbb{H}^{1/2}(\Gamma)\,, \quad \dduality{u}{\W u} \geq c_W\norm{[u]_\Gamma}^2_{\widetilde{H}^{1/2}([\Gamma])}
		\end{equation}
	\end{theorem}
	\begin{proof}
		Let $\chi \in C^\infty_c(\R^3)$ be a compactly supported
		function such that $\chi \equiv 1$ in a neighborhood of
		$\Omega$. Because the support of $\chi$ is bounded, we have
		$C_{\chi} = \sup_{x\in\R^{3}}(\vert \chi(x)\vert +\vert\nabla\chi(x)\vert)(1+\Vert x\Vert^2)<\infty$.          
		By \Cref{lem:poincaré}, the jump relation
		\eqref{jumpRel}, using the quotient definition of the
		$\widetilde{H}^{1/2}([\Gamma])$ norm and using that $\gamma
		(\DL u) = \gamma(\chi \DL u)$,
		\begin{equation*}
			\begin{aligned}
				\norm{[u]_\Gamma}_{\widetilde{H}^{1/2}([\Gamma])}
				& = \norm{[\DL u]_\Gamma}_{\widetilde{H}^{1/2}([\Gamma])}\\
				& \leq C_P \sum_{j = 1}^J \int_{\Omega_j}\abs{\nabla \DL u}^2 + 2\int_{\Omega_0} (\abs{\chi}^2 + \abs{\nabla \chi}^2)(\abs{ \DL u}^2 + \abs{\nabla \DL u}^2)\\
				& \leq C_p' \sum_{j = 0}^J \int_{\Omega_j}\abs{\nabla \DL u}^2 + C_P'\int_{\Omega_0} \abs{ \DL u}^2/(1+\Vert x\Vert^2)\;dx
			\end{aligned}
		\end{equation*}
		where we applied the Leibniz rule to the term $\nabla (\chi \DL u)$, and introduced the
		constant $C_p' = \max(C_P, C_{\chi})$. To conclude, we may apply the Poincaré inequality in the Beppo Levi space 
		\cite[Thm 2.10.10]{sauter2011boundary} which shows that
		$\int_{\Omega_0} \vert\DL u\vert^2/(1 + \norm{x}^2) dx \leq C\int_{\Omega_0} \abs{\nabla \DL u}^2 dx$ for some fixed
		constant $C>0$ that does not depend on $u$. This finishes the proof.
		
	\end{proof}
	
	To prove \Cref{thmWweak}, we start by two elementary technical lemmas. 
	\begin{lemma}[Almost all points of the skeleton are on exactly two boundaries]
		\label{lemnjnk}
		Let $j,k \in \{0,\ldots,J\}$, $j \neq k$. Then, for $\sigma_j$-almost all $x \in \partial \Omega_j \cap \partial \Omega_k$, there holds 
		\begin{equation}
			\label{oppVecs}
			\vec n_j(x) = -\vec n_k(x)\,.
		\end{equation}
		Moreover, if $j,k,l$ are three distinct indices of $\{0,\ldots,J\}$, then $\partial \Omega_j \cap \partial \Omega_k \cap \partial \Omega_l$ is of $\sigma_j$-measure $0$. 
	\end{lemma}
	\begin{proof}
		Let us first assume that $j,k \neq 0$. We then decompose $\partial \Omega_j \cap \partial \Omega_k$ in triangles, using the meshes $\partial \mathcal{M}_{\Omega_j}$ and $\partial \mathcal{M}_{\Omega_k}$, which are both regular. A triangle $T$ of $\partial \mathcal{M}_{\Omega_j}$ is incident to two tetrahedrons exactly, $K_j \in \mathcal{M}_{\Omega_j}$ and $K_k \in \mathcal{M}_{\Omega_k}$. For $x$ in the relative interior of $\abs{T}$, the relation \eqref{oppVecs} is obvious. What remains, i.e. (the convex hulls of) the edges of $\partial \mathcal{M}_{\Omega_j} \cap \partial \mathcal{M}_{\Omega_k}$, is a set of surface measure $0$. The case where one of the indices $j,k$ is $0$ is treated similarly. The last statement is also obtained by reasoning on the decomposition in triangles, edges and vertices.
	\end{proof}

	\begin{lemma}[All points of the skeleton are on at least two boundaries]
		\label{lem2bound}
		For each $j \in \{0\,,\ldots\,,J\}$, there holds
		\[\partial \Omega_j = \bigcup_{k \in \{0\,,\ldots J\} \setminus \{j\}} \partial \Omega_j \cap \partial \Omega_k\,.\]
	\end{lemma}
	\begin{proof}
		Fix $j \in \{0\,,\ldots\,, J\}$, let $x \in \partial \Omega_j$ and, seeking a contradiction, assume that $x$ is not in $\partial \Omega_k$ for any $k \neq j$. We first deduce that $x$ is not in the set
		\[\bigcup_{k \in \{0,\ldots,J\}\setminus \{j\}} \overline{\Omega_k}\,.\]
		Indeed, it is impossible for $x$ to be in $\Omega_k$ for any $k \neq j$, because otherwise, there would be a ball $B_{x}$ centered at $x$ such that $B_{x} \subset \Omega_k$. But since $x \in \partial \Omega_j$, the ball $B_{x}$ contains at least a point of $\Omega_j$, implying that $\Omega_j \cap \Omega_k \neq \emptyset$, contradicting the fact that $\Omega_0\,,\ldots\,,\Omega_J$ is a Lipschitz partition of $\R^3$. We now construct a point $y$ which is not in the union 
		\[\bigcup_{k \in \{0,\ldots\,J\}} \overline{\Omega_k}\,.\]
		To do this, we remark that $x$ is at a positive distance of $\overline{\Omega_k}$ for every $k \neq j$, so there exists $\varepsilon$ small enough so that, for all $k \neq j$,  $B(x,\varepsilon)\cap \overline{\Omega_k} = \emptyset$. In this same ball, we claim that there must be a point $y \notin \overline{\Omega_i}$: if there were not, we would then have $B(x,\varepsilon) \subset \overline{\Omega_i}$, i.e. $x \in \textup{int}(\overline{\Omega})$. But, since every Lipschitz domain $\Omega$ satisfies the property $\textup{int}(\overline{\Omega}) = \Omega$, we would have $x \in \Omega_i$ which is impossible since $x$ was chosen in $\partial \Omega_i$ to begin with. The existence of $y$ is proven, yet impossible since 
		\[\bigcup_{k = 1}^J \overline{\Omega_k} = \R^3\,,\]
		which is the desired contradiction.	
	\end{proof}
	\begin{remark}
		The previous proof does not require a polygonal multi-screens, but can be applied to general, Lipschitz multi-screens.
	\end{remark}
	
	We now prove that the weak representation identity holds when $u$ and $v$ are sufficiently smooth, so that all integrations by parts make sense. We then obtain \Cref{thmWweak} using a density argument.
	\begin{lemma}[Weakly singular representation of the bilinear form $a$ on $X^\infty \times X^\infty$]
		\label{thmW}
		For $u,v \in X^\infty$, there holds 
		\begin{equation}
			\label{identity}
			a\left([u]_\Gamma,[v]_\Gamma\right) = \sum_{j,k = 1}^J \iint_{\Gamma_j\times \Gamma_k} \frac{ \curl_j u_j(x) \cdot \curl_k u_k(x')}{4\pi \norm{x - x'}}\,d\sigma_j(x) d\sigma_k(x')\,.
		\end{equation}
	\end{lemma}
	\begin{proof}
		We adapt the approach of McLean \cite[Chap. 9]{McLean}. 
		In view of \Cref{lemnjnk}, it is not difficult to see that \eqref{identity} can be equivalently written as
		\begin{equation}
			\label{identity2}
			a(\varphi,\psi) = \sum_{j= 0}^J \sum_{k = 0}^J \int_{\partial \Omega_j}\int_{\partial \Omega_k}\frac{ \curl_{j} \varphi_j(x) \cdot \curl_{k} \psi_k(y)}{4\pi\norm{x - y}}\,d\sigma_j(x) \,d\sigma_k(y)\,,
		\end{equation} 
		where $\varphi = [u]_{\Gamma}$, $\psi = [v]_{\Gamma}$. Hence in what follows we prove that eq.~\eqref{identity2} holds. From now on, we fix $u$ and $v$ satisfying the hypothesis of the theorem. Furthermore, let us fix $x \notin \Sigma$, and let $\chi_{x}$ be a smooth compactly supported function that equals $0$ near $x$ and $1$ near $\Sigma$. Then the function $y \mapsto \mathcal{G}_{x}(y)$ (with this choice of $\chi$) is infinitely differentiable on $\partial \Omega_j$ for each $j$, so we may write
		\begin{align*}
			\DL \left(\gamma\,u\right)(x) &= \dduality{\gamma(u)}{\pi_n (\nabla \mathcal{G}_{x})} = -\sum_{j = 0}^J \int_{\partial \Omega_j} \vec n_j \cdot \vec \nabla \mathcal{G}_{x}\, \varphi_j \,d\sigma_j = \sum_{j=0}^J \int_{\partial \Omega_j} \vec n_j \cdot \vec \nabla \mathcal{G}_{x}\, \varphi_j\,d\sigma_j\,. 
		\end{align*}
		Hence, for every $x \in \R^3 \setminus \Sigma$, we have the formula 
		\begin{equation}
			\label{sumDLj}
			\DL \left( \gamma\,u \right)(x) = \sum_{j = 0}^J D_j(x)\,
		\end{equation}	
		where
		$$D_j(x) = \DL_j\,\varphi_j(x) \isdef -\int_{\partial \Omega_j} \vec n(y) \cdot \vec \nabla_{y}\left(\frac{1}{4\pi\norm{x - y}}\right) \varphi_j(y) d\sigma_j(y)$$
		is the classical double-layer potential associated to the domain $\Omega_j$ with density $\varphi_j$. 
		Since $(D_j)_{|\Omega_j}$ is in $H^{1}_{loc}(\Omega_j)$ and $(D_j)_{|\Omega_j^c}$ in $H^1_{loc}(\Omega_j^c)$, we can define a locally integrable vector field $\vec F_j$ by  
		\[\vec F_j(x) = \begin{cases}
			\vec \nabla [(D_{j})_{|\Omega_j}](x) & \textup{for } x \in  \Omega_j \,,\\
			\vec \nabla [(D_{j})_{|\Omega_j^c}](x) & \textup{for } x \in  (\overline{\Omega_j})^c\,.
		\end{cases}\]
		We introduce the single-layer potential $\textup{SL}_j$ associated to the Lipschitz domain $\Omega_j$. For any smooth vector field $\vec u_j$ on $\partial \Omega_j$, it is defined by
		\[\forall x \in \R^3 \setminus \partial \Omega_j\,, \quad  \textup{SL}_j \vec u_j(x) \isdef \int_{\partial \Omega_j} \frac{\vec u_j(y)}{4\pi \norm{x - y}} d\sigma_j(y)\,.\]
		Let $\vec A_j = \textup{SL}_j\, (\curl_j \,\varphi_j)$. For each $k \in \{1\,\ldots\,J\}$, the trace of $\vec A_j$ on $\partial \Omega_k$ is well-defined and given by 
		\begin{equation}
			\label{gammakAj}
			\gamma_k \vec A_j(x) = \int_{\partial \Omega_j} \frac{\curl_j \varphi_j(y)}{4\pi||x - y||} d\sigma_j(y)\,,
		\end{equation}
		since this integral is at most weakly singular for $x \in \partial \Omega_k$. Observe that, by the assumption that $u \in X^\infty$, we can choose $\widetilde{u}_j \in C^\infty_c(\R^3)$ such that $u_j$ coincides with $\widetilde{u}_j$ in a neighborhood of $\partial \Omega_j$, and thus, $\curl_j \varphi_j = \vec n_j \times \gamma_j \nabla \widetilde{u}_j$. The central argument is then the following identity, obtained via integration by part:
		\[{\curl \vec A_j = -\vec F_j\quad \textup{on } \mathbb{R}^3\,,}\]
		see \cite[Lem. 9.14]{McLean} (the difference of sign with respect to \cite{McLean} comes from opposite conventions in the definition of $\DL$). 
		By definition, we have
		\begin{align*}
			\dduality{\W \gamma(u)}{\gamma(v)} &=\dduality{\pi_n(\vec\nabla \,\DL \gamma(u))}{\gamma(v)}\\
			&=\sum_{j = 0}^J \int_{\Omega_j} \vec \nabla \DL (\gamma \,u) \cdot \vec \nabla v + v\left\{\Delta \DL(\gamma\, u)\right\}\,dx\,. 
		\end{align*}
		The second term vanishes, so that, using eq.~\eqref{sumDLj},
		\begin{align*}
			\dduality{\W \gamma(u)}{\gamma(v)}  &= \sum_{j,k = 0}^J \int_{\Omega_j}\vec \nabla D_k \cdot \vec \nabla v \,dx\,,\\
			&=\sum_{j,k = 0}^J \int_{\Omega_j} \vec F_k \cdot \vec \nabla v \,dx\,,\\
			&=\sum_{j,k = 0}^J \int_{\Omega_j}  -\curl(\vec A_k)\cdot \vec \nabla v\,dx = -\sum_{j,k = 0}^J \int_{\Omega_j} \div(\vec A_k \times \vec \nabla v)\,dx\,,
		\end{align*}
		in view of the identities $\div(\vec A \times \vec B) = \curl \,\vec A \cdot \vec B -\vec A \cdot  \curl \,\vec B$ and $\curl \vec \nabla = 0$. Applying the divergence theorem in each $\Omega_j$, we get 
		\begin{align*}
			\dduality{\W \gamma(u)}{\gamma(v)} & = -\sum_{j,k = 0}^J \int_{\partial \Omega_j} \vec n_j \cdot (\gamma_j \vec A_k \times \vec \nabla_j\, \psi_j)d\sigma_j\,.
		\end{align*}
		Permuting the triple product, 
		\begin{equation}
			\label{eq:endProofW}
			\dduality{\W \gamma(u)}{\gamma(v)} = \sum_{j,k = 0}^J \int_{\partial \Omega_j} \gamma_j\vec A_k \cdot \curl_j\, \psi_j\,d\sigma_j\,.
		\end{equation}
		We obtain \eqref{identity2} after replacing $\gamma_k \vec A_j$ with eq.~\eqref{gammakAj}. This proves the Theorem, since $a([u]_\Gamma,[v]_\Gamma) = \dduality{W(\gamma (u))}{\gamma (v)}$. 
	\end{proof}
	\begin{proof}[Proof of \Cref{thmWweak}]
		The proof involves material from  \cite{buffa2002traces,buffa2002boundary}, which we recall here. Let 
		\[\vec L^2_t(\partial \Omega_j) \isdef \enstq{u \in (L^2(\Gamma))^3}{\vec u \cdot \vec n_j = 0 \textup{ a.e. on } \partial \Omega_j}\,.\] 
		Let $\pi_{\tau,j}: L^2(\partial \Omega_j)^3 \to \vec L^2_t(\Gamma)$ be the operator defined, for $\vec u = \vec U_{|\Gamma}$, $\vec U \in \mathcal{D}(\R^3)$, by  
		\[\pi_{\tau,j} \vec u(\vec x)= \vec n \times (\vec U(\vec x) \times \vec n)\,, \quad \forall x \in \partial \Omega_j\]
		and extended to $(L^2(\Gamma))^3$ by density. Let $V_\pi(\partial \Omega_j) \subset \vec L^2_t(\partial \Omega_j)$ be the Hilbert space defined by
		\[V_\pi(\partial \Omega_j) \isdef \enstq{\pi_{\tau,j} \vec u}{\vec u \in (H^{1/2}(\partial \Omega_j))^3}\,,\]
		with the graph norm, and let $V'_\pi(\partial \Omega_j)$ be the dual of $V_\pi(\partial \Omega_j)$. The space $V_\pi$ is dense in $\vec L^2_t(\partial \Omega_j)$, hence one can identify $\vec L^2_t(\partial \Omega_j)$ with a dense subspace of $V'_\pi(\partial \Omega_j)$, and the duality pairing $_{V'_\pi}\langle \cdot,\cdot\rangle_{V_\pi}$ is the unique continuous extension of the $\vec L^2_t$ pairing. Let $\textup{div}_j: \vec L^2_t(\partial \Omega_j) \to H^{-1}(\partial \Omega_j)$ be the adjoint of $\nabla_j$, 
		where $H^{-s}(\partial \Omega_j)$ is the dual of $H^s(\partial \Omega_j)$ for $0 \leq s \leq 1$. Finally, define 
		\[\vec H^{-1/2}(\div_{j},\partial \Omega_j) \isdef \enstq{\vec u \in V'_\pi(\partial \Omega_j)}{\div_j \vec u \in H^{-1/2}(\partial \Omega_j)}\,,\] 
		equipped with the graph norm. Since $u_k \in H^1(\Omega_k)$, we have $\nabla u_j \in H(\curl,\Omega_j)$ hence, by \cite[Thm 4.1]{buffa2002traces}, $\vec n_j \times \gamma_j (\nabla u_j) \in \vec H^{-1/2}(\div_{j},\partial \Omega_j)$ with
		\[\norm{\vec n_j \times \gamma_j (\nabla u_j)}_{\vec H^{-1/2}(\div_{j},\partial \Omega_j)} \leq C \norm{u_j}_{H^1(\Omega_j)}\,.\] 
		Furthermore, by [Prop. 2] and [Thm. 4] of \cite{buffa2002boundary} the map $\textup{SL}_j$ defined for $\varphi_j \in \vec L^2_t(\Gamma)$ by 
		\[\textup{SL}_j \vec \varphi_j \isdef \int_{\partial \Omega_j} \frac{\vec \varphi_j(y)d\sigma_j(y)}{4\pi \norm{\vec x - y}}d\sigma_j(y)\]
		admits a unique linear continuous extension into a mapping $\textup{SL}_j: \vec H^{-1/2}(\div_j,\partial \Omega_j) \to H^1_{\rm loc}(\R^3)$. Namely, this extension reads $V'_\pi \ni \vec \lambda \mapsto \textup{SL}_j (i_\pi(\vec \lambda))$, where $i_\pi: V'_\pi \to (H^{-1/2}(\partial \Omega_j))^3$ is defined in eq. (10) of \cite{buffa2002boundary}. Therefore, the bilinear form $M: H^1(\R^3 \setminus \Gamma) \times {H^1(\R^3 \setminus \Gamma)}$ defined by
		\begin{equation}
			\label{eq:defMuv}
			M(u,v) \isdef \sum_{j,k= 0}^J \duality{\pi_{\tau,k}\,\textup{SL}_j (\vec n_j \times \gamma_j \nabla u_j)}{\vec n_k \times \gamma_k \nabla u_k}_{V_\pi(\partial \Omega_k) \times V'_\pi(\partial \Omega_k)}
		\end{equation} 
		is continuous. Moreover, when $\gamma_j u \in H^1(\partial \Omega_j)$, both terms in the duality pairing appearing in eq.~\eqref{eq:defMuv} are in $\vec L^2_t(\partial \Omega_j)$ so, using the commuting property $\gamma_j \nabla = \nabla_j$, the expression becomes (simplifying the integrals over $\partial \Omega_j \setminus \Gamma$ as pointed out in the proof of \Cref{thmW})
		\begin{equation*}
			M(u,v) = \sum_{j,k = 0}^J \iint_{\Gamma_j \times \Gamma_k} \frac{\curl_j u_j(x) \cdot \curl_k u_k(x')}{\norm{x - x'}}d\sigma_j(x)d\sigma_k(x') \quad \forall u,v  \textup{ s.t. } \gamma_j u, \gamma_j v \in H^1(\partial \Omega_j)\,.
		\end{equation*}
		For $u,v \in X^\infty$, we deduce that that $M(u,v) = a([u]_\Gamma,[v]_\Gamma)$ by \Cref{thmW}. Hence, the continuous bilinear forms $M(\cdot,\cdot)$ and $a(\cdot,\cdot)$ agree on $X^\infty \times X^\infty \subset H^1(\R^3 \setminus \Gamma) \times H^1(\R^3 \setminus \Gamma)$, therefore, by the density result of \Cref{thm:densityXinfty}, $a([u],[v]) = M(u,v)$ for all $u,v \in H^1(\R^3 \setminus \Gamma)$, concluding the proof. 
	\end{proof}

	\section{Convergence of the Galerkin solution}
	\label{sec:proofGalerkConv}
	We first prove a technical result. Let
	$$X^0 \isdef \enstq{u \in C^0(\R^3 \setminus \Gamma)}{u_{|_{\Omega_j}} \in C^0(\overline{\Omega_j}) \textup{ for all $j = 0,\ldots,J$}}$$
	where, for any open set $U \subset \R^3$,   $C^0(\overline{U})$ is the set of uniformly continuous functions on $U$.
	\begin{lemma}
		\label{pwInterp}
		For $u \in X^0$, let $I_h u$ be the element of $L^2(\R^3)$ defined by 
		$$(I_h u)_{|\Omega_j} = I_{h,j} u_j\,,$$
		where $I_{h,j} :C^0(\overline{\Omega_j}) \to V_h(\Omega_j)$ is the standard Lagrange interpolant. Then $I_h u \in V_h(\Omega \setminus \Gamma)$. 
	\end{lemma}
	\begin{proof}
		Let $u \in X^0$ and let $\vec x_i$ be a vertex of $\mathcal{M}_{\Omega,h}$. Let $K$ and $K'$ be two tetrahedra of $\Delta_{i,j}$ with a common face $F$ (in particular, the relative interior of $|F|$ is disjoint from $\Gamma$). The definition of $X^0$ implies that $u$ is uniformly continuous on $\textup{int}(\abs{K})$ for each tetrahedron $K$ of $\mathcal{M}_{\Omega,h}$. Hence, we can define the continuous extension $u_K$ of $u_{|\textup{int}(\abs{K})}$ to the whole $\abs{K}$ (recall that this is the {\em closed} convex hull of $K$). Let $u_{K'}$ be defined similarly on $K'$. We start by showing that $u_{K}(\vec x_i) = u_{K'}(\vec x_i)$. 
		Indeed, assume that those values differ by a positive quantity $\delta \isdef \abs{u_K(\vec x_i) - u_{K'}(\vec x_i)} > 0$, and let $\varepsilon = \delta /4$. Using that $u_K$ is uniformly continuous, one can find $\eta > 0$ such that for all $ x,  x' \in \abs{K}$, $\abs{ x -  x'} < \eta$  implies that $\abs{u(x) - u( x')} \leq \varepsilon$. Let $\eta'$ be defined similarly for $K'$. Finally, choose a point $x$ in the common face $F$ such that $x$ is at a positive distance of $\Gamma$ (it suffices to take $x$ in the relative interior of $\abs{F}$) and $\abs{x - \vec x_i} \leq \min(\eta,\eta')$. We then have 
		\begin{align*}
			0 < \delta& = \abs{u_K(\vec x_i) - u_{K'}(\vec x_i)}\\
			&\leq \abs{u_K(x) - u_K(\vec x_i)} + \abs{u_{K'}(x) - u_{K'}(\vec x_i)} + \underbrace{\abs{u_K(x) - u_{K'}(x)}}_{= 0} \\
			& \leq 2\varepsilon = \frac{\delta}{2}\,,
		\end{align*}
		which is a contradiction (we used that $u_K(x) = u_{K'}(x) = u(x)$, since $u$ is continuous on $\R^3 \setminus \Gamma$). We deduce that as soon as $K$ and $K'$ both belong to $\Delta_{i,j}$, then $u_{K}(\vec x_i) = u_{K'}(\vec x_i)$ by considering a face-connected path in $\Delta_{i,j}$ from $K$ to $K'$. 
		
		For each $(i,j) \in \mathcal{H}(\Omega)$, choose an element $K \in \mathcal{M}_{\Omega,h}$ and let $u_{i,j} \isdef u_{K}(\vec x_i)$. Let 
		\[u_h = \sum_{(i,j)\in \mathcal{H}(\Omega)} u_{i,j} \phi_{i,j}\,,\] 
		where $\{\phi_{i,j}\}_{(i,j) \in \mathcal{H}(\Omega)}$ is the set of split basis functions. The function $u_h$ obviously belongs to $V_h(\Omega \setminus \Gamma)$ and we now show that $I_h u = u_h$. Let $K$ be an arbitrary element of $\mathcal{M}_{\Omega,h}$ and let $\vec x_i$ be a vertex of $K$. Let $(z_{n})_{n \in \N}$ be a sequence of points of $\textup{int}{\abs{K}}$ converging to $\vec x_i$. Let $j$ be the element of $\{1,\ldots,q_i\}$ be such that $K \in \Delta_{i,j}$. Recalling that, by definition of $\{\phi_{i,j}\}_{(i,j) \in \mathcal{H}(\Omega)}$, $\lim_{n \to \infty} \phi_{i',j'}(z_n) = \delta_{i,i'}\delta_{j,j'}$, we deduce 
		\[\lim_{n \to \infty} u_h(z_n) = u_{i,j}\,.\]
		Let $\ell$ be such that $K \in \mathcal{M}_{\Omega_\ell,h}$ and let $u_{\ell}$ denote the continuous extension of $u_{|_{\Omega_\ell}}$ to $\overline{\Omega_\ell}$. Then by definition 
		\[I_h u(z_n) = I_{h,\ell} u_\ell (z_n)\,;\]
		therefore, using the continuity of the piecewise linear function $I_{h,\ell} u_\ell$ and the definition of $I_{h,\ell}$, 
		\[\lim_{n \to \infty} I_h u(z_n) = (I_{h,\ell} u_{\ell})(\vec x_i) = u_{\ell}(\vec x_i)\,.\]
		Since $\textup{int}(\abs{K}) \subset \Omega_\ell$, $u_{\ell}$ coincides with $u_K$ on $\textup{int}(\abs{K})$, implying
		\[\lim_{n \to \infty} I_h u(z_n) = u_K(\vec x_i) = u_{i,j}\]
		by what precedes, since $K \in \Delta_{i,j}$. In conclusion, the restrictions to $\textup{int}(\abs{K})$ of $I_h u$ and $u_h$ are linear functions whose limits at every vertex of $K$ coincide, hence they are equal on this set. The functions $I_h u$ and $u_h$ thus agree almost everywhere on $\R^3$, and the proof is concluded. 
	\end{proof}
	\begin{proof}[Proof of \Cref{thmConvPhiStar}]
		By \Cref{thm:densityXinfty}, it is sufficient to prove that for any and $u \in X^\infty$ with $\textup{supp}\, u \subset \Omega$, there exists a sequence $(u_h)_{h > 0}$ of elements $u_h \in V_h(\Omega \setminus \Gamma)$ such that 
		\[\lim_{h \to 0}\norm{u - u_h}_{H^1(\R^3 \setminus \Gamma)} = 0\,.\]
		For each $h > 0$, we define $u_h \isdef I_h u$ where $I_h$ is the Lagrange interpolant of \Cref{pwInterp}. Since $X^\infty \subset X^0$, we indeed have $I_h u \subset V_h(\Omega \setminus \Gamma)$ by \Cref{pwInterp}, and
		\[\norm{u - I_h u}_{H^1(\R^3 \setminus \Gamma)}^2 = \sum_{j = 0}^J \norm{u_j - I_{h,j} u_j}_{H^1(\Omega_j)}^2\,,\]
		and we conclude using the well-known approximation properties of the Lagrange interpolants $I_{h,j}$.
	\end{proof}

\section{Stability conditions for induced splittings}
\label{sec:proofHiptmairMao}
We break the proof of \Cref{mainthm} into several lemmas. 
\begin{lemma}[Stability of the discrete harmonic lifting]
	\label{lem:StabDisc}
	Let $\Pi_h$ be as in \textup{\textbf{(A)}}, and let $\Phi: \mathbb{X} \to \mathbb{V}$ be the {\bf harmonic lifting} (or minimal norm extension) characterized by 
	\[\forall f \in \mathbb{X}\,, \quad T(\Phi f) = f\,, \textup{ and } \norm{\Phi f}_{\mathbb{V}} = \norm{f}_{\mathbb{X}}\,.\]
	Write $E_h \isdef \Pi_h \circ \Phi$. Then
	\[\forall f_h \in X_h\,, \quad T(E_h f_h) = f_h\,,\]
	and one has the bound
	\[\norm{E_h f_h}_{\mathbb{V}} \leq \norm{\Pi_h}_{\mathcal{L}(\mathbb{V})} \norm{f_h}_{\mathbb{X}}\,.\]
\end{lemma}
\begin{proof}
	Let $f_h \in X_h$, and fix $u_h \in V_h$ such that $T(u_h) = f_h$. Observe that $u_h - \Phi f_h \in \mathbb{V}_0$. Indeed, 
	\[T(u_h - \Phi f_h) = T(u_h) - T(\Phi f_h) = f_h - f_h = 0\,.\]
	Since $\Pi_h$ preserves $\mathbb{V}_0$, one has $\Pi_h (u_h - \Phi f_h) \in \mathbb{V}_0$. Consequently, 
	\begin{align*}
		0 &= T(\Pi_h (u_h - \Phi f_h)) = T(\Pi_h u_h) - T(E_h f_h) = T(u_h) - T(E_h f_h)= f_h - T(E_h f_h)\,,
	\end{align*}
	where in the third equality, we used hat $\Pi_h u_h = u_h$ by \textbf{(A)}. This proves the first claim. The second claim is obvious since $E_h = \Pi_h \circ \Phi$ and, by definition of $\Phi$, $\norm{\Phi}_{\mathcal{L}(\mathbb{X},\mathbb{V})} = 1$. 
\end{proof}
\begin{lemma}
	\label{lemDiscreteTraceNorm}
	Assume that \textup{\textbf{(A)}} holds.  Then
	\[\forall u_h \in V_h\,, \quad \min \enstq{\norm{u_h - u_{h,0}}_{\mathbb{V}}}{u_{h,0} \in V_h \cap \mathbb{V}_0}\leq \norm{\Pi_h}_{\mathcal{L}(\mathbb{V})} \norm{T(u_h)}_{\mathbb{X}}\,.\]
\end{lemma}
\begin{proof}
	Let $u_h \in V_h$, and put 
	\[v_{h,0} \isdef u_h - E_h(T(u_h))\,.\]
	By the \Cref{lem:StabDisc} above, $T(v_{h,0}) = 0$, hence $v_{h,0} \in V_h \cap \mathbb{V}_0$. Therefore, 
	\begin{align*}
		\min \enstq{\norm{u_h - u_{h,0}}_{\mathbb{V}}}{u_{h,0} \in V_h \cap \mathbb{V}_0}
		&\leq \norm{u_h - v_{h,0}}_{\mathbb{V}}\\
		&= \norm{E_h(T(u_h))}_{\mathbb{V}} \leq \norm{\Pi_h}_{\mathcal{L}(\mathbb{V})} \norm{T(u_h)}_{\mathbb{X}}\,,
	\end{align*}
	where in the last inequality, we used the bound on the norm of $E_h$ from \Cref{lem:StabDisc}.
\end{proof}
\begin{corollary}
	\label{cor:Aprime}
	Consider the statement:\\
	
	\textup{\textbf{(A')}} There exists a linear operator $E_h: X_h \to V_h$ such that 
	\[\forall f_h\in X_h \,, \quad T(E_h f_h) = f_h\,.\]
	Then \textup{\textbf{(A)}} implies \textup{\textbf{(A')}}, with $\norm{E_h}_{\mathcal{L}(\mathbb{X},\mathbb{V})} \leq \norm{\Pi_h}_{\mathcal{L}(\mathbb{V})}$. 	
\end{corollary}
\begin{lemma}
	\label{lem:Bprime}
	Consider the statement\\
	
	\textup{\textbf{(B')}} There is a linear operator $E_{h,i}: X_{h,i} \to V_{h,i}$ satisfying
	\[T(E_{h,i} f_{h,i}) = f_{h,i}\,, \quad f_{h,i} \in X_{h,i}\,.\] 
	
	\noindent Assume that \textup{\textbf{(A)}} holds. Then, if \textup{\textbf{(B)}} holds with constants $\kappa_i$, \textup{\textbf{(B')}} holds, with the estimates
	\[\norm{E_{h,i}}_{\mathcal{L}(\mathbb{X},\mathbb{V})} \leq \kappa_i \norm{\Pi_h}_{\mathcal{L}(\mathbb{V})}\,.\]
\end{lemma}
\begin{proof}
	Let $E_{h,i}$ be the linear operator which maps $f_{h,i} = T(u_{h,i})$ to the minimizer of 
	\[\min \enstq{\norm{u_{h,i} - u_{i,h,0}}_{\mathbb{V}}}{u_{h,i,0} \in \mathbb{V}_0 \cap V_{h,i}}\,.\]
	Notice that this quantity is the one in the left-hand side of condition \textbf{(B)}. Also note that $E_{h,i}$ does not depend on the choice of representative $u_{h,i}$ of $f_{h,i}$ and is indeed a linear operator. We thus have, by combining \textbf{(B)} with \Cref{lemDiscreteTraceNorm}:
	\[\norm{E_{h,i}f_{h,i}}_{\mathbb{V}} \leq \kappa_i \norm{\Pi_h}_{\mathcal{L}(\mathbb{V})} \norm{T(u_{h,i})}_{\mathbb{X}} =  \kappa_i \norm{\Pi_h}_{\mathcal{L}(\mathbb{V})} \norm{f_{h,i}}_{\mathbb{X}}\,.\]
\end{proof}
\begin{proof}[Proof of Theorem \ref{mainthm}]
	The theorem is a consequence of \cite[Thm 2.1]{hiptmair2012stable}, since, by \Cref{cor:Aprime} and \Cref{lem:Bprime} the combination of assumptions \textbf{(A)} and \textbf{(B)} implies \textbf{(A')} and \textbf{(B')}. 
\end{proof}

\begin{proof}[{Proof of \Cref{lem:weakening}}]
	Going inside the proof of \cite[Thm 2.2]{hiptmair2012stable}, one replaces the operator $E_0: X_0 \to V_0$ by the global extension operator $E: X_0 \to V$. In the line of the proof below eq.(2.14) of that reference, we then write
	\begin{align*}
		\sum_{i = 0}^L \norm{\xi_i}^2_D \geq \frac{1}{C_2^2} \sum_{i = 1}^L \norm{E_i \xi_i}^2_{A} + \frac{1}{C_1^2} \norm{E \xi_0}^2_{A} &\geq \min\left(\frac{1}{C_1^2},\frac{1}{C_2^2}\right)\left\{ \norm{\sum_{i = 1}^L E_i \xi_i}^2_{A} + \norm{E\xi_0}^2_A\right\}
		\\
		& \geq \frac{1}{2\max(C_1,C_2)^2} \norm{\sum_{i = 1}^L E_i \xi_i + E \xi_0}^2_A\,.
	\end{align*}
	Te rest of the proof carries over without difficulty. 
\end{proof}

\section{Substructuring estimate in the volume}
\label{sec:proofMain}
%

\begin{proof}[Proof of \Cref{thm:splitVol1}]
	We adapt the approach of \cite[Chap. 5]{toselli2004domain}. Remarking that the spaces $V_{h,0}(\Omega_j)$ are pairwise orthogonal and orthogonal to $\mathbf{V}_h(\Omega \setminus \Gamma)$, it suffices to study the stability of the splitting
	\begin{equation}
		\label{eq:splitHarm}
		\mathbf{V}_{h}(\Omega \setminus \Gamma) = \sum_{\mathcal{F}^k \cap \Gamma = \emptyset }\mathbf{V}_{\mathcal{F}^k} + \sum_{\mathcal{F}^k \subset \Gamma} \sum_{\nu = 1}^2 \mathbf{V}_{\mathcal{F}^k,\nu} + \mathbf{V}_{\mathcal{W}} + V_H(\Omega \setminus \Gamma)\,.
	\end{equation}
	Note that one has indeed $V_{H}(\Omega \setminus \Gamma) \subset \mathbf{V}_h(\Omega \setminus\Gamma)$ since elements of $V_H(\Omega \setminus \Gamma)$ are linear, and thus harmonic, in each $\Omega_j$. Using Poincaré's inequality on $H^1_0(\Omega)$, it suffices to show the stability of the splitting with respect to the norm induced by the bilinear form 
	\[A(u,v) \isdef \sum_{j = 1}^J \int_{\Omega_j} \nabla u \cdot \nabla v\,dx\,,\] 
	i.e., replacing the $H^1(\R^3\setminus \Gamma)$ norms in eq.~\eqref{eq:stableSplitVol} by $A(\cdot,\cdot)^{1/2}$ norms. Hence, in what follows, we study the stability of the splitting in eq.~\eqref{eq:splitHarm} through the following modified norm:
	\begin{equation}
		\label{eq:defNormSplitVolA}
		\vertiii{u_h}_{A}^2 \isdef \inf \enstq{\sum_{i = 0}^N A(u_{h,i},u_ {h,i})}{\sum_{i = 0}^N u_{h,i} = u_h\,,\,\, u_{h,i} \in \mathbf{V}_{i}}\,,
	\end{equation}
	where $\mathbf{V}_0 = V_H(\Omega \setminus \Gamma)$, $\mathbf{V}_1 = \mathbf{V}_{\mathcal{W}}$ and $\mathbf{V}_i$, $i = 2,\ldots, N$ are the discrete harmonic face spaces $\mathbf{V}_{\mathcal{F}^k}$ and $\mathbf{V}_{\mathcal{F}^k,\nu}$, in some order. To this aim, we seek estimates for constants $\theta_h$ and $\Theta_h$ such that 
	\begin{equation}
		\label{eq:thetahThetah}
		\theta_h \norm{u_h}^2_{H^1(\R^3 \setminus \Gamma)} \leq \vertiii{u_h}_{A}^2 \leq \Theta_h \norm{u_h}^2_{H^1(\R^3 \setminus \Gamma)} \quad \forall u_h \in \mathbf{V}_h(\Omega \setminus \Gamma)\,.
	\end{equation}
	For $i = 0,\ldots,N$, define $\mathbf{P}_i: \mathbf{V}_h(\Omega \setminus \Gamma) \to \mathbf{V}_h(\Omega \setminus \Gamma)$ by $\mathbf{P}_i u \in \mathbf{V}_i$ and
	\begin{equation}
		\label{eq:defPi}
		A(\mathbf{P}_i u, v_i) = A(u,v_i)\,, \quad \forall v_i \in \mathbf{V}_i\,.
	\end{equation}
	Let $\mathbf{P}_{\rm ad} = \sum_{i = 0}^N \mathbf{P}_i$. The best possible constants $\theta_h$ and $\Theta_h$ in eq.~\eqref{eq:thetahThetah} are given by $\theta_h = \lambda_{\min}(\mathbf{P}_{\rm ad})$ and $\Theta_h = \lambda_{\max}(P_{\rm ad})$, where $\lambda_{\min}(\mathbf{P}_{\rm ad})$ and $\lambda_{\max}(\mathbf{P}_{\rm ad})$ are the smallest and largest eigenvalues of $\mathbf{P}_{\rm ad}$, (see e.g. \cite[Thm. 16]{oswald1994multilevel}). 
	To bound the eigenvalues of $\mathbf{P}_{\rm ad}$, we follow the general theory of additive Schwarz preconditioning as presented in \cite[Section 2.3]{toselli2004domain}. 
	
	The bound on $\lambda_{\max}(\mathbf{P}_{\rm ad})$ is obtained by a classical coloring argument. More precisely, we show the analogs, in our context, of Assumptions 2.3 and 2.4 in this \cite[Section 2.3]{toselli2004domain}. Suppose that $2 \leq i,j \leq N$ are such that $\mathbf{V}_{i}$ corresponds to the face $\mathcal{F}^k$ and $\mathbf{V}_j$ to the face $\mathcal{F}^\ell$. Define 
	\begin{equation}
		\epsilon_{i,j} \isdef \begin{cases}
			1 & \textup{if } \mathcal{F}^k \textup{ and } \mathcal{F}^\ell \textup{ have a common edge} \\
			0 & \textup{ otherwise}\,. 
		\end{cases}
	\end{equation} 
	One then has the strengthened Cauchy-Schwarz inequality 
	\[A(u_i,v_j) \leq \epsilon_{i,j} A(u_i,u_i)^{1/2} A(v_j,v_j)^{1/2}\,, \quad \forall u_i \in \mathbf{V}_i\,,\,\, v_j \in \mathbf{V}_j\,, \quad 2 \leq i,j \leq N\,.\]
	Indeed, the inequality is the usual Cauchy-Schwarz inequality when $\epsilon_{i,j} = 1$. Conversely, if $\epsilon_{i,j} = 0$, then $\mathcal{F}^k$ and $\mathcal{F}^\ell$ have no edge in common, which implies that no tetrahedral element $\Omega_j$ is incident to both $\mathcal{F}^k$ and $\mathcal{F}^\ell$. It follows that $u_i$ and $v_j$ have disjoint support, hence $A(u_i,v_j) = 0$. 
	
	Let $\mathcal{E} = \{\epsilon_{i,j}\}_{2 \leq i,j \leq N}$ and let $\rho(\mathcal{E})$ be the spectral radius of $\mathcal{E}$. Then, bounding the spectral radius by the $\ell^\infty$-norm of the rows, we have immediately 
	\[\rho(\mathcal{E}) \leq 2N_{\sharp}\]
	where $N_{\sharp}$ is defined as the maximal number of faces $\mathcal{F}^k$ incident to a common edge. 
	
	Noting that \cite[Assumption 2.4]{toselli2004domain} holds with constant $\omega = 1$ here, because the same bilinear form $A$ is used on both sides of the equality in the definition of $\mathbf{P}_i$ in eq.~\eqref{eq:defPi}, we conclude by \cite[Lemma 2.6]{toselli2004domain} (adapting the proof to handle two coarse subspaces instead of one), that 
	\[\lambda_{\max}(\mathbf{P}_{\rm ad}) \leq \omega(\rho(\mathcal{E})+1) \leq 2( N_\sharp+1)\,.\]
	This gives a uniform bound on $\lambda_{\max}(\mathbf{P}_{\rm ad})$ with respect to $H$ and $h$ due to the shape-regularity assumption for the coarse triangulation. 
	
	It remains to prove a lower bound for $\lambda_{\min}(\mathbf{P}_{\rm ad}) \geq c\big((1 + \log H/h)^{-2}\big)$ for some $c > 0$. For this, by \cite[Lemma 2.5]{toselli2004domain}, it suffices to show that, given any $u_h \in \mathbf{V}_h(\Omega \setminus \Gamma)$, there exists a decomposition 
	\begin{equation}
		\label{eq:constraintSum}
		u_h = \sum_{i = 0}^N u_{h,i}\,, \quad u_{h,i} \in \mathbf{V}_i,
	\end{equation}
	such that 	
	\begin{align*}
		\sum_{i = 0}^N A(u_{h,i},u_{h,i}) \leq C(1+ \log H/h)^2 A(u_h,u_h)\,. 
 	\end{align*}
	where here and in the following, the letter $C$ is used to denote a generic constant whose value is independent of the parameters $H$ and $h$. We start with the space $\mathbf{V}_0 = V_H(\Omega \setminus \Gamma)$ and define 
	\[u_{h,0} \isdef \widetilde{I}_H u_h \in V_H(\Omega \setminus \Gamma)\] 
	where $\tilde{I}_H$ may be chosen as any Clément-type quasi-interpolant on the coarse triangulation, with the properties 
	\begin{equation}
		\label{eq:ClementProperties}
		|\tilde{I}_H u|_{H^1(\Omega_j)} \leq C \abs{u}_{H^1(\Omega_j)}\,, \quad \norm{u - \tilde{I}_H u}_{L^2(\Omega_j)} \leq C H \abs{u}_{H^1(\Omega_j)}\,,
	\end{equation}
	for each $j = 1,\ldots,J$. Next, let $i \geq 2$ and suppose that $\mathbf{V}_i$ is associated to the face $\mathcal{F}^k$, shared by the domains $\Omega_\ell$ and $\Omega_m$. For the time being, suppose that $\mathcal{F}^k$ does not belong to $\Gamma$. Let $w_h = u_h - u_{h,0}$, and let $u_{h,i} = \mathcal{H}(\theta_{\mathcal{F}^k} w_h)$. Here, as in \cite{toselli2004domain}, the function $\mathcal{H}(\theta_{\mathcal{F}^k} w)$ is the piecewise discrete harmonic extension of the boundary values of $w_h$ on $\mathcal{F}^k_h$, i.e. the element of $\mathbf{V}_h(\Omega \setminus \Gamma)$ with nodal values equal to $0$ outside $\mathcal{F}^k_h$, and to $u_h(x)$ for $x \in \mathcal{F}^k_h$. Then we have by \cite[Lemma 4.24]{toselli2004domain} and the properties of $\widetilde{I}_H$ above
	\begin{align*}
		A(u_{h,i},u_{h,i}) &\leq C(1 + \log H/h)^2 (\abs{w_h}^2_{H^1(\Omega_k)} + \frac{\norm{w_h}^2_{L^2(\Omega_k)}}{H^2} + \abs{w_h}^2_{H^1(\Omega_\ell)} + \frac{\norm{w_h}^2_{L^2(\Omega_\ell)}}{H^2})\\
		&\leq C(1 + \log H/h)^2 \abs{u_h}^2_{H^1(\Omega_k \cup \Omega_\ell)}\,. 
	\end{align*} 
	The case where $\mathcal{F}^j \subset \Gamma$ is similar, the only difference being that the discrete harmonic extension is only non-zero in one of the two domains $\Omega_k$ and $\Omega_\ell$ incident to $\mathcal{F}^j$. 
	Summing those inequalities for $2 \leq i \leq N$ and using that each tetrahedron $\Omega_j$ has only $4$ faces,  
	\[\sum_{i = 2}^N A(u_{h,i},u_{h,i}) \leq C(1 + \log H/h)^2 A(u_h,u_h)\,.\]
	
	It remains to handle the contribution from the wire-basket, which, in order to respect the constraint of eq.~\eqref{eq:constraintSum}, must be defined by  
	\[u_{h,1} := u_h - u_{h,0} - \sum_{i = 2}^N u_{h,i}\,.\]
	Since the functions $u_{h,i}$, $i = 2,\ldots,N$, vanish at the wire-basket, one has in fact $u_{h,1} = \mathcal{H}(\theta_{\mathcal{W}} w_h)$, where, as before, $\mathcal{H}(\theta_{\mathcal{W}} w_h)$ is the discrete piecewise-harmonic function with boundary values matching those of $w_h$ on the wire-basket generalized vertices, and $0$ on the face generalized vertices. We write 
	\[A(u_{h,1},u_{h,1}) = \sum_{j = 1}^J \abs{\mathcal{H}_j(\theta_{\mathcal{W}_j} w_h)}^2_{H^1(\Omega_j)}\,,\]
	where $\mathcal{H}_j(\theta_{\mathcal{W}_j} w)$ is the discrete harmonic extension in $V_h(\Omega_j)$ of the boundary values of $w_h$ on the wire-basket $\mathcal{W}_j$ of $\Omega_j$. By \cite{toselli2004domain}, [Lemma 4.19] and [Lemma 4.16] (arguments in this order)
	\[ \abs{\mathcal{H}_j(\theta_{\mathcal{W}_j} w_h)}^2_{H^1(\Omega_j)} \leq C \norm{w_h}^2_{L^2(\mathcal{W}_i)} \leq C(1 + \log H/h)^2 \norm{w_h}^2_{H^1(\Omega_i)}\,.\]
	We obtain the lower bound on $\lambda_{\min}(\mathbf{P}_{\rm ad})$ using again the properties in eq.~\eqref{eq:ClementProperties}. This concludes the proof of the Theorem. 
\end{proof}

\end{document}